\documentclass[12pt]{article}

\usepackage{titletoc}

\usepackage{zref}
\usepackage[utf8]{inputenc}
\usepackage[T1]{fontenc} 
\usepackage{amsmath}
\usepackage{amsthm}
\usepackage{amssymb}
\usepackage{amsmath,amssymb}
\usepackage{mathtools}
\usepackage{amsthm}
\usepackage{mathrsfs}
\usepackage{enumitem}
\usepackage{tikz}
\usepackage{xfrac}
\usepackage{tikz-cd}\usetikzlibrary{decorations.pathmorphing}
\usepackage{graphicx}
\usepackage[all]{xy}
\usepackage[top=2cm, bottom=2cm, left=3cm, right=3cm]{geometry}
\usepackage{slashed}
\usepackage{adjustbox}
\usepackage[toc,page]{appendix}
\usepackage{stackengine,amssymb,graphicx}
\usepackage{xcolor}

\usepackage{fancyref}

\usepackage{soul}
\usepackage{scalerel}
\usepackage{stackengine,wasysym}

\usepackage{comment}
\usepackage{ dsfont }
\usepackage{url}

\author{\emph{Ludovic MONIER}}
\usepackage{doi}
\usepackage{hyperref}
\usepackage{xcolor}
\hypersetup{
    colorlinks,
    linkcolor={black!100!black},
    citecolor={black!100!black},
    urlcolor={black!100!black}
}

\title{Graded loopspaces in mixed characteristics and de Rham-Witt algebra}

\usepackage{titling}
\usepackage{psvectorian}

\begin{document}

\begin{titlingpage}
\maketitle
\end{titlingpage}

 \makeatletter
\newcommand*{\relrelbarsep}{.386ex}
\newcommand*{\relrelbar}{%
  \mathrel{%
    \mathpalette\@relrelbar\relrelbarsep
  }%
}
\newcommand*{\@relrelbar}[2]{%
  \raise#2\hbox to 0pt{$\m@th#1\relbar$\hss}%
  \lower#2\hbox{$\m@th#1\relbar$}%
}
\providecommand*{\rightrightarrowsfill@}{%
  \arrowfill@\relrelbar\relrelbar\rightrightarrows
}
\providecommand*{\leftleftarrowsfill@}{%
  \arrowfill@\leftleftarrows\relrelbar\relrelbar
}
\providecommand*{\xrightrightarrows}[2][]{%
  \ext@arrow 0359\rightrightarrowsfill@{#1}{#2}%
}
\providecommand*{\xleftleftarrows}[2][]{%
  \ext@arrow 3095\leftleftarrowsfill@{#1}{#2}%
}
\makeatother

\newcommand{\bigslant}[2]{{\raisebox{.2em}{$#1$}\left/\raisebox{-.2em}{$#2$}\right.}}

\newcommand{\ubar}[1]{\text{\b{$#1$}}}

\newcommand{\triplerightarrow}{%
\tikz[minimum height=0ex]
  \path[->]
   node (a)            {}
   node (b) at (1em,0) {}
  (a.north)  edge (b.north)
  (a.center) edge (b.center)
  (a.south)  edge (b.south);%
}

\newcommand{\doublerightarrow}{%
\tikz[minimum height=0ex]
  \path[->]
   node (a)            {}
   node (b) at (1em,0) {}
  (a.center)  edge (b.center)
  (a.south)  edge (b.south);%
}

\newcommand{\ornamentleft}{%
    \psvectorian[width=2em]{2}%
}
\newcommand{\ornamentright}{%
    \psvectorian[width=2em,mirror]{2}%
}
\newcommand{\ornamentbreak}{%
    \begin{center}
    \ornamentleft\quad\ornamentright
    \end{center}%
}

\theoremstyle{definition}
\newtheorem{Def}{Definition}[subsection]
\newtheorem{Rq}[Def]{Remark}
\newtheorem{Ex}[Def]{Example}
\newtheorem{Rep}[Def]{Representation}
\newtheorem{Di}[Def]{Diagram}
\newtheorem{Q}[Def]{Question}
\newtheorem{I}[Def]{Intuition}
\newtheorem{Exo}[Def]{Exo}
\newtheorem{Not}[Def]{Notation}
\newtheorem{Cons}[Def]{Construction}
\newtheorem{Int}[Def]{Introduction}
\newtheorem{Ack}[Def]{Acknowledgements}

\theoremstyle{plain}
\newtheorem{Prop}[Def]{Proposition}
\newtheorem{Thm}[Def]{Theorem}
\newtheorem{Le}[Def]{Lemma}
\newtheorem{Cor}[Def]{Corollary}
\newtheorem{Ques}[Def]{Question}

\renewcommand\qedsymbol{$\square$}

\newpage

{\setlength{\parskip}{5pt plus 25pt}
\tableofcontents
}
\newpage

\section{Introduction}

\subsection{De Rham cohomology}
De Rham cohomology is a cohomology theory offering a purely algebraic construction, based on differential forms, which recovers the Betti cohomology of a smooth algebraic variety of characteristic zero. Precisely, the algebraic de Rham cohomology groups $H^*_{DR}(X/\mathbb{C})$ are isomorphic to the singular cohomology groups with complex coefficients of the underlying complex manifold $X(\mathbb{C})$ when $X$ is a smooth complex algebraic variety. When $X$ is the affine scheme $Spec(A)$, the de Rham complex is based on the $A$-module $\Omega_{A/k}$ of Kähler differentials associated to $A$. This module represents the functor of derivations :
$$Der_k(A,M) \coloneqq Hom_{A-Mod}(\Omega_{A/k},M) \cong Hom_{CRing_{k/A}}(A, A \oplus M)$$
The universal property of the Kähler module defines the universal differential
$$d : A \to \Omega_{A/k}$$
which extends to a chain complex
$$\Omega_{A/k}^0 \xrightarrow{d} \Omega_{A/k}^1 \xrightarrow{d} \Omega_{A/k}^2 \xrightarrow{d} ...$$
where $\Omega_{A/k}^k \coloneqq \Lambda_A^k \Omega_{A/k}$. The algebraic de Rham cohomology $H^*_{DR}(X/k)$ is then defined as the hypercohomology of the de Rham complex.

The HKR theorem, originally proven in \cite{HKR62}, connects de Rham forms and Hochschild cohomology by giving functorial isomorphisms
$$\Omega_{A/k}^i \cong HH^i(A)$$
for a smooth commutative $k$-algebra $A$ and $k$ a field of characteristic zero. The de Rham cohomology is not as well-behaved for non-smooth algebras : it needs to be replaced by a derived variant using the cotangent complex $\mathbb{L}_{A/k}$ instead of Kähler differentials. The comparision between de Rham complex and Hochschild cohomology is recovered for arbitrary algebras $A$ over a field of zero characteristic, see \cite{TV09}. In fact, we get an equivalence of commutative differential graded algebras
$$DR(A) = Sym_A(\mathbb{L}_{A/k}[1]) \simeq HH(A) = A \otimes^\mathbb{L}_{A \otimes^\mathbb{L}_k A} A$$
In characteristic zero, the equivalence of $\infty$-categories
$$\epsilon-cdga \simeq S^1-SCR$$
between graded mixed algebras and $S^1$-equivariant simplicial algebras identifies the mixed structure on $Sym_A(\mathbb{L}_{X/k}[1])$, given by the de Rham differential, and the circle action on
$$A^{S^1} \simeq A \otimes^\mathbb{L}_{A \otimes^\mathbb{L}_k A} A$$

These observations rely heavily on $\infty$-categorical notions, such as derived tensor products of simplicial rings or cotensorization of a simplicial ring by a space. We will use derived geometry in this thesis to formulate and prove our results. The de Rham algebra
$$DR(A) = Sym_A(\mathbb{L}_{A/k}[1])$$
has been central in the study of notions and theorems imported from differential geometry to the world of algebraic geometry or derived geometry. It is used to define shifted symplectic forms in order to study deformation quantization in \cite{CPTVV} and it also appears in the definition of foliations in \cite{TV20}. Using the langage of derived geometry, the de Rham algebra can be described as a mapping stack from a derived stack called the graded circle $S^1_{gr}$ :
$$\textbf{Map}(S^1_{gr},Spec(A)) \simeq Spec(Sym_A(\mathbb{L}_{A/k}[1]))$$

The graded circle $S^1_{gr}$ is a formal version of the topological circle $S^1$ which has a canonical grading, see Definition \ref{GradedCircle} and \cite{MRT20} for details. By analogy with the circle $S^1$, a mapping stack from the graded circle $S^1_{gr}$ will be called a graded loopspace.

The main purpose of this thesis is to find a similar description of the de Rham-Witt complex, as a graded loopspace in the appropriate category.

\subsection{Crystalline cohomology}

Crystalline cohomology was born to fill a gap in the theory of $l$-adic cohomologies for algebraic varieties over a field of characteristic $p$. When $l \neq p$, $l$-adic cohomology is a well-behaved Weil cohomology theory, but it is badly behaved for $l=p$. Initially, an analog of $l$-adic cohomology for $l=p$ was motivated by the Riemann hypothesis for algebraic varieties.

When $X$ is a smooth and proper algebraic variety over $\mathbb{F}_p$, the zêta function of $X$ is given by the formal series
$$Z(X,t) \coloneqq exp \left(\sum_{n \ge 1} Card(X(\mathbb{F}_{p^n})) \frac{t^n}{n}\right) \in 1+t\mathbb{Z}[[t]]$$
Pulling back to the algebraic closure of $\mathbb{F}_p$, $\bar{X}$ admits an action of the Frobenius $Fr$ and we can identify fixed points of $Fr^n$ with $\mathbb{F}_{p^n}$-points of $X$. To show the function $Z(X,-)$ is rational using the Lefschetz fixed points formula, we need a sufficiently well-behaved cohomology theory for varieties over a finite field which takes value in a finite dimensional vector space over a field of characteristic zero.

Grothendieck's idea for a definition of such a cohomology theory was as follow. Let $X$ be a smooth proper scheme over a perfect field $k$ of positive characteristic $p$. A natural candidate for the ring of coefficients of this cohomology theory is the ring of Witt vectors $W(k)$. If $X$ lifts to a smooth proper scheme $\widetilde{X}$ over $W(k)$, taking the hypercohomology of the de Rham complex of $\widetilde{X}$ over $W(k)$ is independent of the lift. This definition, although allowing easy computations for well-behaved schemes, is not satisfactory and does not define a cohomology theory for a large class of schemes.

So as to give an intrinsic definition, the notion of divided powers of an ideal in a ring was used by Berthelot to define crystalline cohomology in \cite{Ber74}. He introduced a ringed site of divided power thickenings over the Zariski site. Cohomology of the structural sheaf defines crystalline cohomology. This powerful approach to build a cohomology theory allowed for easy functorial properties, and a similar type of construction was used more recently by Bhatt and Scholze to define a more general cohomology theory for $p$-adic schemes called prismatic cohomology, see \cite{BS22}.

\subsection{De Rham-Witt complex}

The definition of the de Rham-Witt complex was motivated by the study of the Hodge filtration on crystalline cohomology. The de Rham complex being endowed with a natural filtration, the associated de Rham cohomology acquires a canonical filtration called the Hodge filtration. However, when defining crystalline cohomology of a smooth scheme as de Rham cohomology of a smooth lift to Witt vectors, the canonical filtration is dependant on the chosen lift. The de Rham-Witt complex bridges this gap as it is endowed with a canonical "slope" filtration. Illusie introduced the de Rham-Witt complex $W_n\Omega_X^\bullet$ in \cite{Ill79}. This complex was defined as an initial object in a category of elements having the same structure as the de Rham complex and also the structure naturally arising on Witt vectors. Such elements are called pro-$V$-complex, they are given by families of commutative differential graded $W_n(k)$-algebras endowed with a Verschiebung map $V$ and a Frobenius morphism $F$ satisfying compatibility conditions. A natural element in this category is $\Omega_{W_n(X)}^\bullet$, the de Rham complex of the Witt vectors, therefore there is a canonical map
$$\Omega_{W_n(X)}^\bullet \to W_n \Omega^\bullet_X$$
As the category of pro-$V$-complexes is modeled from the properties of $\Omega_{W_n(X)}^\bullet$. The de Rham-Witt complex $W_n \Omega^\bullet_X$ can be thought of as applying a left adjoint functor to $\Omega_{W_n(X)}^\bullet$, that is taking quotients and completion procedures. We will denote $W_n \Omega^\bullet_X$ the limit of the tower $W \Omega^\bullet_X$, we will call it the de Rham-Witt complex of $X$.

The construction of the de Rham-Witt complex has being recently revisited by Bhatt, Lurie and Matthew in \cite{BLM22}. In this paper, the authors give a simplified construction of this complex and we will review the main definitions and some of the results which motivated this thesis.

\begin{Def}
A Dieudonné complex $(M,d,F)$ is a cochain complex of abelian groups $(M,d)$ and $F$ a morphism of graded abelian groups $F : M \to M$ such that
$$dF = pFd$$

The category of Dieudonné complexes is denoted $\textbf{DC}$.
\end{Def}

Following \cite{BLM22}, a Dieudonné algebra $(A,d,F)$ is given by $(A,d)$ a commutative differential graded algebra and $F : A \to A$ is a morphism of graded rings satisfying compatibilities analogous to the ones defining a Dieudonné complex.

Morphisms between Dieudonné algebras are ring morphisms compatible with differentials and Frobenius structures. This defines the category of Dieudonné algebras, which we denote $\textbf{DA}$.

Forgetting the multiplication defines the forgetful functor $\textbf{DA} \to \textbf{DC}$. The following construction defines what we may call the de Rham complex of a $p$-torsion-free ring endowed with a Frobenius lift.

Let $R$ be a $p$-torsion-free commutative ring and $\phi : R \to R$ a classical Frobenius lift. Then there is a unique ring morphism $F : \Omega_R^\bullet \to \Omega_R^\bullet$ such that
\begin{itemize}
\item $F(x) = \phi(x)$ for $x\in R = \Omega_R^0$.
\item $F(dx) = x^{p-1} dx + d(\frac{\phi(x)-x^p}{p})$ for $x \in R$.
\end{itemize}
In this case, $(\Omega_R^\bullet,d,F)$ is a Dieudonné algebra.

The main purpose of this thesis is to give a geometrical interpretation of these complexes in terms of graded loopspaces, which will give a generalization of this proposition for $R$ a non-necessarily $p$-torsion-free simplicial algebra. The de Rham complex has the following universal property, which we will generalize.

Let $R$ be a $p$-torsion-free commutative ring and $\phi : R \to R$ be a classical Frobenius lift. For $A$ a $p$-torsion-free Dieudonné algebra, the construction given before defines a functorial bijection
$$Hom_{\textbf{DA}}(\Omega_R^\bullet,A) \xrightarrow{\sim} Hom_{Fr}(R,A^0)$$
where $Hom_{Fr}(R,A^0)$ is the set of morphisms $R \to A^0$ compatible with Frobenius structures.

In this thesis, we have generalized the previous constructions to simplicial rings with a Frobenius lift. The $\infty$--category of simplicial rings with a Frobenius lift being a "derived enhancement" of the $1$-category of commutative rings endowed with a classical Frobenius lift, this thesis aims at constructing a derived version of the de Rham-Witt complex.

Our constructions extend the ones previously defined in the following ways :
\begin{itemize}
\item Working in simplicial rings allows for derived tensor products and homotopy limits and colimits.
\item The constructions will be $\infty$-categorical and the functors will be $\infty$-functors.
\item The assumption of the ring to be $p$-torsion-free is dropped.
\item We outline definitions and theorems for global objects, thus defining the de Rham-Witt complex for schemes or derived stacks.
\end{itemize}

Dieudonné complexes and algebras can have the property of being saturated, defining full subcategories of $\textbf{DC}$ and $\textbf{DA}$. Furthermore there is a functorial construction called saturation making a Dieudonné complex or algebra saturated. We will define these constructions in the context of derived Dieudonné complexes and algebras.

\subsection{Witt vectors, $\delta$-rings and Frobenius lifts}

We give a short presentation of some results using $\delta$-rings or Frobenius lifts. A central notion for this thesis is that of homotopy Frobenius lift. It seems to be the appropriate structure to study $p$-adic schemes.

The most recent results on $p$-adic cohomology theories are based on the notion of $\delta$-rings, see \cite{BS22}. We also note that the section 3.7 of \cite{BLM22} gives generalizations of their definitions of de Rham-Witt complex for rings with $p$-torsion using $\delta$-rings.

\begin{Def}(\cite[Definition 3.7.4]{BLM22})
A $\delta$-cdga is a commutative differential graded algebra $(A,d)$ endowed with a morphism of graded rings $F : A \to A$ such that
\begin{itemize}
\item The groups $A^n$ vanishes for $n < 0$, that is $A$ is coconnective since we are using cohomological conventions where the differential increases the degree.
\item The pair $(A_0,\delta)$ is a $\delta$-ring.
\item For $x \in A^0$, we have $F(x)=x^p+p \delta(x)$.
\item For $x \in A^n$, we have $dF(x) = pF(dx)$.
\item For $x \in A^0$, we have $F(dx) = x^{p-1} dx + d\delta(x)$
\end{itemize}
\end{Def}

\begin{Prop}(\cite[Definition 3.7.6]{BLM22}) \label{DRAdjDelta}
Let $(R,\delta)$ be a $\delta$-ring, the de Rham complex $\Omega_R^\bullet$ is canonically a $\delta$-cdga, this defines a functor
$$\Omega_{-} : \delta-CRing \to \delta-cdga$$
which is left adjoint to the forgetful functor
$$\delta-cdga \to \delta-CRing$$
\end{Prop}

The theory of prismatic cohomology is also based on the notion of $\delta$-rings through the definition of prisms.

\begin{Def}(\cite[Definition 3.2]{BS22})
A prism is a pair $(A,I)$ where $A$ is a $\delta$-ring and $I \subset A$ is an ideal defining a Cartier divisor on $Spec(A)$ such that $A$ is derived $(p,I)$-complete and $p \in I + \phi(I)A$ where $\phi : A \to A$ is the Frobenius morphism induced by the $\delta$-structure
\end{Def}

We recall a classical theorem about the de Rham-Witt complex.

\begin{Thm} (\cite[Theorem 1.1.2]{BLM22})
Let $k$ be a perfect ring of characteristic $p$. Let $\widetilde{X}$ be a smooth formal scheme over $Spf(W(k))$ with special fiber $X$. The data of a Frobenius lift $F : \widetilde{X} \to \widetilde{X}$ of the canonical Frobenius $Fr : X \to X$, if it exists, induces a natural quasi-isomorphism
$$f_{\widetilde{X},F} : \Omega^\bullet_{\widetilde{X}/W(k)} \xrightarrow{\sim} W\Omega_X^\bullet$$
\end{Thm}

This theorem defines a natural transformation between
$$\widetilde{X} \in fSch^{sm}_{W(k)} \mapsto \Omega^\bullet_{\widetilde{X}/W(k)} \in \mathcal{D}(X)$$
and
$$\widetilde{X} \in fSch^{sm}_{W(k)} \mapsto W\Omega_X^\bullet \in \mathcal{D}(X)$$
where $fSch^{sm}_{W(k)}$ denotes the $1$-category of smooth formal schemes over  $Spf(W(k))$ and $\mathcal{D}(X)$ is the derived $1$-category of quasi-coherent sheaves on $X$.
Therefore working on categories of schemes or commutative rings with Frobenius lifts seems fundamental.

\subsection{Derived geometry and cohomology theories}

In the past few years, the use of homotopy algebraic arguments to study cohomology theories of schemes has not stopped growing. We give a brief summary of some work related to derived algebra, or derived geometry, and the cohomology of schemes.

In the papers \cite{TV09}, \cite{MRT20} and \cite{Rak20}, a universal property is given connecting graded mixed structures with the de Rham functor.

\begin{Def}
We define the category of mixed graded simplicial algebras $\epsilon-SCR^{gr}$ as the full subcategory of
$$S^1_{gr}-dSt^{op}$$
on $S^1_{gr}$-equivariant derived stacks which are affine schemes.
\end{Def}

\begin{Thm} \label{DRAdjClassique}
The de Rham functor defines a functor to graded mixed algebras
$$SCR \to \epsilon-SCR^{gr}$$
which is left adjoint to taking the $0$-weighted part of a mixed graded simplicial algebra.
\end{Thm}

An explicit connection is also given between Hochschild cohomology and De Rham cohomology in \cite{MRT20}. Since each of these cohomology theories have geometrical avatars in derived geometry, we can obtain a global comparision between Hochschild cohomology and de Rham cohomology of a derived scheme.

Let $X$ be a derived scheme over $\mathbb{Z}_{(p)}$, the loopspace of $X$ is given by the mapping stack
$$\mathcal{L}(X) \coloneqq \textbf{Map}(S^1,X)$$
It admits a natural filtration which has as its associated graded stack the shifted tangent stack
$$T[-1]X \coloneqq Spec_X Sym_{\mathcal{O}_X}(\mathbb{L}_X [1])$$

Taking global functors defines a filtration of Hochschild cohomology of $X$ which admits as its associated graded the de Rham algebra of $X$.

We will use in this thesis the geometrical interpretation of graded derived stacks as $\mathbb{G}_m$-equivariant derived stacks and filtered derived stacks as $\mathbb{A}^1/\mathbb{G}_m$-equivariant derived stacks, see \cite{Mou19}.

In the world of $p$-adic schemes, derived crystalline cohomology has proven to be better-behaved than naive crystalline cohomology. Several papers have compared de Rham and derived crystalline cohomology, such as \cite{BDJ11} and \cite{Mon21b}. We recall a theorem of the latter.

\begin{Thm} (\cite[Theorem 5.0.1]{Mon21b})
Let
$$dR : \mathbb{E}_\infty Alg^{sm}_{\mathbb{F}_p} \to CAlg(Mod_{\mathbb{F}_p})$$
be the algebraic de Rham cohomology functor defined on the category of smooth $\mathbb{E}_\infty-\mathbb{F}_p$-algebras. Given $(A,m)$ an Artinian local ring with residue field $\mathbb{F}_p$, the functor $dR$ admits a unique deformation
$$dR' : \mathbb{E}_\infty Alg^{sm}_{\mathbb{F}_p} \to CAlg(Mod_{A})$$
This deformation $dR'$ is unique up to unique isomorphism. Furthermore, this deformation is given by crystalline cohomology.
\end{Thm}

We also note the importance of derived $p$-completions instead of classical $p$-completions used in \cite{BS22} and \cite{BDJ11}.

\paragraph*{Related work}
We would like to mention connections between this thesis and work in progress of Antieau. The following principle seems to central : graded objects defined by geometric constructions should be naturally associated graded to canonical HKR-type filtrations defined geometrically. As the graded circle is the associated graded of the affinization of the topological circle, any geometrical object defined with the graded circle $S^1_{gr}$ should satisfy this heuristic. This philosophy will probably be correct for our construction of de Rham-Witt complex as a graded loopspace and also to the definition of foliations which relies on the graded circle and graded mixed complexes. Indeed, in Section \ref{futuPers}, we will outline the construction of an analogue of the Hochschild cohomology complex, based on a crystalline topological circle, which should admit a filtration which has as its associated graded the de Rham-Witt complex. This Hochschild cohomology complex seems to be a new and interesting object to study.

\subsection{Content of the thesis}
We present the results of this thesis and some intuitions behind the main definitions and theorems. We will work over the base commutative ring $k=\mathbb{Z}_{(p)}$, therefore simplicial algebras will be simplicial $\mathbb{Z}_{(p)}$-algebras and derived stacks will be derived stacks over $Spec(\mathbb{Z}_{(p)})$.

\subsubsection{Frobenius lifts and graded loopspaces}

\begin{Def}
We define the category of derived stacks with a derived, or homotopy, Frobenius lift $dSt^{Fr}$ as the pullback of categories
$$dSt^{endo} \times_{dSt_{\mathbb{F}_p}^{endo}} dSt_{\mathbb{F}_p}$$
where the functor $dSt_{\mathbb{F}_p} \to dSt_{\mathbb{F}_p}^{endo}$ is the canonical functor adjoining the Frobenius endomorphism to a derived stack on $\mathbb{F}_p$. Similarly, we define the category of graded derived stacks endowed with a derived Frobenius lift $dSt^{gr, Fr}$ as the pullback of categories
$$dSt^{gr,endo} \times_{dSt_{\mathbb{F}_p}^{endo}} dSt_{\mathbb{F}_p}$$
where the forgetful functor $dSt^{gr,endo} \to dSt_{\mathbb{F}_p}^{endo}$ is given by taking $0$-weights, that is $\mathbb{G}_m$-fixed points, and taking the fiber on $\mathbb{F}_p$.
\end{Def}

\begin{Rq}
Similarly, we define the category of simplicial algebras endowed with a derived Frobenius lift $SCR^{Fr}$ and the category of graded simplicial algebras endowed with a derived Frobenius lift $SCR^{gr,Fr}$.
\end{Rq}

As the classical definition of Dieudonné module encodes the structure of the de Rham complex of a smooth commutative ring $A$ endowed with a classical Frobenius lift, our notion of mixed graded Dieudonné modules is modeled on
$$DR(A) = Sym_A \mathbb{L}_A[1]$$
which is naturally a mixed graded complex. The usual compatiblity
$$dF = pFd$$
becomes
$$\epsilon F = pF \epsilon$$

A central object of our construction is the "crystalline graded circle" denoted $S^1_{gr}$ which is a group in $dSt^{gr,Fr}$. It is the affine stack
$$S^1_{gr} \simeq Spec^\Delta(\mathbb{Z}_{(p)} \oplus \mathbb{Z}_{(p)}[-1])$$
where $\mathbb{Z}_{(p)} \oplus \mathbb{Z}_{(p)}[-1]=\mathbb{Z}_{(p)}[\eta]$ is the denormalized cosimplicial algebra of a square zero extension commutative differential graded algebra. Sending $\eta$ to $p \eta$ defines the Frobenius structure on this stack. For details on the functor $Spec^\Delta(-)$ and affine stacks, see \cite{To06}. Now interpreting mixed graded modules as quasi-coherent sheaves in
$$QCoh(BS^1_{gr})$$
the forgetful functor along $[p]$ is given by
$$(M,d,\epsilon) \in QCoh(BS^1_{gr}) \mapsto (M,d,p\epsilon) \in QCoh(BS^1_{gr})$$
Therefore we will see a graded mixed Dieudonné complex as a colax fixed points of $[p]^*$, that is a pair $(M,f)$ where $M$ is a mixed graded complex and $f$ is morphism of mixed graded complexes
$$f : [p]^*M \to M$$

\begin{Rq}
The category of quasi-coherent complexes $QCoh(BS^1_{gr})$ is interpreted with $S^1_{gr}$ seen as a graded stack, therefore this category is implicitly defined as
$$QCoh(B(S^1_{gr} \rtimes \mathbb{G}_m))$$
See Remark \ref{S1grH} for details.
\end{Rq}

\begin{Def}
We define the category of mixed graded Dieudonné complexes, also called derived Dieudonné complexes, by
$$\epsilon-D-Mod^{gr} \coloneqq CFP_{[p]^*}(\epsilon-Mod^{gr})$$
where $CFP_{[p]^*}(\epsilon-Mod^{gr})$ denotes the category of colax fixed points in $\epsilon-Mod^{gr}$.
\end{Def}

The category of mixed graded Dieudonné complexes admits the following alternative description.

\begin{Prop} (Proposition \ref{TwoDefDieudComplex})
We have a natural identifications
$$\epsilon-D-Mod^{gr} \simeq (k[\epsilon],[p])-Mod^{gr,endo}$$
where $[p]$ is the canonical endomorphism on $k[\epsilon]$, given by
$$\epsilon \in k[\epsilon] \mapsto p \epsilon \in k[\epsilon]$$
The object $(k[\epsilon],[p])$ is seen here as a commutative algebra object in $Mod^{gr,endo}$.
\end{Prop}

\begin{Def}
Motivated by the previous proposition, we define mixed graded Dieudonné stacks, also called derived Dieudonné stacks, by
$$\epsilon-D-dSt^{gr} \coloneqq S^1_{gr}-dSt^{gr,Frob}$$

This notion gives a definition of Dieudonné structures on global objects, such as derived schemes or derived stacks.

We also define mixed graded Dieudonné simplicial algebras, also called Dieudonné simplicial algebras, by
$$\epsilon-D-SCR^{gr} \coloneqq (S^1_{gr}-dAff^{gr,Frob})^{op}$$
In this definition, by abuse of notation, $S^1_{gr}-dAff^{gr,Frob}$ is the category of elements of $S^1_{gr}-dSt^{gr,Frob}$ which have as an underlying derived stack an affine derived scheme. An element of $S^1_{gr}-dAff^{gr,Frob}$ can be thought of as a morphism $X \to BS^1_{gr}$, in the topos $dSt^{gr,Frob}$, which is relatively affine.

A description of the category of derived Dieudonné stacks in terms of strict models, for example using a model category, seems difficult. Therefore, the use of derived stacks and the language of $\infty$-categories is essential in order to define these objects. Similarly, in \cite{Rak20}, the use of $\infty$-monads and homotopical algebra is essential to the description of the universal property of the de Rham complex.

The definition of derived Dieudonné stack is fundamental for the study of the multiplicative structure on the de Rham-Witt complex, however it adds a considerable amount of difficulty by requiring the use of the $\infty$-categories of simplicial rings and derived stacks.

The graded stack $S^1_{gr}$ admits an action from $\mathbb{G}_m$, which is identified with the grading structure on $S^1_{gr}$, it also has a natural action of $\mathbb{N}$ induced by its Frobenius lift structure, therefore we can form the semi-direct product $(S^1_{gr} \rtimes (\mathbb{N} \times \mathbb{G}_m))$. We notice that the category $S^1_{gr}-dSt^{gr,endo}$ identifies with $(S^1_{gr} \rtimes (\mathbb{N} \times \mathbb{G}_m))$-equivariant derived stacks, that is derived stacks over $B(S^1_{gr} \rtimes (\mathbb{N} \times \mathbb{G}_m))$. Furthermore, the category $QCoh(B(S^1_{gr} \rtimes (\mathbb{N} \times \mathbb{G}_m))$ identifies with $\epsilon-D-Mod$. These identifications can be seen as geometrical interpretations of Dieudonné algebras and complexes.
\end{Def}

We will then construct the functor of functions on a Dieudonné derived stack
$$C : \epsilon-D-dSt^{gr,op} \to \epsilon-D-Mod^{gr}$$
refining the usual functor of functions
$$C : dSt \to Mod$$

\begin{Def}
Let $X$ a derived scheme endowed with a Frobenius lift. We define the Frobenius graded loopspaces on $X$ as the mapping stack internal to the category $dSt^{Fr}$ :
$$\mathcal{L}^{gr,Fr}(X) \coloneqq \textbf{Map}_{dSt^{Fr}}(S^1_{gr},X)$$
where $S^1_{gr}$ is endowed with its canonical Frobenius action. It is a graded derived stack endowed with a Frobenius lift.

The canonical point $* \to S^1_{gr}$ defined by the augmentation $\mathbb{Z}_{(p)}[\eta] \to \mathbb{Z}_{(p)}$, is a morphism of graded derived stacks with Frobenius structures and induces a morphism of graded derived stacks with Frobenius structures
$$\mathcal{L}^{gr,Fr}(X) \to X$$
\end{Def}

The following proposition and corollary asserts that the crystalline graded circle is the correct analogue of the graded circle in order to recover the de Rham-Witt complex.

\begin{Prop} (Proposition \ref{ComputeLoopSpace})
Let $(X,F)$ be a affine derived scheme endowed with a Frobenius structure. We write $X = Spec(C)$. The graded Frobenius loopspace's underlying graded stack identifies with the shifted linear stack:
$$U \mathcal{L}^{gr,Fr}(X) \simeq Spec_X Sym_{\mathcal{O}_X}\left(\mathbb{L}_{(X,F)}^{tw}[1]\right)$$
where $Sym$ denotes the free $(C,F)$-module construction of Proposition \ref{FreeCFMod}.

The $(C,F)$-module $\mathbb{L} \coloneqq \mathbb{L}_{(X,F)}^{tw}$ fits in a triangle of $(C,F)$-modules
$$\bigoplus_{\mathbb{N}} \mathbb{L}_{(C,F)} \otimes \mathbb{F}_p \to \mathbb{L} \to \mathbb{L}_{(C,F)}$$
\end{Prop}

\begin{Cor} (Corollary \ref{CorGrLoopSpaceUnderlying})
With the notation of the previous proposition, the Beilinson truncated derived stack associated to $\mathcal{L}^{gr,Fr}(X)$ is given, after moding out by the $p$-torsion, by
$$Spec(Sym(\Omega_C[1]))$$
where the induced endomorphism is induced by
$$\frac{dF}{p} : \Omega_C[1] \to \Omega_C[1]$$
\end{Cor}

To construct the Dieudonné de Rham functor, we will need a classification theorem of Dieudonné structures on free simplicial algebras. This result is inspired by Proposition 2.3.1 in \cite{To20}, which was stated with a partial proof.

\begin{Thm} (Theorem \ref{classifClassique}) Let $A$ be smooth commutative $\mathbb{Z}_{(p)}$-algebra, $M$ a projective $A$-module of finite type. We define $X$ as the derived affine scheme $Spec(Sym_A(M[1]))=\mathbb{V}(M[1])$ endowed with its natural grading. The classifying space of mixed graded structures on $X$ compatible with its grading is discrete and in bijection with the set of commutative differential graded algebra structures on the graded commutative $\mathbb{Z}_{(p)}$-algebra $\bigoplus_i \wedge^i_A M [-i]$.
\end{Thm}

This key theorem will be proven using an analogue of the usual cellular decomposition of the topological space $BS^1 \simeq \mathbb{C}P^\infty$. Since $BS^1_{gr}$ is given by 
$$BS^1_{gr} \simeq Spec^\Delta(D(H^*(BS^1,\mathbb{Z}_{(p)})))$$
where $D$ is the denormalization functor,
we will need a decomposition on the denormalized cosimplicial algebra
$$D(H^*(BS^1,\mathbb{Z}_{(p)})) \simeq D(\mathbb{Z}_{(p)}[u]/u^2) $$
where $u$ is in degree $2$. This decomposition is not the one obtained from the topological decomposition therefore the construction has to be made from scratch. This is due to the fact that a pushout of topological spaces has a cosimplicial chain complex which is given by a pullback of the cosimplicial chain complexes of the spaces, however, this identification is not compatible with the natural grading. We will introduce weight-shifted formal spheres $S^{2n+1}(n+1)$ which will exhibit  $D(H^*(BS^1,\mathbb{Z}_{(p)}))$ as the appropriate limit, as a graded cosimplicial algebra.

Another fundamental ingredient of the proof will be a Postnikov decomposition to compute the various mapping stacks, using obtruction theory. We will first show that we can reduce to the $4$-skeleton $(BS^1_{gr})_{\le 4}$, this stack will be given by the homotopy cofiber of
$$S^3_{f}(2) \to S^2_f(1)$$
where $S^3_{f}(2)$ and $S^2_f(1)$ are formal version of the topological spheres $S^3$ and $S^2$ with a shift of the gradings. This map is an analogue of the Hopf fibration and we will study it using Koszul duality.
 
We will extend the previous theorem to the case of Dieudonné mixed graded structures. This theorem is the main result of this thesis.

\begin{Thm} (Theorem \ref{classiDieud}) Let $A$ be smooth commutative $\mathbb{Z}_{(p)}$-algebra, $M$ a projective $A$-module of finite type. We fix a derived Frobenius lift structure on the graded  simplicial algebra $Sym_A(M[1])$, with $M$ in weight $1$. From Proposition \ref{FLSymStr}, it is equivalent to a classical Frobenius lift $F$ on $A$ and a linear map of $A$-modules $\phi : M \to M$. We define $X$ as the derived affine scheme $Spec(Sym_A(M[1]))=\mathbb{V}(M[1])$ endowed with its natural grading, we regard it as an element of $dSt^{gr,Frob}$. The classifying space of Dieudonné mixed graded structures on $X$ compatible with its grading and Frobenius structure is discrete and in bijection with the set of Dieudonné algebra structures on the graded commutative $\mathbb{Z}_{(p)}$-algebra $\bigoplus_i \wedge^i_A M [-i]$ endowed with its natural canonical Frobenius lift structure.
\end{Thm}

\begin{Cor} (Corollary \ref{CorClassDieudStr})
With the notations of Theorem \ref{classiDieud}, the graded derived affine scheme $Spec(Sym_A(\Omega_A[1]))$ admits a unique Dieudonné structure induced by the standard Dieudonné structure on $\bigoplus_{i \ge 0} \Lambda_A^i \Omega_A$.
\end{Cor}

\begin{Def}
Let $(A,F)$ be a simplicial algebra with a Frobenius lift, we define the Dieudonné de Rham complex of $(A,F)$ as
$$\mathcal{O}(\mathcal{L}^{gr,Fr}(Spec(A),Spec(F)))$$
where $\mathcal{O}$ is the functor of "simplicial functions" on derived Dieudonné stacks
$$\epsilon-D-dSt^{gr,op} \to \epsilon-D-SCR^{gr}$$
defined in the previous proposition. We denote the element we obtain $DDR(A,F)$.
\end{Def}

\begin{Thm} (Theorem \ref{ComparaisonDRW})
When $A$ is a smooth discrete algebra, $DDR(A,F)$ is naturally the derived Dieudonné algebra
$$Spec(Sym_A(\Omega_A[1]))$$
endowed with its canonical Frobenius lift and mixed graded structure.
\end{Thm}

We will prove a theorem analoguous to Theorem \ref{DRAdjClassique} and Proposition \ref{DRAdjDelta}. It asserts that the Dieudonné de Rham functor is left adjoint to forgetting a graded mixed structure.

\begin{Thm} (Theorem \ref{AdjointDeRhamW}) The Dieudonné de Rham functor
$$SCR^{Fr} \to \epsilon-D-SCR^{gr}$$
is left adjoint to the forgetful functor
$$(0) : \epsilon-D-SCR^{gr} \to SCR^{Fr}$$
\end{Thm}

We also compare more precisely our construction of Dieudonné de Rham functor with the one defined in \cite{BLM22}. We define a t-structure inspired by Beilinson t-structure, see \cite[\S 3.3]{Rak20}, and prove truncation with respect to this t-structure of our Dieudonné de Rham functor recovers the one in \cite{BLM22} for non-necessarily smooth algebras.

\begin{Def}
We recall the definition of the Beilinson t-structure on graded mixed complex. Let $M$ be a graded mixed complex, $M$ is said to be t-connective for the t-structure when $H^i(M(n))=0$ for $i>-n$, $M$ is said to be t-coconnective for the t-structure when $H^i(M(n))=0$ for $i<-n$.
\end{Def}

A derived Dieudonné complex or algebra is said to be t-truncated when its underlying graded mixed complex is. We will abuse terminology and call these constructions "t-structures".

\begin{Prop} (Proposition \ref{HeartDMod})
The heart of $\epsilon-D-Mod$ identifies with the abelian category of Dieudonné complexes of \cite{BLM22}.
\end{Prop}

\begin{Thm}
For $A$ a commutative algebra, non-necessarily smooth, the t-truncation of $DDR(A)$, with respect to the Beilinson t-structure, is equivalent to the classical Dieudonné complex. That is
$$t_{\ge 0}(DDR(A)) \simeq i(\Omega^\bullet_A)$$
where $\Omega^\bullet_A$ is endowed with its canonical classical Dieudonné structure, see Proposition \ref{FormDieud} and $i$ is the functor
$$ i : \textbf{DA} \to \epsilon-D-SCR^{gr}$$
identifying the category of classical Dieudonné algebra with the heart of $\epsilon-D-SCR^{gr}$.
\end{Thm}

We will then define and study the notion of of saturation of a mixed graded Dieudonné algebra.

\begin{Def}
We define the décalage $\eta_p$ functor as an endofunctor of $1$-categories $\bf{\epsilon-Mod^{gr}}$ sending $M$ to the sub-graded mixed complex of $M \times M$ on elements $(x,y)$ such that $\epsilon x = py$ and $\epsilon y=0$.
\end{Def}

The following proposition defines and characterizes the saturation $\infty$-functor.

\begin{Prop} (Proposition \ref{adjDec})
We have an adjunction of $1$-categories
$$ [p]^* : \bf{\epsilon-Mod^{gr}} \rightleftarrows \bf{\epsilon-Mod^{gr}} : \eta_p$$
\end{Prop}

In fact, this adjunction is a Quillen adjunction for the injective model structures, therefore it induces an adjunction of $\infty$-categories.

A derived Dieudonné algebra is said to be saturated if its underlying Dieudonné complex is. We will verify that the definition of saturated Dieudonné complexes coincides with the one in \cite{BLM22}.

\subsubsection{Graded functions on graded stacks}

This thesis also contains the results presented in \cite{Mon21} on graded functions on linear stacks. We fix $X=Spec(A)$ an affine derived scheme.

\begin{Def}
Let $E$ be a quasi-coherent complex on $X$, the linear stack associated with $E$ is the stack over $X$ given by
$$u: Spec(B) \to X \mapsto Map_{N(B)-mod}(u^*E,N(B)) \in SSet$$
with $N$ the normalization functor. It is denoted $\mathbb{V}(E)$.
\end{Def}

\begin{Prop} (Proposition \ref{AdjLinStack})
The functor
$$\mathbb{V} : QCoh(X)^{op} \to \mathbb{G}_m-dSt_{/X}$$
has a left adjoint given by $\mathcal{O}_{gr}(-)(1)$.
\end{Prop}

\begin{Thm} (Theorem \ref{FunctionLinStack}) For $E \in QCoh^{-}(X)$, ie a bounded above complex over $X$, the natural map constructed above
$$E \to \mathcal{O}_{gr}(\mathbb{V}(E))(1)$$
is an equivalence.
\end{Thm}

\begin{Cor} (Corollary \ref{CoroLinStackFunction}) The functor $\mathbb{V} : QCoh(X)^{op} \to \mathbb{G}_m-dSt_{/X}$ is fully faithful when restricted to $QCoh^-(X)$.
\end{Cor}

This corollary is a crucial result for the study of foliations as it allows the study of foliations of non-zero characteristic with cotangent complexes which are not necessarily connective. We will combine it with the notion of Dieudonné foliations so as to construct examples of Dieudonné foliations.

\begin{Prop} (Proposition \ref{FunctionLinStackchar0}) If $k$ is a $\mathbb{Q}$-algebra and $E$ is perfect complex on $X$, $\mathcal{O}_{gr}(\mathbb{V}(E))(p)$ is canonically identified with $Sym_{\mathbb{E}_\infty}^p(E)$, where the monad $Sym_{\mathbb{E}_\infty}$ is the one associated to the formation of free $\mathbb{E}_\infty$-$k$-algebras, or equivalently commutative differential graded algebras.
\end{Prop}

\subsubsection{Notations}

\begin{Not}
In this thesis, all higher categorical notations are borrowed from \cite{Lur09,Lur17}. We fix $k$ to be a discrete commutative ring.

\begin{itemize}
\item "$\infty$-category" will always mean $(\infty,1)$-categories, modeled for example by quasicategories developed in \cite{Lur09}. Everything is assumed to be $\infty$-categorical, e.g "category" will mean $\infty$-category. We will refer to the $1$-categories and $1$-categorical notions as "classical" or "discrete", for example the $1$-category of discrete commutative rings will be the usual category of commutative rings, denoted $\bf{CRing}$.

\item We will denote $1$-categories and model categories with bold text, e.g $\bf{dg-mod_k}$ will be the $1$-category of chain complexes on $k$. The model structure on a $1$-category will be denoted in subscripts or omitted, e.g $\bf{dg-mod_{k,proj}}$ or $\bf{dg-mod_k}$ will be the model category of chain complexes on $k$ with the projective model structure. All $\infty$-categories will be denoted without bold text, e.g $Mod_k$ or $Mod$ the $\infty$-category of chain complexes on $k$.

\item We will use cohomological conventions : all chain complexes with have increasing differentials, hence negative chain complexes will be connective.

\item $Mod_k$ is the $\infty$-category of chain complexes of $k$-modules, $Mod_k^{\le 0}$ or $Mod^{cn}$ is the full sub-$\infty$-category of connective complexes and $Mod_k^{\ge 0}$ or $Mod^{ccn}$ is the full sub-$\infty$-category of coconnective complexes. Similarly, for $A$ a simplicial commutative $k$-algebra, $Mod_A$ will denote the usual $\infty$-category obtained by inverting quasi-isomorphisms in the $1$-category of differential graded modules $\bf{N(A)-dg-Mod}$ on the normalization of $A$.

\item $CAlg(\mathcal{C})$ is the $\infty$-category of commutative algebra objects in a symmetric monoidal $\infty$-category, see \cite{Lur17}.

\item $\mathbb{E}_\infty Alg_k = CAlg(Mod_k)$ is the $\infty$-category of $\mathbb{E}_\infty$-algebras on $k$, see \cite{Lur17}.

\item $\bf{CRing_k}$ is the $1$-category of discrete commutative algebras on $k$.

\item $SCR_k$ is the $\infty$-category of simplicial commutative rings over $k$

\item $LSym$ or simply $Sym$ denotes the monad on $Mod_k^{\le 0}$ associated with the forgetful functor $SCR_k \to Mod_k^{\le 0}$

\item $St_k$ and $dSt_k$ will denote respectively the $\infty$-category of stacks over the étale site of discrete commutative $k$-algebras and the $\infty$-category of derived stacks, also on the étale site.

\item Following \cite[6.2.2.1, 6.2.2.7, 6.2.3.4]{Lur18}, $Qcoh(X)$ will be the $\infty$-category of quasi-coherent sheaves on a derived stack $X$, ie $QCoh(X)$ is informally given as the limit of stable $\infty$-categories
$$QCoh(X) \coloneqq lim _{Spec\text{ }A \to X} Mod_A$$

\item A simplicial set or a derived stack $X$ will be $n$-truncated when $\pi_i(X) = 0$ for $i>n$, it will be $n$-connected when $\pi_i(X) = 0$ for $i \le n$.

\item All tensor products are assumed to  be derived.

\item The subscript notation $\mathcal{C}_p$ will denote the category of objects over $k \otimes \mathbb{F}_p$ when $\mathcal{C}$ is a category of "algebra type objects" over $k$. For example, when $SCR$ denoted the category of simplicial algebras over $k$, $SCR_p$ is the category of simplicial algebras over $k \otimes \mathbb{F}_p$.

\item The letters $\epsilon$, $\rho$ and $\eta$ will be reserved for generators of commutative differential graded algebras in specific degrees : $\epsilon$ will be in degree $-1$, $\rho$ in degree $0$ and $\eta$ in degree $1$. As commutative differential graded algebras, we have identifications
$$H^*(S^1) \simeq k[\eta]$$
and
$$H_*(S^1) \simeq k[\epsilon]$$
These three generators will be in weight $1$ when they generate a graded commutative differential graded algebra.

\end{itemize}
\end{Not}

This thesis was supported by the Nedag project ERC-2016-ADG-741501.

\section{Categorical notions}

In order to define our crystalline circle and to describe the de Rham-Witt complex, in mixed characteristics, as a graded loopspace
$$\textbf{Map}(S^1_{gr},X)$$
we will need to review some results on $\infty$-categories and derived geometry.

\subsection{$\infty$-categories}

This section gives the $\infty$-categorical notations and results which will be used in this thesis. We ferer to \cite{Lur09} for more details.

\begin{Not}
We will use the following notations and conventions regarding $\infty$-categories :

\begin{itemize}
\item  We will use quasi-categories as models for our $\infty$-categories.

\item  We will use Kan complexes as models for our $\infty$-groupoids.

\item The $\infty$-category of $\infty$-categories is denoted $Cat_\infty$, and the sub-$\infty$-category of Kan complexes is denoted $Gpd_\infty$.

\item The inclusion $Cat_\infty \to Gpd_\infty$ admits a left adjoint denoted $\widetilde{(-)}$.

\item The $\infty$-category of spaces is denoted $\mathcal{S}$.

\item When $\mathcal{C}$ is an $\infty$-category, and $x$ and $y$ are elements of $\mathcal{C}$, $Map_\mathcal{C}(x,y)$ denotes the mapping space of morphisms from $x$ to $y$. The set of connected components of morphisms from $x$ and $y$ is denoted $Hom_\mathcal{C}(x,y)$ or $[x,y]_\mathcal{C}$.

\item For $K$ a simplicial set and $\mathcal{C}$ an $\infty$-category, $Fun(K,\mathcal{C}) = \mathcal{C}^K$ is the $\infty$-category of functors from $K$ to $\mathcal{C}$. In particular, $Arrow(\mathcal{C}) \coloneqq \mathcal{C}^{\Delta^1}$ is the $\infty$-category of morphisms in $\mathcal{C}$. 

\item The $\infty$-category of presheaves on an $\infty$-category $\mathcal{C}$ is $\mathcal{C} \coloneqq Fun(\mathcal{C}^{op},\mathcal{S})$

\item In an $\infty$-category $\mathcal{C}$, with $X \in \mathcal{C}$ and $K$ a simplicial set, the tensorization of $X$ by $K$ denoted $K \otimes X$ is, when it exists, e.g when $\mathcal{C}$ is presentable, an object representing the functor
$$Z \in \mathcal{C} \mapsto Map_\mathcal{C}(X,Z)^K$$
and the cotensorization of $X$ by $K$ denoted $X^K$ is, when it exists, e.g when $\mathcal{C}$ is presentable, an object representing the functor
$$Z \in \mathcal{C}^{op} \mapsto Map_\mathcal{C}(Z,X)^K$$

\item A localization is a full sub-$\infty$-category $\mathcal{C}_0 \subset \mathcal{C}$ admitting a left adjoint.

\end{itemize}

\end{Not}

\begin{Thm} [Adjoint functor theorem] \label{adjointFunct}
Let $F : \mathcal{C} \to \mathcal{D}$ be a functor between presentable $\infty$-categories.
\begin{itemize}
\item The functor $F$ admits a right adjoint if and only if it preserves small colimits.
\item The functor $F$ admits a left adjoint if and only if it is accessible and preserves small limits.
\end{itemize}
\end{Thm}

\begin{proof}
See \cite[Corollary 5.5.2.9]{Lur09}.
\end{proof}

\subsection{Lax and colax fixed points}

\begin{Def} \label{defLFP}
Let $\mathcal{C}$ be a category and $F$ an endofunctor of $\mathcal{C}$. We define the category of lax fixed points with respect to $F$ as the fiber product of categories
$$LFP_F(\mathcal{C}) = \mathcal{C}^{\Delta^1} \times_{\mathcal{C}\times\mathcal{C}} \mathcal{C}$$
where $\mathcal{C} \to \mathcal{C} \times \mathcal{C}$ is given by the composition :
$$\mathcal{C} \xrightarrow{\Delta} \mathcal{C} \times \mathcal{C} \xrightarrow{Id \times F} \mathcal{C} \times \mathcal{C}$$

We define the category of fixed points with respect to $F$ as the fiber product of categories
$$FP_F(\mathcal{C}) = \mathcal{C}^{\widetilde{\Delta^1}} \times_{\mathcal{C}\times\mathcal{C}} \mathcal{C}$$

The inclusion $\Delta^1 \subset \widetilde{\Delta^1}$ defines an inclusion $FP_F(\mathcal{C}) \subset LFP_F(\mathcal{C})$.
\end{Def}

\begin{Rq}
The elements of the lax fixed points category $LFP_F(\mathcal{C})$ are given by pairs $(x,f)$ with $x \in \mathcal{C}$ and $f : x \to Fx$. Then the category of fixed points $FP_F(\mathcal{C})$ is the full subcategory of $LFP_F(\mathcal{C})$ on pairs $(x,f)$ where $f$ is an equivalence.
\end{Rq}

\begin{Rq} \label{CFPDef}
Colax fixed points are defined similarly : $CFP_F(\mathcal{C})$ has elements of the form $(x,f)$ with $x \in \mathcal{C}$ and $f : Fx \to x$.
\end{Rq}

\begin{Prop} \label{LFPop}
We have natural identifications
$$CFP_F(\mathcal{C})^{op} \simeq LFP_F(\mathcal{C}^{op})$$
\end{Prop}

\subsection{Objects with endomorphisms}

\begin{Def}
Let $\mathcal{C}$ be a category, we define the category of objects of $\mathcal{C}$ endowed with an endomorphism as $Fun(B\mathbb{N},\mathcal{C})$, where $B\mathbb{N}$ is the $1$-category with one object and $\mathbb{N}$ as its monoid of arrows. We denote the associated category $\mathcal{C}^{endo}$.

The base point $* \to B\mathbb{N}$ gives by restriction the forgetful functor
$$(X,\alpha) \in \mathcal{C}^{endo} \mapsto X \in \mathcal{C}$$
\end{Def}

\begin{Prop} \label{PropAdjointEndo}
When $\mathcal{C}$ has $\mathbb{N}$-indexed coproducts, the forgetful functor
$$\mathcal{C}^{endo} \to \mathcal{C}$$
admits a left adjoint $L$, given by
$$L : X \mapsto (\bigsqcup_\mathbb{N} X,S)$$
with the shift $S$ given by the inlusion in the second factor
$$\bigsqcup_\mathbb{N} X \xrightarrow{i_2} X \sqcup \bigsqcup_\mathbb{N} X \cong \bigsqcup_\mathbb{N} X$$

When $\mathcal{C}$ has $\mathbb{N}$-indexed products, the forgetful functor
$$\mathcal{C}^{endo} \to \mathcal{C}$$
admits a right adjoint $R$, given by
$$R : X \mapsto (\prod_\mathbb{N} X,S)$$
with the shift $S$ induced by the projection map
$$\prod_\mathbb{N} X \cong (\prod_\mathbb{N} X) \Pi (X) \xrightarrow{p_1} \prod_\mathbb{N} X$$

\end{Prop}

\begin{Def}
Let $(X,\alpha)$ and $(Y,\beta)$ be two objects in $\mathcal{C}^{endo}$, we define, when it exists, the morphism object with endomorphism $\underline{Hom}_{\mathcal{C}^{endo}}((X,\alpha),(Y,\beta))$ as an element of $\mathcal{C}^{endo}$ representing :
$$(T,\gamma) \mapsto Map_{\mathcal{C}^{endo}}((X \times T, \alpha \times \gamma),(Y,\beta)) : \mathcal{C}^{endo} \to \mathcal{S}$$

\end{Def}

\begin{Prop} \label{Endendo}

We assume $\mathcal{C}$ has internal morphism objects and equalizers. Let $(X,\alpha)$ and $(Y,\beta)$ be two objects in $\mathcal{C}^{endo}$, the morphism object with endomorphism $\underline{Hom}_{\mathcal{C}^{endo}}((X,\alpha),(Y,\beta))$ can be computed as the equalizer :
$$eq(\underline{Hom}_{\mathcal{C}}((X,\alpha),(Y,\beta))^\mathbb{N} \rightrightarrows \underline{Hom}_{\mathcal{C}}((X,\alpha),(Y,\beta))^\mathbb{N})$$
where one of the arrows is given by postcomposing by $\beta$ and the other is defined as by applying the shift $((f_i) \mapsto (f_{i+1}))$ and precomposing by $\alpha$. Its endomorphism is given by the shift $(f_i) \mapsto (f_{i+1})$.

Explicitly, an object of $\underline{Hom}_{\mathcal{C}^{endo}}((X,\alpha),(Y,\beta))$ can be thought of as a sequence of morphisms $f_i : X \to Y$ such that $\beta \circ f_i$ is identified with $f_{i+1} \circ \alpha$.
\end{Prop}

\subsection{Nerve and conerve constructions}

\begin{Def}
Let $\mathcal{C}$ be a presentable category endowed with the cartesian symmetric monoidal structure. We denote $Gp(\mathcal{C})$ the category of group objects in $\mathcal{C}$, it can be described by a full subcategory of $Fun(\Delta^{op},\mathcal{C})$ on objects satisfying the Segal condition. See \cite{Lur09} for details.

The colimit functor, or geometric realization, will be denoted $$B = | - | : Gp(\mathcal{C}) \to \mathcal{C}_*$$
where $\mathcal{C}_*$ is the category of pointed objects in $\mathcal{C}$.
\end{Def}

\begin{Prop}
The functor $B$ admits a right adjoint sending $* \to X$ to its \v{C}ech nerve.
\end{Prop}

\begin{Rq}
Informally the nerve of $U \to X$ has $n$-simplices given by
$$U \times_X U \times_X ... \times_X U$$
with $n+1$ copies of $U$. See \cite[Proposition 6.1.2.11]{Lur09} for details on \v{C}ech nerves.

\end{Rq}

\begin{Def}
Similarly, we denote $coGr(\mathcal{C})$ the category of cogroup objects in $\mathcal{C}$, it can be described by a full subcategory of $Fun(\Delta,\mathcal{C})$ on objects satisfying the Segal condition.

The limit functor, or totalization, will be denoted $$lim_\Delta = Tot(-) : coGr(\mathcal{C}) \to \mathcal{C}_{/*}$$
where $\mathcal{C}_{/*}$ is the category of copointed objects in $\mathcal{C}$.
\end{Def}

\begin{Prop}
The functor $lim_\Delta$ admits a left adjoint where $X \to *$ is sent to its conerve.
\end{Prop}

\subsection{Simplicial sets}
\begin{Def}
Let $X$ be a simplicial set, the cosimplicial algebra of cohomology of $X$ has $n$-simplices given by
$$(k^X)_n \coloneqq k^{X_n}$$
where faces and degeneracies are induced by the ones on $X$.
\end{Def}

\subsection{Chain complexes}

\begin{Thm} [The projective model structure on chain complexes]
For $f : M \to N$ a chain complex morphism in $\bf{dg-mod_k}$, $f$ is said to be :
\begin{itemize}
\item a weak equivalence if it is a quasi-isomorphism, ie it induces an isomorphism on homology groups $H^i(M) \to H^i(N)$.
\item a fibration if it is a levelwise surjection.
\end{itemize}

This data makes $\bf{dg-mod_k}$ into a model category called the "projective model structure on chain complexes". We will denote this model category $\bf{dg-mod_{k, proj}}$
\end{Thm}

\begin{proof}
See \cite{Ho91}.
\end{proof}

\begin{Thm} [The injective model structure on chain complexes]
For $f : M \to N$ a chain complex morphism in $\bf{dg-mod_k}$, $f$ is said to be :
\begin{itemize}
\item a weak equivalence if it is a quasi-isomorphism, ie it induces an isomorphism on homology groups $H^i(M) \to H^i(N)$.
\item a cofibration if it is a levelwise injection.
\end{itemize}

This data makes $\bf{dg-mod_k}$ into a model category called the "injective model structure on chain complexes". We will denote this model category $\bf{dg-mod_{k, inj}}$
\end{Thm}

\begin{proof}
See \cite{Dun10}.
\end{proof}

\begin{Def}
We define the naive coconnective truncation of a chain complex on $k$ as the left adjoint to inclusion of $\bf{dg-mod_k^{\ge 0}} \subset \bf{dg-mod_k}$
It can be described as
$$(A_{-1} \rightarrow A_0 \rightarrow A_1 \rightarrow A_2) \mapsto (0 \to A_0 \rightarrow A_1 \rightarrow A_2)$$
It does not preserve quasi-isomorphisms.

We define the coconnective truncation of a chain complex on $k$ as the right adjoint to inclusion of $\bf{dg-mod_k^{\ge 0}} \subset \bf{dg-mod_k}$
It can be described as
$$(A_{-1} \rightarrow A_0 \rightarrow A_1 \rightarrow A_2) \mapsto (0 \to coker(d) \rightarrow A_1 \rightarrow A_2)$$
It does induces a functor of $\infty$-categories $Mod_k \to Mod_k^{\ge 0}$ which is right adjoint to the inclusion.

We define the naive connective truncation of a chain complex on $k$ as the right adjoint to inclusion of $\bf{dg-mod_k^{\le 0}} \subset \bf{dg-mod_k}$
It can be described as $$(A_{-2} \rightarrow A_{-1} \rightarrow A_0 \rightarrow A_1) \mapsto (A_{-2} \rightarrow A_{-1} \rightarrow A_0 \rightarrow 0)$$
It does not preserve quasi-isomorphisms.

We define the connective truncation of a chain complex on $k$ as the left adjoint to inclusion of $\bf{dg-mod_k^{\le 0}} \subset \bf{dg-mod_k}$
It can be described as $$(A_{-2} \rightarrow A_{-1} \rightarrow A_0 \rightarrow A_1) \mapsto (A_{-2} \rightarrow A_{-1} \rightarrow Ker(d) \rightarrow 0)$$
It does induces a functor of $\infty$-categories $Mod_k \to Mod_k^{\le 0}$ which is left adjoint to the inclusion.

\end{Def}

\begin{Def}
A commutative differential graded algebra is a chain complex $(A,d)$ endowed with a graded algebra structure over $k$ such that
\begin{itemize}
\item The multiplication is graded commutative, that is, for $x \in A_n$ and $y \in A_m$, we have $xy = (-1)^{mn}yx$ in $A_{m+n}$.
\item If $x \in A_{2n+1}$ is in odd degree, $x^2=0$.
\item The differential $d$ satisfies the Leibniz rule
$$d(xy) = (dx)y+(-1)^m x(dy)$$
for $x \in A_m$.
\end{itemize}
\end{Def}

\begin{Rq}
The category of commutative differential graded algebras is identified with a subcategory of the category of commutative algebra objects in the $1$-category $\bf{dg-mod_k}$ on objects satisfying $x^2=0$ for $x \in A_{2n+1}$.
\end{Rq}

\subsection{Graded mixed chain complexes}

We recall here elementary results on graded mixed complexes. We recommend \cite{TV09} for definitions and usage of graded mixed complexes in derived geometry.

\begin{Def}
We define the $1$-category of graded mixed complexes $\bf{\epsilon-mod_k^{gr}}$ as the module category $\bf{k[\epsilon]-dg-mod^{gr}_k}$ where $k[\epsilon]$ is the commutative algebra object in $\bf{dg-mod_k^{gr}}$ where $\epsilon$ is in weight $1$ and degree $-1$, therefore the action of $\epsilon$ on a graded mixed module is a morphism $M \to M((1))[-1]$ : the operator decreases the degree, contrary to the differential, and increases the weight.
\end{Def}

\begin{Rq}
The object $k[\epsilon]$ as an algebra object in $\bf{dg-mod_k^{gr}}$ can be alternatively defined as the cohomology algebra of the circle $H^*(S^1;k)$. In fact Proposition \ref{grMixedEqui} makes this identification more precise.
\end{Rq}

\begin{Def}
The object $k[\epsilon]$ is naturally a coalgebra in graded differential graded modules using the comultiplication
$$ \Delta : \epsilon \in k[\epsilon] \mapsto 1 \otimes \epsilon + \epsilon \otimes 1 \in k[\epsilon] \otimes_k k[\epsilon]$$
and counit
$$ n : \epsilon \in k[\epsilon] \mapsto 0 \in k$$

This makes $\bf{\epsilon-dg-mod_k^{gr}}$ into a symmetric monoidal category.
\end{Def}

\begin{Def}
The $\infty$-category of graded mixed complex is the localization of $\bf{\epsilon-dg-mod_k^{gr}}$ at quasi-isomorphisms. It is naturally a symmetric monoidal $\infty$-category.
\end{Def}

\begin{Prop}
Defining fibrations and trivial fibrations in $\bf{\epsilon-dg-mod_k^{gr}}$ as the ones in $\bf{dg-mod_{k,proj}}$, makes it into a model category called the projective model structure on $\bf{\epsilon-dg-mod_k^{gr}}$.

Defining cofibration and trivial cofibration in $\bf{\epsilon-dg-mod_k^{gr}}$ as the ones in $\bf{dg-mod_{k,proj}}$, makes it into a model category called the injective model structure on $\bf{\epsilon-dg-mod_k^{gr}}$.

They both present the same $\infty$-category denoted $\epsilon-Mod^{gr}_k$.
\end{Prop}

\begin{proof}
See \cite{TV09} for details.
\end{proof}

\subsection{The bar construction}

We recall the following construction which is useful to construct resolutions, see \cite[\S I.1.5]{Ill71} for details.

\begin{Def}
Let $T$ be a comonad on a category $\mathcal{C}$. Let $c \in \mathcal{C}$, we define an augmented simplicial object in $\mathcal{C}$ by $T^{\bullet+1}(c) \to c$. We call this the bar construction and denote the simplicial object $Bar_C(c)$.
\end{Def}

\section{Reminders on algebra}

\subsection{Simplicial and cosimplicial rings}

\subsubsection{Simplicial algebras}

\begin{Def}
We define $SCR_k$ the $\infty$-category of simplicial commutative algebra over $k$ as the localization of the $1$-category of simplicial commutative algebra over $k$ with respect to the model structure where we ask weak equivalences to be isomorphisms on homology and fibrations to be level-wise surjections. From \cite[\S 25.1.1]{Lur18}, this category admits a useful alternative description as the category generated from $Poly_k$ under sifted colimits, where $Poly_k \subset CRing_k$ is the full sub-$1$-category of polynomial rings on a finite set of variables. Explicitly this category can be defined by the full subcategory $Fun(Poly_k^{op},\mathcal{S})$ on functors preserving finite products.
\end{Def}

\begin{Def}
The simplicial version of the Dold-Kan construction induces a functor
$$\theta^{cn} : SCR_k \to \mathbb{E}_\infty Alg_k^{cn}$$
where $\mathbb{E}_\infty Alg_k^{cn} = CAlg(Mod_k^{cn})$ is the $\infty$-category of connective $\mathbb{E}_\infty$-algebra on $k$.
\end{Def}

\begin{Prop}
The functor $\theta^{cn}$ previously defined preserves small limits and colimits and is conservative.
\end{Prop}

\begin{proof}
See \cite[Proposition 25.1.2.2]{Lur18}.
\end{proof}

\begin{Def}
Let $A$ be a simplicial algebra and $M$ a simplicial module over $A$. The square zero extension $A \oplus M$ of $A$ by $M$ is given by a square zero extension degreee wise :
$$(A \oplus M)_n \coloneqq A_n \oplus M_n$$
with the canonical faces and degeneracies maps, every element in $M_n$ squares to zero.
\end{Def}

\subsubsection{Cosimplicial algebras}

\begin{Def}
Let $f : A \to B$ be a morphism of cosimplicial algebras.
\begin{itemize}

\item The map $f$ is said to be a weak equivalence when it induces an equivalence on all the homology groups.

\item The map $f$ is said to be a fibration if it is levelwise surjective
\end{itemize}
\end{Def}

\begin{Prop}
The data of weak equivalences and fibrations defines a simplicial model structure on $\bf{coSCR_k}$.
\end{Prop}

\begin{proof}
See \cite[Theorem 2.1.2]{To06}.
\end{proof}

\begin{Def}
We define the $\infty$-category of cosimplicial commutative rings $coSCR_k$ over $k$ as the localization of $\bf{coSCR_k} = Fun(\Delta,\bf{CRing_k})$ with respect to the simplicial model structure defined above.
\end{Def}

\begin{Def}
The cosimplicial version of the Dold-Kan construction induces a functor
$$\theta^{ccn} : coSCR \to CAlg_k^{ccn}$$
Furthermore, $\theta$ preserves small limits and colimits and is conservative.
\end{Def}

\begin{Def}
We recall the cotensorization of  a cosimplicial algebra $A$ by a simplicial set $X$, it is the cosimplicial algebra which has as $n$-simplices the algebra
$$(A^X)_n \coloneqq \prod_{X_n} A_n$$

The mapping space of cosimplicial algebras, by requiring the usual universal property, is given by the following construction. Let $A, B \in coSCR_k$,
$$Map_{coSCR_k}(A,B)_n \coloneqq Hom_{coSCR_k}(A,B^{\Delta^n})$$
\end{Def}

\subsubsection{Dold-Kan correspondance}

\begin{Def}
For $\mathcal{A}$ an abelian category, the normalization construction is a functor : $$\bf{\mathcal{A}^{\Delta^{op}}} \to \bf{Ch^{\le 0}(\mathcal{A})}$$
where we denote by $\bf{Ch^{\le 0}(\mathcal{A})}$ the category of connective chain complexes in $\mathcal{A}$.
A simplicial object $E$ is sent to a chain complex having in degree $-n$ the element $\bigcap_{i=1}^{n} d^n_i$, meaning the intersection of the $i$-th face maps for $i>0$ : the induced differential is given by the remaining $0$-face map.

\end{Def}

\begin{Prop}
The normalization functor $N$ admits a right adjoint $D$ called denormalization. And they are equivalences of categories. Let $A$ be a simplicial ring, there is a Quillen equivalence
$$\bf{NA-dg-mod^{\le 0}} \simeq \bf{A-Mod}$$
between the category of connective dg-modules over $NA$ and the category of simplicial modules over $A$.
\end{Prop}

\begin{Rq}
The natural isomorphism $N D \simeq Id$ is naturally monoidal with respect to the shuffle and Alexander-Whitney maps and the isomorphism $D N \simeq Id$ admits a non-symmetriccal monoidal structure.

The normalization is  lax symmetric monoidal but $D$ is only lax monoidal.
\end{Rq}

\begin{proof}
See : \cite{SS03}
\end{proof}

\begin{Def}
The normalization functor
$$N : \bf{cosMod_k} \to \bf{dg-mod_k^{\ge 0}}$$
from cosimplicial modules to differential graded modules induces the notion of homology group of a cosimplicial algebra $A$ by composing the forgetful functor
$$ \bf{coSCR_k} \to \bf{cosMod_k}$$
with the normalization functor and the cohomology functor of a chain complex $H^i$.
\end{Def}

\subsection{$\delta$-structures}

\begin{Def}
Let $A$ be a discrete commutative ring, a classical Frobenius lift on $A$ is given by a ring endomorphism $F : A \to A$ such that $F$ induces the Frobenius morphism on $A/p$, ie $F(a)-a^p$ is $p$-divisible for any $a \in A$.
\end{Def}

\begin{Def}
Let $A$ be a discrete commutative ring, a $\delta$-structure on $A$ is given by a function $A \to A$ such that $\delta(0)=0$,   $\delta(1)=0$, and for any $a,b \in A$,
$$\delta(a+b) = \delta(a) + \delta(b) - \sum_{i=1}^{p-1} \frac{1}{p} \binom{p}{i}a^ib^{p-i}$$
$$\delta(ab) = a^p \delta(b) + b^p \delta(a) + p\delta(a)\delta(b)$$
\end{Def}

\begin{Rq}
Any $\delta$-structure induces a classical Frobenius lift $F(a) = a^p+p\delta(a)$. This defines a bijection between $\delta$-structures and classical Frobenius lifts when $A$ is $p$-torsion free. In fact, Theorem \ref{Witt2Fiber} identifies $\delta$-structures with derived Frobenius lifts .
\end{Rq}

\subsection{The ring of Witt vectors}
As seen in \cite{Ser62}, when $k$ is a field of positive characteristic $p$, there is a unique complete discrete valuation ring of characteristic $0$ having $k$ as a residual field and $(p)$ as its maximal ideal, this is the ring of Witt vectors on $k$. The ring of Witt vectors admits various generalizations when $k$ need not be a field of positive characteristic. For details on Witt vectors, see \cite{Haz08} and \cite{Rab14}.

\begin{Def}
We define the ring of $p$-typical Witt vectors $W(A)$ associated to a commutative ring $A$. As a set $W(A)$ is given by $A^\mathbb{N}$.

We define the ghost map as function : 
$$gh : (a_n) \in W(A) \mapsto (w_n) \in A^\mathbb{N}$$
where $w_n \coloneqq \sum_{i=0}^{n} p^i a_i^{p^{n-i}}$.

The set $W(A)$ admits a unique functorial ring structure making $gh$ into a natural transformation of commutative rings. An element $x = (a_n) \in W(A)$ will have $(a_n)$ as its "Witt coordinates" and $(w_n)$ as its "ghost coordinates".

The ghost morphism can be seen as a morphism of commutative rings in affine schemes
$$W \to \mathbb{G}_a^\mathbb{N}$$

\end{Def}

\begin{Prop}
The ghost map :
$$gh : W(A) \to A^\mathbb{N}$$
is injective when $A$ is $p$-torsion free. In this case, the ghost coordinates can be use without ambiguity.
\end{Prop}

\begin{Def}
We define the $p$-typical Witt vectors of length $m$ as the quotient ring $$W_m(A) = W(A)/I$$ where $I$ is the ideal of elements $(a_n)$ where $a_n=0$ for $n < m$. Elements of $W_m(A)$ can be written $(a_0, ..., a_{n-1})$.
\end{Def}

\begin{Def}
The Verschiebung map is a natural transformation
$$ V : W(A) \to W(A)$$
defined as $(a_i) \mapsto (0,a_0,a_1, ...)$. In ghost coordinates, it is expressed as
$$(w_n) \mapsto (0, p w_0, p w_1, ...)$$

It is an additive morphism and induces a morphism of truncated Witt vectors :
$$ V : W_m(A) \to W_{m+1}(A)$$
\end{Def}

\begin{Def}
The Frobenius map is a natural transformation
$$ F : W(A) \to W(A)$$
defined implicitly in ghost coordinates as $(w_n) \mapsto (w_1,w_2, ...)$ and by requiring functoriality.

The Frobenius $F$ is a ring morphism and induces a morphism of truncated Witt vectors :
$$ F : W_m(A) \to W_{m-1}(A)$$
\end{Def}

\begin{Prop}
We have the usual identities :
$$FV = p$$
and
$$V(F(x)y) = x V(y)$$
\end{Prop}

\begin{proof}
Both results can be proven by functoriality and reducing to the case of $A$ being without $p$-torsion where ghost coordinates can be used.
\end{proof}

\begin{Def}
The projection morphism $W(A) \to A$ admits a multiplicative section $a \mapsto [a] \coloneqq (a,0, ...)$, $[a]$ is called the Teichmüller representative of $a$.
\end{Def}

\begin{Thm} \label{Witt2Fiber}
The ghost coordinates induce a homotopy pullback diagram :
$$\begin{tikzcd}
W_2(A) \arrow[r,"w_0"] \arrow[d,"w_1"] \arrow[rd, phantom, "\lrcorner", very near start]
& A  \arrow[d,"Fr_p \circ pr"] \\
A \arrow[r,"pr"]
& A \otimes_{\mathbb{Z}}^{\mathbb{L}} \mathbb{F}_p
\end{tikzcd}$$

As sections of the projection morphism $w_0 : W_2(A) \to A$ are given by $\delta$-structures, we can identify $\delta$-structures with derived Frobenius lifts on $A$.
\end{Thm}

\begin{proof}
See \cite[Remark 2.5]{BS22}.
\end{proof}

\begin{Prop}
The Witt vectors admits a natural $\delta$-structure, the forgetful functor $$\bf{\delta-CRing_k} \to \bf{CRing_k}$$ admits $W$ as a right adjoint. From the previous proposition, we see the Witt vector functor as the construction of adding cofreely a $\delta$-structure on a commutative ring.
\end{Prop}

We also recall a theorem of Almkvist for the reference, see \cite{A78}, which we will not use. For a commutative ring $R$, we define $\bf{Proj(R)}$ the category of projective finitely generated $R$-modules. We write the category of endomorphisms in $\bf{Proj(R)}$ :
$$\bf{End(Proj(R))} \coloneqq \bf{Fun(B\mathbb{N},Proj(R))}$$
as a functor category where $\bf{B\mathbb{N}}$ is the category with one object and $\mathbb{N}$ as its monoid of arrows.

\begin{Def}
We define $K_0(\bf{End(Proj(R))})$ as the free abelian group generated by isomorphism classes of objects in $\bf{End(Proj(R))}$ modulo the subgroup generated by elements of the form
$$[(M,f)]-[(M',f')]-[(M'',f'')]$$
when there is a morphism of short exact sequences :
$$\begin{tikzcd}
0 \arrow[r]
& M' \arrow[r] \arrow[d,"f'"]
& M \arrow[r] \arrow[d,"f"]
& M'' \arrow[r] \arrow[d,"f''"]
& 0 \\
0 \arrow[r]
& M' \arrow[r]
& M \arrow[r]
& M'' \arrow[r]
& 0
\end{tikzcd}$$

The category $K_0(\bf{Proj(R)})$ is similarly defined. The tensor product of modules endows $K_0(\bf{End(Proj(R))})$ and $K_0(\bf{Proj(R)})$ with the structure of commutative rings.

The kernel of the canonical morphism
$$[(M,f)] \in K_0(\textbf{End(Proj(R))}) \mapsto [M] \in K_0(\textbf{Proj(R)})$$
is denoted $\widetilde{K}_0(\bf{End(Proj(R))})$.
\end{Def}

\begin{Thm}
There is a canonical ring isomorphism
$$\widetilde{K}_0(\textbf{End(Proj(R))}) \cong Rat(R)$$
where $Rat(R)$ is the subring of the ring of big Witt vectors $W_{big}(R)$ which are rational functions. When $W_{big}(R)$ is identified with formal series $(1+tR[[t]])$, $Rat(R)$ is the subring
$$\left \{ \frac{1+a_1 t + ... + a_m t^m}{1+b_1 + ... + b_n t_n} : a_i, b_i \in R \right \} \subset (1+tR[[t]])$$

The morphism is given by taking the characteristic polynomial of the endomorphism
$$[(M,f)] \in \widetilde{K}_0(\textbf{End(Proj(R))}) \mapsto \chi_f(t) \in (1+tR[[t]])$$
\end{Thm}

\begin{Rq}
Many construction on the Witt vectors can be described at the level of the category $\textbf{End(Proj(R))}$, therefore defining a satisfying explanation for the structure arising on Witt vectors and the compatibility between the various maps.

For example, the Frobenius $F_n$ acts on $\textbf{End(Proj(R))}$ as follows
$$[(M,f)] \mapsto [(M,f^n)]$$
The Verschiebung $V_n$ acts as
$$[(M,f)] \mapsto [(M^{\oplus n},f_n)]$$
where $f_n$ sends the $i$-th copy of $M$ to the $(i+1)$-th copy of $M$ by identity and sends the $n$-th copy of $M$ to the first copy of $M$ by $f$.

The $\Lambda$-structure can be seen as taking exterior products : $[(M,f)] \mapsto [(\Lambda_n M,\Lambda_n f)]$
\end{Rq}

\section{Preliminaries in derived geometry}

In this section we recall notions of derived geometry. We refer to \cite{To06b} and \cite{To13} for overviews on the subject. The various notions and results of this section are borrowed from \cite{TV06}, \cite{Lur17} and \cite{Lur18}.

\subsection{Stacks and derived stacks}

\begin{Def}
We will consider stacks and derived stacks with respect to the étale topology on simplicial algebras. We denote $St_k$ the $\infty$-category of hypercomplete stacks and $dSt_k$ the $\infty$-category of hypercomplete derived stacks.
\end{Def}

\begin{Prop}
The inclusion of stacks into derived stacks : $St_k \subset dSt_k$ admits a right adjoint called truncation, denoted $t_0$, any derived stack $X$ admits a canonical morphism from its truncation $t_0(X) \to X$.
\end{Prop}

\begin{Rq}
On affine derived schemes, truncation is given by
$$t_0(Spec(A)) = Spec \textrm{ }\pi_0(A)$$
\end{Rq}

\begin{Rq}
Seeing a derived scheme $(X,\mathcal{O}_X)$ as a classical scheme $(X,\pi_0(\mathcal{O}_X))$ with $\pi_i(\mathcal{O}_X)$ being a quasi-coherent $\pi_0(\mathcal{O}_X)$-modules, the truncation $t_0(X,\mathcal{O}_X)$ is given by the underlying scheme $(X,\pi_0(\mathcal{O}_X))$.
\end{Rq}

We include here a motivation for the importance of the Frobenius morphism.

\begin{Prop}
Let $k$ be a commutative ring, there is an isomorphism of abelian groups
$$End_{Gp_k}(\mathbb{G}_a) \cong k Id \oplus \bigoplus_p \bigoplus_{n \ge 1} k[p] F_{p^n}$$
where $p$ belongs in the prime numbers and $k[p] \coloneqq \{x \in k : px=0 \}$ is the ideal of $p$-torsion elements of $k$.

The set $k[p] F_{p^n}$ corresponds with morphisms of the form
$$b \in B = \mathbb{G}_a(B) \mapsto xb^{p^n} \in B = \mathbb{G}_a(B)$$
where $x \in k[p]$. We notice that composition is characterized by
$$xF_{p^m} \circ y F_{q^n} = 0$$
when $p \neq q$ and
$$xF_{p^m} \circ y F_{p^n} = xy^{p^m} F_{p^{m+n}}$$

\end{Prop}

\begin{proof}
The derived stack $\mathbb{G}_a$ being a classical stack, the space of group endomorphisms on $\mathbb{G}_a$ is a set. It identifies with the set of polynomials $P \in k[X]$ satisfying
$$P(0)=0$$
and
$$P(X + Y) = P(X)  + P(Y)$$
the second condition can be checked degreewise, meaning we have to find the elements $x \in k$ such that
$$x(X+Y)^m = xX^m + xY^m$$
elementary computations using the Legendre formula find that $x$ is null when $m$ admits two disctinct prime divisors and $x$ must satisfy $px=0$ when $m=p^n$ with $n \ge 1$.
\end{proof}

\begin{Rq}
In particular, for $k$ a torsion-free commutative ring, morphisms of group schemes of $\mathbb{G}_a \rightarrow \mathbb{G}_a$ are given by multiplication by an element of $k$. That is
$$End_{Gp_k}(\mathbb{G}_a) \cong k$$

For $k$ a field of characteristic $p$, the ring of morphisms is given by the non-commutative algebra $k[F]$ with $F$ the Frobenius, with the relations $F.a = a^p.F$, for $a \in k$.
\end{Rq}

\subsection{$\mathcal{O}$-modules and quasi-coherent sheaves}

\begin{Def} 
For a derived stack $X$, we introduce the $\infty$-category of $\mathcal{O}_X$-modules. We first define the structure stack $\mathcal{O}_X$ of $X$ as the stack over the big étale site over $X$ sending $Spec(B) \to X$ to the underlying $\mathbb{E}_\infty$-ring $\theta(B)$. It is a stack of $\mathbb{E}_\infty$-rings. Then the category $\mathcal{O}_X-Mod$ of $\mathcal{O}_X$-modules is the category of modules over the commutative ring object in derived stack $\mathcal{O}_X$.

\end{Def}

\begin{Rq}
Intuitively an $\mathcal{O}_X$-module a given by a $\theta(B)$-module for every simplicial commutative algebra $B$, functorially in $B$ and satisfying descent as a presheaf on the étale site over $X$. In this description the full subcategory $QCoh(X) \subset \mathcal{O}_X-Mod$ consists of $\mathcal{O}_X$-modules where transition maps between modules are equivalences.
\end{Rq}

\begin{Prop}
The inclusion $QCoh(X) \subset \mathcal{O}_X-Mod$ has a right adjoint, which we will denote $Q$.
\end{Prop}

\begin{proof}
The source and target categories are both presentable, and the inclusion preserves limits. We conclude using \cite[5.5.2.9]{Lur17}.
\end{proof}

\begin{Prop} [Functoriality of modules and quasi-coherent sheaves]
A morphism of derived stack $\pi : X \to Y$ induces an adjunction between their category of $\mathcal{O}$-modules : 

$$\widetilde{\pi}^* : \mathcal{O}_Y-Mod \rightleftarrows \mathcal{O}_X-Mod : \widetilde{\pi}_*$$

and between their category of quasi-coherent sheaves :
$$\pi^* : QCoh(Y) \rightleftarrows QCoh(X) : \pi_*$$

We have compatibility between pullbacks and the inclusion of quasi-coherent sheaves into $\mathcal{O}$-module :

\[\begin{tikzcd}
QCoh(Y) \arrow{r}{\pi^*} \ar[d] & QCoh(X) \ar[d]\\
\mathcal{O}_Y-Mod \arrow{r}{\widetilde{\pi}^*} & \mathcal{O}_X-Mod,
\end{tikzcd}\]

which means that the $\mathcal{O}$-module pullback of a quasi-coherent sheaf is quasi-coherent. However, the analogous statement for pushforwards is incorrect, the adjunctions allow the calculation of quasi-coherent pushforwards as the composition of $\mathcal{O}$-module pushforward with the adjoint of the inclusion $QCoh(Y) \subset \mathcal{O}_Y-Mod$, ie $$\pi_* = Q \circ \widetilde{\pi}_*.$$
\end{Prop}

\begin{proof}
See \cite[Proposition 2.5.1]{Lur11c} .
\end{proof}

\begin{Def}
We will use the standard t-structure on quasi-coherent complexes, see \cite[Remark 6.2.5.8]{Lur18} for details. Let $X$ be a derived stack, a quasi-coherent complex is said to be connective when its pullback to every affine is connective.
\end{Def}

\begin{Prop}
Let $X$ a derived stack and $\mathcal{F} \in QCoh(X)$, the following conditions are equivalent: :
\begin{itemize}
\item The quasi-coherent sheaf $\mathcal{F}$ is perfect.
\item The quasi-coherent sheaf $\mathcal{F}$ is dualizable in $QCoh(X)$.
\end{itemize}
\end{Prop}

\begin{proof}
See \cite[Proposition 2.7.28]{Lur11c}.
\end{proof}

\begin{Prop}
Let $X$ be a Deligne-Mumford stack, then :
\begin{itemize}
\item The inclusion $QCoh(X) \subset \mathcal{O}_X-Mod$ preserves small colimits.
\item $QCoh(X)$ is stable and presentable.
\item $QCoh(X)$ is both right and left t-complete, with respect to the standard t-structure.
\end{itemize}
\end{Prop}

\begin{proof}
See \cite[Proposition 2.3.13, Proposition 2.3.18]{Lur11c}.
\end{proof}

\begin{Rq} \label{tstrQCoh}
Let $X$ is a stack represented as a simplicial scheme $Spec(A^\bullet)$ with $A$ an augmented cosimplicial commutative algebra such that $H^i(A)=k$ and $k$ is $p$-torsion-free. Proposition \ref{conneAffSt} provides with a cofibrant replacement $QA$ such that boundary maps $QA^m \to QA^n$ are flat. Then proceeding as in \cite[Notation 1.2.12]{MRT20} and \cite[\S 6.2.5]{Lur18} defines a left-complete t-structure on $QCoh(X)$ compatible with the equivalence
$$QCoh(X) \simeq lim_n Mod_{QA^n}$$
Then
$$QCoh(X)_{\ge 0} \simeq lim_n Mod_{QA^n}^{\ge 0}$$
and
$$QCoh(X)_{\le 0} \simeq lim_n Mod_{QA^n}^{\le 0}$$

Furthermore $A-Mod$ admits a t-structure and the inclusion
$$QCoh(A) \subset A-Mod$$
identifies $QCoh(A)$ with the left completion of $A-Mod$ with respect to its $t$-tructure. We will not prove the previous result, the proof in \cite[\S 6.2.5]{Lur18} should carry out without difficulty.
\end{Rq}

\subsection{Construction of derived stacks}

\paragraph*{Linear stacks}
\begin{Def}
Let $E$ be a quasi-coherent complex on $X$, the linear stack associated to $E$ is the stack over $X$ given by
$$u: Spec(B) \to X \mapsto Map_{N(B)-mod}(u^*E,N(B)) \in SSet$$
with $N$ the normalization functor. It is denoted $\mathbb{V}(E)$.

For example, if $E$ is connective, $\mathbb{V}(E)$ is simply the relative affine scheme $Spec_X Sym_{\mathcal{O}_X}(E)$.
\end{Def}

\begin{Def}
We define the derived stack $\mathbb{G}_m$ associating to a simplicial commutative algebra $B$ the simplicial set of autoequivalences of $N(B)$ :
$$\mathbb{G}_m(B) \coloneqq Map^{eq}_{N(B)-mod}(N(B),N(B))$$

In fact, $\mathbb{G}_m$ is representable by the discrete ring $k[X,X^{-1}]$, and is equivalent to the usual derived stack classifying the group of units of a simplicial algebra : $B \mapsto B^*$.
\end{Def}

\paragraph*{$K(-,n)$ construction}

\begin{Def}
When $G$ is a group, we can define the simplicial set $E(G,1)$ which $m$-simplices are $(G^{m+1})_m$ and faces and degeneracies are given by projections and diagonals, $G$ acts on the left on $E(G,1)$ and $BG \coloneqq K(G,1)$ is defined as the quotient simplicial set $E(G,1)/G = (G^n)_n$.

This construction generalizes to a simplicial group $G$ where $K(G,1)_n \coloneqq K(G_n,1)_n$. The simplicial set $E(G,1)$ is simply $G_n^{n+1}$

We generalize further to a group in derived stacks $G$, then $BG \coloneqq K(G,1)$ and $E(G,1)$ are defined by functoriality.

Inductively, we define, when $G$ is abelian,
$$E(G,m) \coloneqq E(K(G,m-1),1)$$
and $$K(G,m) \coloneqq K(K(G,m-1),1)$$
\end{Def}

\begin{Prop} \label{GActionStack}
For $G$ a group in derived stacks, $BG$ is connected and its homotopy groups are given by
$$\pi_i(BG) \simeq \pi_{i-1}(G)$$

Furthermore, when $A$ is an abelian group in derived stacks, $K(A,n)$ classifies cohomology of derived stacks and $Hom(F,K(A,n))$ will be denoted $H^n(F,A)$. This construction recovers various cohomology theories, such as sheaf cohomology and singular cohomology.
\end{Prop}

\begin{proof}
See \cite{TV06}.
\end{proof}

\paragraph*{Mapping stack}

\begin{Def}
Let $X$ and $Y$ be derived stacks, the mapping stack $\textbf{Map}(X,Y)$ of morphisms from $X$ to $Y$ is defined as a derived stack by
$$\textbf{Map}(X,Y) : B \in SCR \mapsto Map(X \times Spec(B),Y)$$
\end{Def}

\subsection{Postnikov towers}

\begin{Def}
The inclusion of $n$-truncated derived stacks into derived stacks admits a left adjoint denote $t_{\le n}$ and it will be called truncation. We have an induced tower of morphisms
$$F \to ... t_{\le n}(F) \to t_{\le n-1}(F) \to ... \to t_{\le 0}(F)$$
The induced morphism $F \to lim_n t_{\le n}(F)$ is not necessarily an equivalence.
\end{Def}

\subsection{Functions on a derived stack}

\begin{Prop} \label{pushMonoidal}
Let $\pi : X \to Y$ be a morphism of derived stacks. The pushforward of quasi-coherent sheaves 
$$\pi_* : QCoh(X) \to QCoh(Y)$$
is canonically lax monoidal.
\end{Prop}

\begin{proof}
Its left adjoint $\pi^*$, the pullback of quasi-coherent sheaves, is canonically symmetric monoidal, see \cite[Corollary 7.3.2.7]{Lur17}.
\end{proof}

\begin{Def}
We recall the definition of the $\mathbb{E}_\infty$-algebra of functions on a morphism of derived stack $\pi : X \to Y$. Since the structural sheaf $\mathcal{O}_F$ of $F$ is an $\mathbb{E}_\infty$-algebra over $F$, meaning an object of $CAlg(QCoh(F))$, $\pi_*(\mathcal{O}_F)$ has a canonical $\mathbb{E}_\infty$-algebra structure, see Proposition \ref{pushMonoidal}. We denote this quasi-coherent algebra $C_{\mathbb{E}_\infty}(F)$ or simply $C(F)$ when there is no ambiguity.

The pushforward functor also induces a morphism on $\mathbb{E}_1$-algebras :
$$\mathbb{E}_1 Alg(QCoh(X)) \to \mathbb{E}_1 Alg(QCoh(Y))$$
It sends $\mathcal{O}_X$ to an $\mathbb{E}_1$-algebra on $Y$ denoted $C_{\mathbb{E}_1}(X)$.
\end{Def}

\subsection{Affine stacks}

Following \cite{To06}, we recall the definition and basic properties of affine stacks.

\begin{Def}
Let $A$ a cosimplicial algebra, the affine stack associated to $A$ is the stack
$$Spec^\Delta(A) : Aff \to SSet$$
sending $B$ to the simplicial set $Hom_{\textbf{CRing}_k}(A,B_\bullet)$

We can see $Spec^\Delta(A)$ as a derived stack by Kan extension. We define the category of affine stacks as the essential image of $Spec^\Delta$, we denote it $AffSt$.
\end{Def}

\begin{Def}
The cosimplicial algebra of functions on a stack $F$ is defined as
$$C^\bullet_{\Delta}(F) = Hom(F_\bullet,\mathbb{G}_a)$$
where we use the set of morphisms from $F_n$ to $\mathbb{G}_a$ seen as presheaves.
\end{Def}

\begin{Prop}
The previous two functors give an adjunction 
$$ C^\bullet_{\Delta} : St \rightleftarrows coSCR : Spec^{\Delta} $$
and $Spec^\Delta$ is fully faithful.
\end{Prop}

\begin{proof}
See \cite[Corollary 2.2.3]{To06}.
\end{proof}

\begin{Prop}
The cosimplicial functor $C_\Delta(-)$ is an exhancement of the $\mathbb{E}_\infty$-algebra functor $C_{\mathbb{E}_\infty}(-)$, meaning that $C_{\mathbb{E}_\infty}(-)$ factors as the composition of taking the normalization $\theta^{ccn}$ of the cosimplicial functions $C_\Delta(-)$.
\end{Prop}

\begin{Prop}
The category of affine stacks is stable by small limits in stacks.
\end{Prop}

\begin{proof}
See \cite[Proposition 2.2.7]{To06}
\end{proof}

\begin{Def}
We say a derived stack is an affine stack if it lives in the essential image of $Spec^\Delta$. We denote this image $AffSt_k$.

We define the affinization of a stack $F$ as $Spec^\Delta(C_\Delta(F))$.
\end{Def}

\begin{Def}
A morphism of derived stacks $X \to Y$ is said to be a relative affine stack when the pullback functor
$$dSt_{/Y} \to dSt_{/X}$$
sends affine stacks to affine stacks.
\end{Def}

\begin{Prop}
Let $X$ be a space, ie a simplicial set, its stack affinization is given by
$$AffSt(X) = Spec^\Delta k^X$$
where $k^X \simeq C^*(X)$ is the cosimplicial algebra of cohomology of $X$.
\end{Prop}

\begin{Prop}
Let $i$ be a positive integer, the stack $K(\mathbb{G}_a,n)$ is an affine stack and we have
$$K(\mathbb{G}_a,n) \simeq Spec^\Delta(D(k[n]))$$
\end{Prop}

\begin{proof}
See \cite[Lemma 2.2.5]{To06}
\end{proof}

\subsection{Cotangent complex and deformation theory}

\begin{Def}
Let $X$ be a derived stack, $X$ admits cotangent complex at $x : Spec A \rightarrow X$ if there exists $f^* \mathbb{L}_X \in A-Mod$ representing the functor
$$Hom_{Spec A/dSt}(Spec(A \oplus -), X)$$

explicitly,
$$Hom_{Spec A/dSt}(Spec(A \oplus M), X) = fib(X(A \oplus M) \rightarrow X(A))$$
is given by the fiber over $x$.

We say $X$ has global cotangent complex if it has a cotangent complex at any point and for every two points $x \in X(A)$ and $y \in X(B)$ and
$u : A \rightarrow B$ compatible with $x$ and $y$, we have an equivalence
$$u^*x^* \mathbb{L}_X \simeq y^* \mathbb{L}_X$$
\end{Def}

\begin{Prop}
When $X$ is a derived Artin stack locally of finite presentation, its cotangent complex $\mathbb{L}_X$ is perfect, see \cite[Corollary 2.2.5.3]{TV06}. We denote $\mathbb{T}_X$ its dual in $QCoh(X)$. Then its shifted tangent stacks $TX[-n] = \mathbb{V}(\mathbb{L}_X[n])$ are also derived Artin stack locally of finite presentation.
\end{Prop}

\begin{Def}
Let $X$ be a derived stack, $T=Spec(A)$ a derived affine scheme and $x : T \to X$ a $T$-point. Let $M$ be an $A$-module, we write the square zero extension $A \oplus M$ and the associated derived affine scheme $T[M]$. The set of derivations from $T$ to $M$ at $x$ is
$$Def_X(T,M) \coloneqq Map_{T/dSt}(T[M],X)$$
\end{Def}

\begin{Rq}
When it exists, the cotangent complex $\mathbb{L}_{X,x}$ of $X$ at $x$ represents the functor $Der_X(T,-)$
\end{Rq}

\begin{Def}
We define the cotangent stack of $X$ as the linear stack $T^*X[n]=\mathbb{V}(\mathbb{L}_X[-n])$, for an integer $n$.
\end{Def}

\begin{Def}
Let $A$ be a simplicial algebra over $k$, $M$ an $A$-module and
$$d : A \to A \oplus M[1]$$ a derivation, we define the square zero extension induced by $d$ denoted $A \oplus_d M$ by
$$\begin{tikzcd}
A \oplus_d M \arrow[dr, phantom, "\lrcorner", very near start] \arrow[r,"p"] \arrow[d] 
& A  \arrow[d,"d"] \\
A \arrow[r,"s"]
& A
\end{tikzcd}$$
where $s$ is the "zero" section. The morphism $p : A \oplus_d M \to A$ will be called the natural projection.
\end{Def}

\begin{Prop}
Let $X$ be a derived stack which has an obstruction theory, see \cite[Definition 1.4.2.2]{TV06}. Let $A$ be a simplicial algebra, $M$ a simplicial $A$-module, $d \in Der(A,M[1])$ a derivation with $A \oplus_d M$ the corresponding square zero extension. We denote
$$T \coloneqq Spec(A) \to T[M] \coloneqq Spec(A \oplus_d M)$$
the morphism corresponding to the projection $A \oplus_d M \to A$. Let $x : T \to X$ be an $A$-point.

\begin{itemize}
\item There exists a natural obstruction
$$\alpha(x) \in Map_{A-Mod}(\mathbb{L}_{X,x},M[1])$$
vanishing if and only if $x$ extends to a morphism $x' : T[M] \to X$.

\item If we assume $\alpha(x)=0$, then the space of lifts of $x$
$$Map_{X/dSt}(T[M],X)$$
is non canonically isomorphic to
$$Map_{A-Mod}(\mathbb{L}_{X,x},M)$$
\end{itemize}
\end{Prop}

\begin{proof}
See \cite[Proposition 1.4.2.5]{TV06}.
\end{proof}

We will need the relative version of the previous proposition, so as to apply obstruction theory to derived stacks endowed with a grading.

\begin{Prop}
Let $f : X \to Y$ be a morphism of derived stack which has obstruction theory. Let $A \in SCR$, $M \in A-Mod^{cn}$, $d \in Der(A,M[1])$ a derivation with $A \oplus_d M$ the corresponding square zero extension. We denote
$$T \coloneqq Spec(A) \to T[M] \coloneqq Spec(A \oplus_d M)$$
the morphism corresponding to the projection $A \oplus_d M \to A$. Let $x : T \to X$ be a point in $Map(T,X) \times_{Map(T[M],Y)} Map(T,Y)$.
The fiber at $x$ of the morphism
$$Map(T[M],X) \to Map(T,X) \times_{Map(T[M],Y)} Map(T,Y)$$
is denoted $L(x)$.

There is a natural point $\alpha(x) \in Map_{A-Mod}(\mathbb{L}_{X/Y,x},M[1])$ and a natural equivalence
$$ L(x) \simeq \Omega_{\alpha(x),0} Map_{A-Mod}(\mathbb{L}_{X/Y,x},M[1])$$
\end{Prop}

\paragraph*{The tangent bundle formalism}
\label{tangentFormalism}

\begin{Def} 
Following \cite{Lur07}, let $\mathcal{C}$ be a presentable $\infty$-category, the tangent bundle of $\mathcal{C}$ is a functor $T_\mathcal{C} \to \mathcal{C}^{\Delta^1}$ which exhibits $T_\mathcal{C}$ as the stable envelope of $\mathcal{C}^{\Delta^1} \to \mathcal{C}$, the evaluation at $\{1\} \subset \Delta^1$.

We can think of an element of $T_\mathcal{C}$ as a pair $(A,M)$ where $A \in \mathcal{C}$ and $M$ an infinite loop space in $\mathcal{C}$. 
\end{Def}

\begin{Prop}
Let $A$ be a simplicial algebra, we denote $Stab(SCR_{/A})$ the stable category constructed as the stable enveloppe of $SCR_{A/}$. The functor sending a non necessarily connected module $M$ over $A$ to the square zero extension simplicial algebra $A \oplus M$, with its natural augmentation defines an equivalence :
$$A-Mod \xrightarrow{\sim} Stab(SCR_{/A})$$
\end{Prop}

\begin{Prop}
The cotangent complex formalism from \cite{Lur07} defines a functor
$$SCR^{\Delta^1} \to T_{SCR}$$
which is identifies under the equivalence of the previous proposition with the cotangent complex formation $A \to B \mapsto \mathbb{L}_{B/A}$.

We deduce by composition the absolute complex functor
$$SCR \to T_{SCR}$$
which is left adjoint to the forgetful functor $T_{SCR} \to SCR$ sending $(A,M)$ to $A$.
\end{Prop}

\begin{proof}
See \cite[Remark 1.2.3]{Lur07}.
\end{proof}

\begin{Rq}
This description of the cotangent complex makes it easier to understand and prove the usual formulas of base change for the cotangent complex.
\end{Rq}

\subsection{Group actions}

\begin{Def}
Let $G$ be a group in derived stacks, the trivial action functor :
$$dSt \to G-dSt$$
admits a left adjoint denoted $(-)_{G}$ and a right adjoint denoted $(-)^G$, called (homotopy) coinvariants and (homotopy) fixed points.
\end{Def}

\begin{Def} 
In the case $G = \mathbb{G}_m$, $\mathbb{G}_m-dSt$ is the category of graded derived stacks, then taking fixed points correspond to taking the $0$-weighted associated derived stack.
\end{Def}

\begin{Prop} \label{Weight0}
Taking fixed points $(-)^G$ of a derived stack endowed with a $G$ action and taking weight $0$ of a graded derived stack preserve small limits.
\end{Prop}

\begin{Prop} \label{ActionBG}
For $G$ a group in derived stacks, we have a natural equivalence :
$$G-dSt_k \simeq dSt_{k/BG}$$
\end{Prop}

\begin{proof}
See \cite[Proposition 1.3.5.3]{TV06}.
\end{proof}

\subsection{Graded and filtered objects}

In this section, we review basic definitions and properties of graded and filtered objects in $\infty$-categories, following \cite[\S 3.1, \S 3.2]{Lur15}
\begin{Def}
Let $\mathcal{C}$ be an $\infty$-category, we define the category of graded objects in $\mathcal{C}$ :
$$\mathcal{C}^{gr} \coloneqq Fun(\mathbb{\textbf{Z}}^{disc},\mathcal{C}) \simeq \prod_{\mathbb{Z}} \mathcal{C}$$
where $\mathbb{\textbf{Z}}^{disc}$ is the discrete $1$-category on $\mathbb{Z}$. We also define the category of filtered objects in 
$\mathcal{C}$ :
$$Fil(\mathcal{C}) \coloneqq Fun(\bf{(\mathbb{\textbf{Z}},\le)},\mathcal{C})$$
where $(\mathbb{\textbf{Z}},\le)$ is the $1$-category associated to the usual partial ordered structure on $\mathbb{Z}$.
The inclusion $\mathbb{\textbf{Z}}^{disc} \subset (\mathbb{\textbf{Z}},\le)$ induces a forgetful functor : $$ Fil(\mathcal{C}) \to \mathcal{C}^{gr}$$
\end{Def}

\begin{Rq}
In our definition, the structure morphisms $X_n \to X_{n+1}$ of a filtered object need not be monomorphisms. However the two notions are shown to be equivalent in \cite{SS13}.
\end{Rq}

\begin{Def}
When $\mathcal{C}$ is a symmetric monoidal $\infty$-category, $\mathcal{C}^{gr}$ and $Fil(\mathcal{C})$ are both symmetric monoidal $\infty$-category when endowed with the Day product.
\end{Def}

\begin{Def}
The associated graded and underlying object functors are given by symmetric monoidal functors :
$$gr : Fil(\mathcal{C}) \to \mathcal{C}^{gr}$$
and
$$colim : Fil(\mathcal{C}) \to \mathcal{C}$$
See \cite{Lur15} for details on these construction.
\end{Def}

We now recall the geometrical interpretation of graded modules and filtered modules. See \cite{Sim96} and \cite{Mou19} for details.

\begin{Prop}
There is an equivalences of stable symmetric monoidal $\infty$-categories :
$$QCoh(\mathbb{A}^1/\mathbb{G}_m) \simeq Fil(Mod_k)$$
and
$$QCoh(B\mathbb{G}_m) \simeq Mod_k^{gr}$$
when $\mathbb{A}^1$ and $\mathbb{G}_m$ are the usual additive and multiplicative groups in derived stacks over $k$.

Furthermore, pullback along the closed point $0 : B\mathbb{G}_m \to \mathbb{A}^1/\mathbb{G}_m$ and the open point $1 : * \simeq \mathbb{G}_m/\mathbb{G}_m \to \mathbb{A}^1/\mathbb{G}_m$ recover respectively the associated graded and the underlying object constructions.
\end{Prop}

\begin{proof}
See \cite[Theorem 1.1, Theorem 4.1]{Mou19} and \cite[Theorem 2.2.10]{MRT20}.
\end{proof}

\subsection{Graded stacks}

\begin{Def} [Graded derived stack]
A graded derived stack is defined as a derived stack endowed with an action of $\mathbb{G}_m$, a morphism of graded derived stack is one compatible with the actions. The category of derived affine schemes endowed with a $\mathbb{G}_m$-action is naturally equivalent to the category of graded simplicial algebras over $k$, see \cite{Mou19}. The category of graded derived stacks is denoted $\mathbb{G}_m$-dSt.
\end{Def}

\begin{Rq}
A $\mathbb{G}_m$-action on a derived stack $X$ is equivalent to a morphism of derived stacks $Y \to B\mathbb{G}_m$ and an identification of $Y \times_{B\mathbb{G}_m} *$ with $X$, using Proposition \ref{ActionBG}. We call $X$ the total space associated to $Y \to B\mathbb{G}_m$.
\end{Rq}

\begin{Thm}
From \cite{Mou19}, there is a symmetric monoidal equivalence of categories
$$Mod_B^{gr} \simeq QCoh(B\mathbb{G}_m \times Spec(B))$$

We will call graded quasi-coherent complex over $X$ an object of either category.
\end{Thm}

\begin{Def}
For $E$ a quasi-coherent complex of $X$, and $n$ an integer, $E((n))$ is defined to be the graded quasi-coherent complex with $E$ in pure weight $n$.
\end{Def}

\begin{Def} [Relative global functions on graded derived stacks]
Given a morphism of graded stacks $Y \to X$, we define graded functions on $Y$ relative to $X$ as
$$\mathcal{O}_{gr,X}(Y) \coloneqq \pi_*(\mathcal{O}_{[Y/\mathbb{G}_m]}) \in CAlg(QCoh([X/\mathbb{G}_m]))$$
with $\pi$ the canonical morphism $[Y/\mathbb{G}_m] \to [X/\mathbb{G}_m]$.
\end{Def}

\begin{Def}
We introduce graded $\mathcal{O}$-module function similarly :
$$\widetilde{\mathcal{O}}_{gr,X}(Y) \coloneqq \widetilde{\pi}_*(\mathcal{O}_{[Y/\mathbb{G}_m]}) \in CAlg(\mathcal{O}_{[X/\mathbb{G}_m]}-Mod)$$

using the pushforward $\widetilde{\pi}_* : \mathcal{O}_{[Y/\mathbb{G}_m]}-Mod \to \mathcal{O}_{[X/\mathbb{G}_m]}-Mod$. It is lax monoidal, as left adjoint of a symmetric monoidal functor, therefore it induces a canonical morphism on commutative algebras in the respective symmetric monoidal categories. We deduce that quasi-coherent graded functions are simply $\mathcal{O}$-module graded functions after applying the left adjoint of the forgetful functor from quasi-coherent complexes to $\mathcal{O}$-modules.
\end{Def}

\begin{Def} [Graded stacks over a base]
Let $X$ be a graded derived stack, the $\infty$-category of graded derived stacks over $X$ is the undercategory $\mathbb{G}_m-dSt_{/X}$, with the trivial grading on $X$.
\end{Def}

\begin{Ex}
Linear stacks are naturally $\mathbb{G}_m$-graded, we can construct the action
$$\mathbb{G}_m \times \mathbb{V}(E) \to \mathbb{V}(E)$$
 as follows. For a given simplicial commutative $k$-algebra $B$, $\mathbb{G}_m(B)$ is $Map^{eq}_{N(B)-mod}(N(B),N(B))$, it naturally operates on
 $$\mathbb{V}(E)(B) = Map_{N(B)-Mod}(E,N(B))$$ through the composition map
$$Map_{N(B)-Mod}(N(B),N(B)) \times Map_{N(B)-Mod}(E, N(B)) \to Map_{N(B)-Mod}(E,N(B))$$
\end{Ex}

\begin{Def} \label{DefWeight0dSt} Let $X$ be a graded derived stack, the associated derived stack of $0$-weights is defined as the derived stack of fixed points by the action of $\mathbb{G}_m$, we denote it
$$X^{=0} \coloneqq X^{\mathbb{G}_m}$$
\end{Def}

\begin{Prop} \label{DefWeight0SCR}
The functor
$$SCR \to SCR^{gr}$$
endowing a simplicial algebra with its trivial grading admits a left adjoint, which we denote $(-)^{=0}$.
\end{Prop}

\begin{proof}
Using the adjoint functor theorem, see \cite[Corollary 5.5.2.9]{Lur09}.
\end{proof}

\begin{Rq}
For a general graded simplicial algebra $A$, $(A)^{=0}$ and $A(0)$ will not coincide. We will call the former "$0$-weights" and the latter "naive $0$-weights".
\end{Rq}

\begin{Rq}
When $A$ is a discrete commutative algebra, we notice that $(A)^{=0}$ is given by the quotient ring of $A(0)$ where we quotient by elements of the form $ab$ where $a$ is in weight $n>0$ and $b$ in weight $-n$.
\end{Rq}

\begin{Rq} \label{WeightNaiveVsNormal}
When $A$ is a graded simplicial algebra, either positively graded or negatively graded, then $(A)^{=0}$ coincides with $A(0)$.
\end{Rq}

\begin{Rq} \label{CompatWeight0}
The two constructions of $0$-weights are compatible. Meaning that for $A$ a graded simplicial algebra, we have a natural identification
$$(Spec(A))^{=0} \simeq Spec(A^{=0})$$

We verify it on a test simplicial algebra $R$
$$(Spec(A))^{=0}(R) = Map_{\mathbb{G}_m}(Spec(R),Spec(A)) \simeq Map_{gr}(A,R) \simeq Map(A^{=0},R)$$
\end{Rq}

\subsection{Line bundles and graded modules}

In this section we discuss the relationship between $\mathbb{G}_m$-torsors and invertible modules.

\begin{Def}
If $B$ is a simplicial $k$-algebra, a morphism $Spec(B) \to B \mathbb{G}_m$ is the data of an invertible $B$-module, we will denote both the morphism and the module by $\mathcal{L}$. There is an ambiguity in the choice, the invertible module associated to the morphism could be defined to be $\mathcal{L}^\vee$, we choose $\mathcal{L}$ in such a way that the total space of the $\mathbb{G}_m$-torsor associated to $Spec(B) \to B \mathbb{G}_m$ is $Spec((Sym_B \mathcal{L}^\vee)[(\mathcal{L}^\vee)^{-1}])$, it is the linear stack associated to $\mathcal{L}^\vee$ where we removed the zero section, as a $B$-module. The function algebra $(Sym_B \mathcal{L}^\vee)[(\mathcal{L}^\vee)^{-1}]$ is simply $\bigoplus_{n \in \mathbb{Z}} (\mathcal{L}^\vee)^{\otimes n}$. With this convention $\mathcal{L}$ is functorial in the morphism $Spec(B) \to B \mathbb{G}_m$.
\end{Def}

\begin{Rq}
We can try and understand $[\mathbb{V}(E)/\mathbb{G}_m]$ over $B \mathbb{G}_m \times X$ through its funtor of points. A morphism $Y = Spec(B) \to [\mathbb{V}(E)/\mathbb{G}_m]$ over $B\mathbb{G}_m \times X$ is the data of a morphism $Y \to B \mathbb{G}_m \times X$, ie a $\mathbb{G}_m$-torsor $\widetilde{Y} \to Y$ with a map $Y \to X$, and a $\mathbb{G}_m$-equivariant morphism $\widetilde{Y} \to \mathbb{V}(E)$ over $X$. The torsor has total space $\widetilde{Y} = Spec((Sym_B \mathcal{L}^\vee)[(\mathcal{L}^\vee)^{-1}])$, therefore the $\mathbb{G}_m$-equivariant map $\widetilde{Y} \to \mathbb{V}(E)$ corresponds to a morphism of graded $N(B)$-modules $u^*E((1)) \to (Sym_B \mathcal{L}^\vee)[(\mathcal{L}^\vee)^{-1}]$ ie a morphism of $N(B)$-module $u^*E \to \mathcal{L}^\vee$, with $u^*$ the pullback morphism induced by the morphism $u : A \to B$ associated with $Y \to X$. This is equivalent to having a morphism of $N(B)$-modules $\mathcal{L} \otimes u^*E  \to N(B)$, which is a $B$-point of $\mathbb{V}(\mathcal{L} \otimes u^*E)$.
\end{Rq}

\begin{Rq} \label{EquivGraduée}
We deduce an equivalence of $B$-modules \[ [\mathbb{V}(E)/\mathbb{G}_m] \times_{B\mathbb{G}_m \times X} Spec(B) \simeq \mathbb{V}(\mathcal{L} \otimes u^*E). \]

The equivalence between graded complexes and quasi-coherent modules on $B \mathbb{G}_m$ is given informally by sending a graded complex $\bigoplus_i E_i$ to the $\mathcal{O}$-module stack
$$Spec(B)\xrightarrow{\mathcal{L}} B \mathbb{G}_m \mapsto \bigoplus_i E_i \otimes \mathcal{L}^{\otimes i}$$

Therefore, for any integer $k$, $E((k))$ is the $\mathcal{O}$-module 
$$Spec(B)\xrightarrow{(\mathcal{L},u)} B \mathbb{G}_m \times X \mapsto \mathcal{L}^{\otimes k} \otimes u^*E$$
\end{Rq}

\subsection{The graded circle}

\begin{Def}
Let $A \coloneqq \mathbb{Z}[\frac{x^n}{n!}] \subset \mathbb{Q}[x]$, we define $Ker$ the functor represented by $A$ pulled back to $k$.
\end{Def}

\begin{Prop}
The presheaf $Ker$ coincides with the intersection of the kernels of the Frobenius morphisms on big Witt vectors
$$F_p : W_{big} \to W_{big}$$
therefore it has a natural abelian group structure. See \cite{Haz08} and \cite{Rab14} for the construction of big Witt vectors.
\end{Prop}

\begin{Def} \label{GradedCircle}
We define the graded circle $S^1_{gr}$ as
$$S^1_{gr} \coloneqq BKer$$
\end{Def}

\begin{Prop} \label{BS1affine}
Recalling from \cite[Proposition 3.2.7]{MRT20}, the stacks $S^1_{gr}$ and $BS^1_{gr}$ are affine stacks.
\end{Prop}

\begin{Le} 
Let $M$ be $\bigoplus_{m \ge 0}  k[-2m]((m)) \in Mod^{gr}$, the space of $\mathbb{E}_\infty$-algebra structures on $M$ compatible with the grading is equivalent to the set of classical commutative graded algebra structures on its cohomology $H^*(M)$.
\end{Le}

\begin{proof}
We follow the proof of \cite[Lemma 3.4.11]{MRT20}, the space of $\mathbb{E}_\infty$-algebra structures considered is given by the mapping space of $\infty$-operads :
$$Map_{Op}(\mathbb{E}_\infty^\otimes,End_{gr}(M)^\otimes)$$
where $\mathbb{E}_\infty^\otimes$ is the $\mathbb{E}_\infty$ operad and $End_{gr}(M)^\otimes$ is the endomorphism operad of $M$ in the category of graded modules. 

The space of $n$-ary operations of $End_{gr}(M)^\otimes$ is given by $Map_{Mod^{gr}}(M^{\otimes n}, M)$.

We compute
$$M^{\otimes n} = \left(\bigoplus_{m \ge 0}  k[-2m]((m)) \right)^{\otimes n} = \bigoplus_{m \ge 0}  k[-2m]((m))^{\oplus c_{m,n}}$$
with $c_{m,n} = \binom{n+m-1}{n-1}$ is the cardinal of $\{ (m_1, m_2, ... , m_n) \in \mathbb{N}^n, \sum_{i=1}^n m_i = m \}$, since the differential is zero, $Map_{Mod^{gr}}(M^{\otimes n}, M)$ is discrete.

We deduce that the space $Map_{Op}(\mathbb{E}_\infty^\otimes,End_{gr}(M)^\otimes)$ is discrete.

Since $H^*$ is lax monoidal, we get a map :
$$Map_{Op}(\mathbb{E}_\infty^\otimes,End_{gr}(M)^\otimes) \to Map_{Op}(\mathbb{E}_\infty^\otimes,End_{gr,cl}(H^*(M))^\otimes)$$
where $End_{gr,cl}(H^*(M))^\otimes)$ is the classical operad of graded endomorphisms on $M$. This map is an equivalence of spaces.
\end{proof}

\begin{Cor} \label{kuEinf}
The graded $\mathbb{E}_1$-algebra $k[u]$, with $u$ in degree $2$ and weight $1$, admits a unique $\mathbb{E}_\infty$-algebra structure.
\end{Cor}

\begin{Prop}
The cosimplicial functions of $BS^1_{gr}$ are given by
$$C_{\Delta}(BS^1_{gr}) \simeq k[u]$$
and the equivalence is compatible with the augmentation, given by $BS^1_{gr} \to *$.
Here $k[u]$ denoted the denormalization of the free commutative differential graded algebra on one generator in degree $2$ and weight $1$. The identification is compatible with gradings.

\end{Prop}

\begin{proof}
We start by the construction of a morphism $k[u] \to C_{\Delta}(BS^1_{gr})$.

Since
$$H^2(BS^1_{gr}) \simeq H^2(Tot(k \to k[\eta] \rightrightarrows k[\eta]^{\otimes 2} ...))$$
the element $\eta$ defines an element $u \in H^2(BS^1_{gr})$ in weight $1$. We deduce a morphism of $\mathbb{E}_1$-algebra $k[u] \to C_{\mathbb{E}_1}(BS^1_{gr})$ as $k[u]$ is the free $\mathbb{E}_1$-algebra on $k[-2]$. Now since the $\mathbb{E}_1$-algebra $k[u]$ admits a unique $\mathbb{E}_\infty$-algebra structure, from Corollary \ref{kuEinf}, we get a morphism of $\mathbb{E}_\infty$-algebra
$$k[u] \to C_{\mathbb{E}_\infty}(BS^1_{gr})$$
which is an equivalence after taking cohomology, therefore it is an equivalence. We write it as

$$\phi : \theta^{ccn}(k[u]) \xrightarrow{\sim} \theta^{ccn}(C_\Delta(BS^1_{gr}))$$

Now we prove it is an equivalence, following the proof of \cite[Theorem 3.4.17]{MRT20}, we apply the double conerve functor to $\phi$ :
$$coN^2(\theta^{ccn}(k[u])) \simeq coN^2(\theta^{ccn}(C_\Delta(BS^1_{gr})))$$
which is a morphism of bisimplicial $\mathbb{E}_\infty$-algebra. By definition $S^1_{gr} \coloneqq BKer$, therefore we deduce an equivalence
$$ coN^2(\theta^{ccn}(C_\Delta(BS^1_{gr}))) \simeq \theta^{ccn}(C_\Delta(Ker))^{\bullet \bullet}$$

The stack $Ker$ being a classical affine scheme, $\theta^{ccn}(C_\Delta(Ker))^{\bullet \bullet}$ is a bisimplicial object in discrete $\mathbb{E}_\infty$ algebras. As $\theta^{ccn} : coSCR \to \mathbb{E}_\infty-Alg$ restricts to an equivalence on discrete objects, the equivalence
$$\theta^{ccn}(coN^2(k[u])) \simeq \theta^{ccn}(C_\Delta(Ker))^{\bullet \bullet}$$
lifts uniquely to an equivalence of bisimplicial cosimplicial algebras
$$coN^2(k[u]) \simeq C_\Delta(Ker)^{\bullet \bullet}$$

Now totalization of the bisimplicial objects gives the required equivalence of cosimplicial algebras :
$$k[u] \simeq C_\Delta(BS^1_{gr})$$
\end{proof}

\begin{Cor} \label{BS1grequi}
The affine stack $BS^1_{gr}$ is given by
$$BS^1_{gr} \simeq Spec^\Delta(k[u])$$
where the identification is compatible with augmentations and gradings.
\end{Cor}

\begin{Prop} \label{grMixedEqui}
There is an equivalence of symmetric monoidal categories compatible with grading, ie which commutes with the forgetful functors to graded complexes
$$\epsilon-Mod_k \simeq Rep(S^1_{gr}) = QCoh(BS^1_{gr})$$
\end{Prop}

\begin{proof}
See \cite[Proposition 4.2.3]{MRT20}.
\end{proof}

\begin{Def}
We define the category of graded mixed derived stacks as the category of $S^1_{gr}$-equivariant graded stacks
$$\epsilon-dSt^{gr} \coloneqq S^1_{gr}-dSt^{gr}$$
\end{Def}

\begin{Rq} \label{S1grH}
Seeing $S^1_{gr}$ as a derived stack with a $\mathbb{G}_m$-action, it admits a canonical morphism
$$S^1_{gr} \to B\mathbb{G}_m$$
which has as a total space the semi-direct product
$$\mathcal{H} \coloneqq \mathbb{G}_m \ltimes S^1_{gr}$$
Therefore we can see a graded mixed structure on a derived stack as a $\mathbb{G}_m$ action and an action of $S^1_{gr}$ which are compatible. Meaning that we have an identification
$$\epsilon-dSt^{gr} \simeq \mathcal{H}-dSt$$

Similarly $QCoh(BS^1_{gr})$ is defined with $S^1_{gr}$ considered as a element of $dSt^{gr}$, that is $QCoh(BS^1_{gr})$ is implicitly defined as $QCoh(B\mathcal{H})$.
\end{Rq}

\section{Graded function on linear stacks}
This section is devoted to the computation of $1$-weighted graded functions on a linear stack. This section is available online at \cite{Mon21}.

\subsection{Adjoint of the linear stack functor}

\begin{Prop} \label{AdjLinStack}
The functor
$$\mathbb{V} : QCoh(X)^{op} \to \mathbb{G}_m-dSt_{/X}$$
has a left adjoint given by $\mathcal{O}_{gr}(-)(1)$.
\end{Prop}

\begin{proof}
Let us consider the two functors
$$Map_{\mathbb{G}_m-dSt_{/X}}(-,\mathbb{V}(E)), Map_{QCoh(X)}(E,\mathcal{O}_{gr}(-)(1)) : \mathbb{G}_m-dSt_{/X}^{op} \to \mathcal{S}$$
landing in the category of spaces. Both functors send all small colimits to limits and are canonically identified on the subcategory of $\mathbb{G}_m$-equivariant derived stack of the form $Spec(B) \times \mathbb{G}_m$ with $\mathbb{G}_m$ acting by multiplication on the right side : any $\mathbb{G}_m$-stack $Y$ can be recovered as a colimit of objects of the form $Y \times \mathbb{G}_m$ with the action being multiplication on the right side since $Y$ is canonically the colimit of a diagram on $B\mathbb{G}_m$ sending the point to $Y \times \mathbb{G}_m$ and sending elements of $\mathbb{G}_m$ to their diagonal action on $Y \times \mathbb{G}_m$. Therefore the two functors are equivalent. The equivalence being functorial in $E$, we deduce the adjunction.

\end{proof}

\begin{Rq}
\label{Adj}
In the adjunction above, if we restrict to $\mathbb{G}_m$-equivariant derived stack of the form $Spec(B) \times \mathbb{G}_m$ with $\mathbb{G}_m$ acting by multiplication on the right side, we deduce another adjunction :
$$Map_{dSt}(Y,\mathbb{V}(E)) \simeq Map_{QCoh(X)}(E,\mathcal{O}(Y)).$$

\end{Rq}

\subsection{The equivalence}

\begin{Cons}
\label{Cons}
We now construct a natural arrow $E((1)) \to \mathcal{O}_{gr}(\mathbb{V}(E))$ in $QCoh([X/\mathbb{G}_m])$, ie a map of graded complexes on $X$ a derived stack. By adjunction, it just means we have to construct a map of graded $\mathcal{O}$-modules $E((1)) \to \mathcal{O}_{gr,\mathcal{O}}(\mathbb{V}(E))$. For the construction, we go back and forth between quasi-coherent complexes and $\mathcal{O}$-modules, exploiting the facts that quasi-coherent complexes have better categorical properties and $\mathcal{O}$-modules are more computable.

We assume $E$ to be cofibrant, we take $B$ a simplicial commutative $k$-algebra and $u$ a morphism from $Spec(B)$ to $B\mathbb{G}_m\times X$. We want to construct a map
$$E((1))(B) \simeq \mathcal{L} \otimes u^*E \to \mathcal{O}_{gr,\mathcal{O}}(\mathbb{V}(E))(B).$$

However, using Remark \ref{EquivGraduée}, we have :
$$\mathcal{O}_{gr,\mathcal{O}}(\mathbb{V}(E)))(B) \simeq \mathcal{O}_{\mathcal{O}}([\mathbb{V}(E)/\mathbb{G}_m] \times_{B\mathbb{G}_m \times X} Spec(B)) \simeq \mathcal{O}_{\mathcal{O}}(\mathbb{V}(\mathcal{L} \otimes u^*E)).$$

Therefore we are reduced to constructing a functorial map $\mathcal{L} \otimes u^*E \to \mathcal{O}_{\mathcal{O}}(\mathbb{V}(\mathcal{L} \otimes u^*E))$. By the adjunction property of quasi-coherent sheaves, it is equivalent to constructing a functorial map 
$$\mathcal{L} \otimes u^*E \to \mathcal{O}(\mathbb{V}(\mathcal{L} \otimes u^*E)).$$ We can use the unit of the adjunction in Remark \ref{Adj} to construct this map. From this map we deduce in weight $1$ a map :
$$E \to \mathcal{O}_{gr}(\mathbb{V}(E)))(1)$$

\end{Cons}

\begin{Prop}
If $E$ is connective, the natural map constructed above
$$E \to \mathcal{O}_{gr}(\mathbb{V}(E)))(1)$$
is an equivalence

\end{Prop}

\begin{proof}
Since $E$ is connective, $\mathbb{V}(E)$ is simply the relative affine scheme $Spec_X Sym_{\mathcal{O}_X}(E)$. The grading corresponds to having weight $p$ of $Sym_{\mathcal{O}_X}(E)$ to be $Sym^p_{\mathcal{O}_X}(E)$. The morphism
$$E((1)) \to \mathcal{O}_{gr}(Spec_X Sym_{\mathcal{O}_X}(E((1)))) \simeq Sym_{\mathcal{O}_X}(E((1)))$$ corresponds to the inclusion of $E$ in weight $1$. Therefore it is an equivalence when restricting to weight $1$.

\end{proof}

\begin{Thm} \label{FunctionLinStack}
Let $X$ be a derived affine scheme, for $E \in QCoh^{-}(X)$, ie a bounded above complex over $X$, the natural map constructed above
$$E \to \mathcal{O}_{gr}(\mathbb{V}(E))(1)$$
is an equivalence.
\end{Thm}

\begin{proof}
Let us proceed by induction on the Tor amplitude of $E$. Let $E$ be a complex of Tor amplitude in $]-\infty,b]$, $b$ a positive integer. The case $b=0$ has already been dealt with, so we assume $b>0$.

We know, as $X$ is assumed to be affine, there is a non canonical triangle
$$0 \to V[-b] \to E \to E' \to 0$$
with $V$ of tor amplitude $0$, ie a classical vector bundle on $X$, and $E'$ concentrated of Tor amplitude in $]-\infty,b-1]$. To construct $E'$, we can take a model of $E$ concentrated in degree smaller than $b$ and cellular as an $N(A)$-module, then define $E'$ as its naive truncation in degree smaller than $b-1$. We deduce the morphism $E \to E'$.

Let us introduce the notation $R$ denoting $\mathcal{O}_{gr}(\mathbb{V}(-))(1)$. By naturality of the construction above, we have the following commutative diagram

\[\begin{tikzcd}
E \ar{d} \ar{r}& Tot(coN_\bullet(E \to E')) \ar{d} \\
R(E) \ar{r} & Tot(R(coN_\bullet(E \to E')))
\end{tikzcd}\]

with $coN$ the conerve construction and $Tot$ the totalization of a cosimplicial object. We will show these four maps are all equivalences.

The top arrow is an equivalence as $E$ and $E'$ live in the stable category of quasi-coherent sheaves on $X$, see \cite[Proposition 1.2.4.13]{Lur17}.

We now show that the right arrow is an equivalence, in fact $coN_\bullet(E \to E') \to R(coN_\bullet(E \to E'))$ is a levelwise equivalence. We compule the conerve in the following lemma :

\begin{Le}
For any natural integer $n$, we have a canonical equivalence
$$coN_n(E \to E') \simeq E' \oplus V[-b+1]^{\oplus n}$$
\end{Le}

\begin{proof}
The computation is standard, we proceed by induction on the degree $n$, in zero degree it is obvious. Now $E' \oplus_E E'$ has a canonical split projection to $E'$, its kernel is the fiber of $E' \to K$, which is $V[-b]$. Therefore $E' \oplus_E E'$ is identified with $E' \oplus V[-b+1]$. The general case follows.
\end{proof}

We are reduced to showing that the natural map
$$ E' \oplus V[-b+1]^{\oplus n} \to R(E' \oplus V[-b+1]^{\oplus n})$$
is an equivalence. Since $E'$ and $V[-b+1]$ are in Tor amplitude in $]-\infty,b-1]$, it follows from the induction hypothesis.

We now show the bottom arrow is an equivalence. We will use a lemma to compute $\mathbb{V}(E)$ :

\begin{Le} $\mathbb{V}(E') \to \mathbb{V}(E)$ is an epimorphism of derived stacks.
\end{Le}

\begin{proof}
Let us fix $B$ is a simplicial commutative $k$-algebra and $u$ a $B$-point of $X$. We want to show that $\mathbb{V}(E')(B) \to \mathbb{V}(E)(B)$ is surjective after applying $\pi_0$. Let us fix $\alpha \in \mathbb{V}(E)(B)$ ie $\alpha : u^* E \to N(B)$ a morphism of $B$-module.

Since there is an exact triangle
$$0 \to V[-b] \to E \to E' \to 0$$
therefore an exact triangle
$$0 \to u^*V[-b] \to u^*E \to u^*E' \to 0$$

we deduce that $\alpha$ lifts to $u^*E' \to N(B)$ if and only if composition $u^*V[-b] \to N(B)$ is zero. Since $u^*V[-b]$ is strictly coconnective and $u^*N(B)$ is connective, the lift exists.

This concludes that $\mathbb{V}(E') \to \mathbb{V}(E)$ is an epimorphism of presheaves, hence is an epimorphism of derived stacks.

\end{proof}

\begin{Rq}
\label{RqEpi}
The proof of the above gives a slightly stronger result, indeed we have shown that the morphism
$$\mathbb{V}(E') \to \mathbb{V}(E)$$
is a epimorphism of functors.
\end{Rq}

Using this lemma and the connections between effective epimorphisms in an $\infty$-topos and simplicial resolutions of \cite[Corollary 6.2.3.5]{Lur09}, we have an equivalence

$$ | N(\mathbb{V}(E') \to \mathbb{V}(E)) | \xrightarrow{\sim} \mathbb{V}(E)$$
with $|-|$ the geometric realization. $N$ is the nerve construction. 

Since $\mathbb{V}$ and $\mathcal{O}_{gr}$ send colimits to limits, between their categories of definition, we deduce
$$\mathcal{O}_{gr}(\mathbb{V}(E)) \xrightarrow{\sim} Tot(\mathcal{O}_{gr}(\mathbb{V}(coN(E \to E'))))$$ 

We can then apply the weight $1$ functor $(-)(1)$, which commutes with limits since limits can be computed levelwise for graded objects. Therefore

$$R(E) \xrightarrow{\sim} Tot(R(coN_\bullet(E \to E'))).$$

We conclude that $E \to R(E)$ is an equivalence.

\end{proof}

\begin{Cor} \label{CoroLinStackFunction}
The functor $\mathbb{V} : QCoh(X)^{op} \to \mathbb{G}_m-dSt_{/X}$ is fully faithful when restricted to $QCoh^-(X)$.
\end{Cor}

\begin{Rq}
These claims are still correct for $X$ a general derived stack. We can check it using a descent argument.
\end{Rq}

\subsection{General weights}

\begin{Prop} \label{FunctionLinStackchar0}
If $k$ is a $\mathbb{Q}$-algebra and $E$ is a perfect complex on $X$, $\mathcal{O}_{gr}(\mathbb{V}(E))(p)$ is canonically identified with $Sym_{\mathbb{E}_\infty}^p(E)$, the monad here $Sym_{\mathbb{E}_\infty}$ is the one associated to the formation of free $\mathbb{E}_\infty$-$k$-algebras, or equivalently commutative differential algebras.
\end{Prop}

Before the proof of the proposition, we will establish a preliminary lemma.

\begin{Le}
If $k$ is a $\mathbb{Q}$-algebra, $E$ is perfect complex on $X$ and $E^\bullet$ is a cosimplicial object in perfect complexes on $X$. An equivalence $E \xrightarrow{\sim} Tot(E^\bullet)$ induces a canonical map $Sym^p(E) \rightarrow Tot(Sym^p(E^\bullet))$, which is an equivalence. Therefore  $Tot(Sym^p(E^\bullet))$ is perfect.
\end{Le}

\begin{proof}
Starting from the equivalence $E \xrightarrow{\sim} Tot(E^\bullet)$, we deduce by taking duals an equivalence of perfect complexes :
$$ | (E^\bullet)^\vee | \xrightarrow{\sim}  E^\vee$$
since dualization induces an equivalence of $\infty$-category on perfect complexes. Knowing $Sym^p$ commutes with sifted colimits, we apply it on both sides
$$ | Sym^p( (E^\bullet)^\vee) |  \xrightarrow{\sim} Sym^p(E^\vee).$$

We can then dualize again and use the fact that $Sym^p$ and dualization commute since we are in zero characteristic.
$$Sym^p(E) \xrightarrow{\sim} Tot(Sym^p(E^\bullet)).$$

And it concludes the proof.
\end{proof}

\begin{proof}
The proof of the proposition follows the same structure as the main theorem. The canonical map
$$E((1)) \to \mathcal{O}_{gr}(\mathbb{V}(E))$$ gives by adjunction a morphism of $\mathbb{E}_{\infty}$-algebras over $X$ :
$$Sym(E((1))) \to \mathcal{O}_{gr}(\mathbb{V}(E))$$
which is in degree $p$ :
$$Sym^p(E) \to \mathcal{O}_{gr}(\mathbb{V}(E))(p).$$

We obtain a commutative diagram

\[\begin{tikzcd}
Sym^p E \ar{d} \ar{r}& Tot(Sym^p(coN_\bullet(E \to E'))) \ar{d} \\
R_p(E) \ar{r} & Tot(R_p(coN_\bullet(E \to E')))
\end{tikzcd}\]
with $R_p = \mathcal{O}_{gr}(\mathbb{V}(-))(p)$

The right arrow is similarly shown to be an equivalence by induction and the bottom one is an equivalence by the same reasoning above.

We now need to show the top arrow is an equivalence. Using the previous lemma with $$E^\bullet \coloneqq coN_\bullet(E \to E')$$

concludes the proof.

\end{proof}

\begin{Rq}
In general characteristic, and general weight, $\mathcal{O}_{gr}(\mathbb{V}(E))(p)$ seems to correspond to $Sym_{M}^p(E)$, the weight $p$ part of the symmetric algebra over a monad on complexes generalizing simplicial algebras to non-necessarily connective complexes. This monad is currently being investigated by Bhatt-Mathew, see \cite{Rak20}. Hopefully, this connection will  be made explicit in a following paper.
\end{Rq}

\begin{Rq}
\label{Contre-ex}
We provide an example of a situation where functions on a linear stack are not easily expressed in terms of the complex of sheaves. Working over a field of characteristic zero, we can define a quasi-coherent sheaf on the point by $E = k[2] \oplus k[-2]$. Then $Sym(E) \simeq Sym(k[2]) \otimes Sym(k[-2])$ is a free commutative differential graded $k$-algebra on two generators, one of degree $2$, the other of degree $-2$, the differential is identically zero.

On the other hand $\mathcal{O}(\mathbb{V}(E)) = H^*(K(\mathbb{G}_a,2),\mathcal{O}[u])$ and it identifies with the ring $k[u][[v]]$ with $v$ a generator of $H^2(K(\mathbb{G}_a,2),\mathcal{O})$ : $u$ is in degree $-2$ and $v$ is a degree $2$. In particular, in degree $0$, we get for $Sym(E)$ an infinite product $\prod_{\mathbb{N}} k$ generated by the elements of the form $u^n v^n$. Similarly, in degree $0$, $\mathcal{O}(\mathbb{V}(E))$ only gives a direct sum $\bigoplus_{\mathbb{N}} k$, also generated by the elements $u^n v^n$. Intuitively, the connective part of $E$ tends to give a contribution by a free algebra (polynômial algebras) and the coconnective part of $E$ gives a contribution by a completed free algebra (power series algebras).

\end{Rq}

\section{Overview of de Rham and crystalline cohomology}

\subsection{Crystalline cohomology}

We review the theory of classical crystalline cohomology, following \cite{Ber74}, \cite{CL91} and \cite{LZ03}. For results on comparisions between de Rham and crystalline cohomology, see \cite{BDJ11} and \cite{Mon21b}.

Crystalline cohomology was motivated to fill a gap arising from $l$-adic cohomology : $l$-adic cohomology for schemes over a field of characteristic $p$ is a well-behaved Weil cohomology theory under the assumption that $l \neq p$. However it is badly behaved when $l=p$. Crystalline cohomology was constructed to give a Weil cohomology theory in the case $l=p$.
Serre notices in $\cite{Ser58}$ that there is no Weil cohomology theory with coefficients in $\mathbb{Q}$ for schemes over a field of characteristic $p$. This motivates the use of Witt vectors, which give the appropriate ring of coefficients, being both a ring of characteristic zero and having a residue field of characteristic $p$.

\begin{Def}
For $X$ a classical scheme over a $\mathbb{F}_p$-algebra $k$, one defines a crystalline site relative to $W_n$, the ring of Witt vector of length $n$ on $k$, and a structural sheaf $\mathcal{O}_{X/W_n(k)}$ which defines cohomology groups $H^i_{crys}(X/W_n)$.

Crystalline cohomology is then the limit on $n$
$$H^i_{crys}(X/W) \coloneqq \varprojlim H^i(X/W_n)$$
\end{Def}

\begin{Rq}
On Witt vectors of length $1$, ie $W_1=k$, we recover De Rham cohomology, therefore this defines a canonical morphism
$$H^i_{crys}(X/W) \to H^i_{dR}(X/k)$$
\end{Rq}

\begin{Prop}
Crystalline cohomology is a Weil theory for smooth projective schemes over a perfect field of characteristic $p$.
\end{Prop}

\begin{proof}
See \cite{Ber74}.
\end{proof}

\subsection{De Rham-Witt complex}

This section is devoted to reviewing the de Rham-Witt complex, as developed in \cite{Ill79}. The motivation for defining the de Rham-Witt complex is twofold  : 
\begin{itemize}
\item Giving an explicit tool to compute crystalline cohomology closer to the definition of de Rham cohomology.
\item Comparing crystalline cohomology to other $p$-adic cohomology theories, étale or de Rham cohomology, and studying other structures, such as the Hodge filtration.
\end{itemize}

\begin{Cons}
Let $X$ be a scheme over $\mathbb{F}_p$, the de Rham-Witt pro-complex $W_\bullet \Omega_X^\bullet$ is a projective system $(W_n \Omega_X^\bullet)$ of coconnective commutative differential graded algebras such that
$$W_n \Omega_X^0 = W_n \mathcal{O}_X$$
endowed with an additive morphism $V : W_n \Omega_X^i \to W_{n+1} \Omega_X^i$ such that 
$$V(x dy) = Vx dVy \text{ and } (d[x])Vy = V([x]^{p-1} d[x]y)$$

There is a unique initial complex with theses properties and $W_n \Omega^\bullet_X$ is a quotient of $\Omega_{W_n X}$ by the required relations, it is a $ W_n(k)$-algèbre, $W_1 \Omega_X^\bullet$ is simply the de Rham complex $\Omega_X^\bullet$.

We define $W \Omega^i_X = \varprojlim W_n \Omega_X^i$, it is a $p$-torsion free commutative differential graded $W(k)$-algebra.

We recover the truncated de Rham-Witt complex $W_n \Omega_X$ by a quotient :
$$W_n \Omega_X^i = W \Omega^i_X / (V^n W\Omega_X^i + dV^nW \Omega_X^{i-1})$$

Since $X$ is assumed to be smooth, there is a unique commutative differential graded algebra morphism
$F : W_n \Omega_X^\bullet \to W_n \Omega_X^\bullet$ such that
$$FdV=d \text{  and } F(d[x])=[x]^{p-1}d[x]$$
\end{Cons}

\begin{Prop}
A few properties of the de Rham-Witt vectors and its structure maps
$$V : W_n\Omega_X^\bullet \to W_{n+1}\Omega_X^\bullet$$
$$F : W_n\Omega_X^\bullet \to W_n\Omega_X^\bullet$$

are given by :
\begin{itemize}
\item $W_n \Omega_X^0$ is given by $W_n \mathcal{O}_X$
\item $F$ and $V$ are additive.
\item $F$ coincides in degree $0$ with the Frobenius morphism and the Verschiebung map on Witt vectors.
\item $FV=VF=p$
\item $FdV=d$
\item $dF=pFd$
\item $Fd[x]=[x^{p-1}]d[x]$ for $x\in \mathcal{O}_X$
\item $F(xy) = xy$
\item $xVy=VF(xy)$
\item $V(xdy)=V(x)dV(y)$
\end{itemize}

\end{Prop}

\begin{Rq}
The endomorphism deduced functorially by the Frobenius map of $X$ is on $W_\bullet \Omega^i$ given by $p^i F$.
\end{Rq}

\begin{Thm}
There is a canonical isomorphism compatible with the Frobenius action
$$H^*_{crys}(X/W_n(k)) \cong H^*(X,W_n \Omega_X^\bullet)$$
between crystalline cohomology and the hypercohomology of the de Rham-Witt complex sheaf canonically defined.
\end{Thm}

\begin{Thm}
Furthermore, there is even a more structured result : there is an equivalence of $W_n(k)$-chain complexes on $X$
$$ Ru_* \mathcal{O}_{X/W_n(k)} \xrightarrow{\sim} W_n \Omega_X^\bullet \in \mathcal{D}(X,W_n(k))$$
where $u$ is the forgetful functor from sheaves on the crystalline site of $X$ to sheaves on the Zariski site of $X$. We have used the notation $\mathcal{D}(X,W_n(k))$ to denote the derived category of $W_n(k)$-modules on $X$.
\end{Thm}

\subsection{De Rham Witt from Bhatt and Lurie}
In this section, we recall the main notions and results of \cite{BLM22}.

\begin{Def}
A Dieudonné complex $(M,d,F)$ is a cochain complex of abelian groups $(M,d)$ and $F$ a morphism of graded abelian groups $F : M \to M$ such that
$$dF = pFd$$

The $1$-category of Dieudonné complexes is denoted $\textbf{DC}$.
\end{Def}

\begin{Def}
Let $(M,d,F)$ be a Dieudonné complex, we define a Dieudonné sub-complex of $M$ denoted $(\eta_p M)$ as follows
$$(\eta_p M)^n \coloneqq \{ x \in p^n M : dx \in p^{n+1}M^{n+1} \}$$
\end{Def}

\begin{Def}
Let $(M,d)$ be a $p$-torsion-free complex, the Frobenius morphism $F$ induces a morphism of abelian groups
$$\alpha_n : M^n \to (\eta_p M)^n = \{ x \in p^n M : dx \in p^{n+1}M^{n+1} \}$$

The complex $M$ is said to be saturated when $\alpha_n$ is an isomorphism for all $n$. This defines the sub-$1$-category of saturated Dieudonné complexes $\textbf{DC}_{sat} \subset \textbf{DC}$.
\end{Def}

\begin{Def}
Let $M$ be a saturated Dieudonné complex, for every $x \in M_n$, there is a unique $V(x)$ such that $FV(x) = px$. This defines an additive morphism
$$V : M \to M$$
\end{Def}

\begin{Prop} (See \cite[Proposition 2.3.1]{BLM22})
The inclusion of $1$-categories $\textbf{DC}_{sat} \subset \textbf{DC}$ admits a left adjoint, called saturation. It is denoted $Sat$.
\end{Prop}

\begin{Rq}
This saturation is constructed as the colimit of the diagram
$$M \to \eta_p M \to \eta_p^2 M ...$$
\end{Rq}

\begin{Def}
For $M$ a saturated Dieudonné complex, we define $\mathcal{W}_r(M)$ as the quotient complex of $M$ by the subcomplex $Im(V^r) + Im(dV^r)$.

The natural morphisms $\mathcal{W}_{n+1}(M) \to \mathcal{W}_n(M)$ defines a diagram, which admits a limit called completion of $M$. We denote it $\mathcal{W}(M)$.
\end{Def}

\begin{Rq}
The maps $F$ and $V$ naturally extend to the completion of $M$.

We have a canonical map given by the projection : $\rho_M : M \to \mathcal{W}(M)$.
\end{Rq}

\begin{Def}
The Dieudonné complex $M$ is said to be strict when it is saturated and $\rho_M$ is an isomorphism.

This defines a sub-$1$-category $\textbf{DC}_{str} \subset \textbf{DC}_{sat}$.
\end{Def}

\begin{Prop} (See \cite[Proposition 2.7.8]{BLM22})
The inclusion $\textbf{DC}_{str} \subset \textbf{DC}_{sat}$ admits a left adjoint given by $\mathcal{W}$.
\end{Prop}

\begin{Prop} (See \cite[Notation 2.8.4]{BLM22})
The inclusion $\textbf{DC}_{str} \subset \textbf{DC}$ admits a left adjoint given by $\mathcal{W}Sat$.
\end{Prop}

\begin{Def} \label{DADef}
A Dieudonné algebra $(A,d,F)$ is given by $(A,d)$ a commutative differential graded algebra and $F : A \to A$ is a morphism of graded rings such that
\begin{itemize}
\item For $x \in A$, $dF(x) = pF(dx)$.
\item $A_n=0$ for $n<0$.
\item $Fx=x^p $ (mod $p)$, for $x \in A^0$.
\end{itemize}

Morphisms between Dieudonné algebras are required to be ring morphisms compatible with differentials and Frobenius structures. This defines the $1$-category of Dieudonné algebras $\textbf{DA}$. Forgetting the multiplication defines the forgetful functor $\textbf{DA} \to \textbf{DC}$.
\end{Def}

\begin{Prop} (See \cite[Proposition 3.2.1]{BLM22})
Let $R$ a $p$-torsion-free commutative ring and $\phi : R \to R$ a classical Frobenius lift. Then there is a unique ring morphism $F : \Omega_R^\bullet \to \Omega_R^\bullet$ such that
\begin{itemize}
\item $F(x) = \phi(x)$ for $x\in R = \Omega_R^0$.
\item $F(dx) = x^{p-1} dx + d(\frac{\phi(x)-x^p}{p})$ for $x \in R$.
\end{itemize}
Furthermore, $(\Omega_R^\bullet,d,F)$ is a Dieudonné algébra.
\end{Prop}

\begin{Prop} \label{FormDieud} (See \cite[Proposition 3.2.3]{BLM22})
Let $R$ a $p$-torsion-free commutative ring and $\phi : R \to R$ a classical Frobenius lift. For $A$ a $p$-torsion-free Dieudonné algebra, the construction given before defines a functorial bijection
$$Hom_{\textbf{DA}}(\Omega_R^\bullet,A) \xrightarrow{\sim} Hom_{CRing^{Fr}}(R,A^0)$$
where $Hom_{CRing^{Fr}}(R,A^0)$ is the set of morphisms $R \to A^0$ compatible with Frobenius structures.
\end{Prop}

\begin{Def}
A Dieudonné algebra is said to be saturated is it is as a Dieudonné complex. This defines a sub-$1$-category $\textbf{DA}_{sat} \subset \textbf{DA}$.
\end{Def}

\begin{Prop} (See \cite[Proposition 3.4.3]{BLM22})
The inclusion $\textbf{DA}_{sat} \subset \textbf{DA}$ admits a left adjoint given by $Sat(A)$ endowed with the canonical Dieudonné algebra structure.
\end{Prop}

\begin{Prop} (See \cite[Proposition 3.5.5]{BLM22})
The constructions of $V$ and $\mathcal{W}$ extend naturally to Dieudonné algebras. For $A$ a Dieudonné algebra, $\mathcal{W}(A)$ is a saturated Dieudonné algebra.
\end{Prop}

\begin{Def}
A saturated Dieudonné algebra $A$ is said to be strict if it is as a Dieudonné complex. This defines a sub-$1$-category $\textbf{DA}_{str} \subset \textbf{DA}_{sat}$.
\end{Def}

\begin{Prop} (See \cite[Proposition 3.5.9]{BLM22})
The inclusion $\textbf{DA}_{str} \subset \textbf{DA}_{sat}$ admits a left adjoint given by $\mathcal{W}$.
\end{Prop}

\begin{Prop} (See \cite[Proposition 3.5.10]{BLM22})
The inclusion $\textbf{DA}_{str} \subset \textbf{DA}$ admits a left adjoint given by $\mathcal{W}Sat$.
\end{Prop}

\begin{Prop} (See \cite[Proposition 4.1.4]{BLM22})
The functor $B \in \textbf{DA}_{str} \mapsto B^0/VB^0 \in CRing_{\mathbb{F}_p}$ admits a left adjoint denoted by $R \mapsto \mathcal{W}\Omega_R^\bullet$ and called the saturated de Rham-Witt complex of $R$. The complex can be constructed as $$\mathcal{W}\Omega_R^\bullet \coloneqq \mathcal{W}Sat(\Omega^\bullet_{W(R)})$$
\end{Prop}

\begin{Def}
Let $R$ be a commutative $\mathbb{F}_p$-algebra. An $R$-framed $V$-pro-complex is given by
\begin{itemize}
\item A chain of commutative differential graded algebras
$$ ... \to A_\bullet(3) \to A_\bullet(2) \to A_\bullet(1) \to A_\bullet(0)$$
with the limit denoted $A_\bullet(\infty)$.
\item A collection of maps $V : A(r) \to A(r+1)$ of graded abelian groups.
\item A ring morphism $\beta : W(R) \to A_0(\infty)$, where we denote $\beta_r$ the composition
$$W(R) \xrightarrow{\beta} A_0(\infty) \to A_0(r)$$
\end{itemize}
such that
\begin{itemize}
\item $A(r)=0$ when r=0 and $A_i(r)=0$ when $i<0$.
\item $V$ commutes with the restriction maps. Therefore $V$ induces a morphism on $A(\infty)$ denoted $V$.
\item $\beta : W(R) \to A_0(\infty)$ is compatible with $V$.
\item $V : A(r) \to A(r+1)$ satisfies $V(xdy) = V(x)dV(y)$.
\item For $\lambda \in R$ and $x \in A(r)$, we have
$$(Vx)d\beta_{r+1}[\lambda] = V(x (\beta_r[\lambda])^{p-1} d\beta_r [\lambda] )$$
\end{itemize}

A morphism of $R$-framed $V$-pro-complexes is a morphism of diagrams of commutative differential graded algebras compatible with restriction, Verschiebung and $\beta$ maps.

This defines a $1$-category $VPC_R$.
\end{Def}

\begin{Prop} (See \cite[Proposition 4.4.4]{BLM22})
Let $R$ be a commutative $\mathbb{F}_p$-algebra, the $1$-category $VPC_R$ has an initial object. We denote this object $(W_r \Omega_R^\bullet)$ and the limit of the induced tower $(W \Omega_R^\bullet)$, we call it the classical de Rham-Witt complex of $R$.
\end{Prop}

\begin{Rq}
This construction gives the de Rham-Witt complex as defined in \cite{Ill79}. See \cite[Warning 4.4.8]{BLM22} for details.
\end{Rq}

\begin{Thm} (See \cite[Proposition 4.4.12]{BLM22})
There is a functorial morphism of commutative differential graded algebra
$$W\Omega_R^\bullet \to \mathcal{W} \Omega_R^\bullet$$
which is an isomorphism when $R$ is a Noetherian $\mathbb{F}_p$-algebra. The left hand side it the classical de Rham-Witt complex and the right hand side is the de Rham-Witt complex defined by saturation and completion.
\end{Thm}

\section{Graded loopspaces in mixed characteristics}
In this section we will take $k = \mathbb{Z}_{(p)}$.

\subsection{Frobenius lift structures}

\subsubsection{Classical and derived Frobenius lifts}

\begin{Def}
Let $A$ be a discrete commutative ring, a classical Frobenius lift on $A$ is given by a ring endomorphism $F : A \to A$ such that $F$ induces the Frobenius morphism on $A/p$, ie $F(a)-a^p$ is $p$-divisible for any $a \in A$.
\end{Def}

\begin{Def} 
We define the category of derived stacks with a derived, or homotopy, Frobenius lift $dSt^{Fr}$ as the pullback of categories
$$dSt^{endo} \times_{dSt_{\mathbb{F}_p}^{endo}} dSt_{\mathbb{F}_p}$$
where the functor $dSt_{\mathbb{F}_p} \to dSt_{\mathbb{F}_p}^{endo}$ is the canonical functor adjoining the Frobenius endomorphism to a derived stack on $\mathbb{F}_p$. Similarly, we define the category of graded derived stacks endowed with a derived Frobenius lift $dSt^{gr, Fr}$ as the pullback of categories
$$dSt^{gr,endo} \times_{dSt_{\mathbb{F}_p}^{endo}} dSt_{\mathbb{F}_p}$$
where the forgetful functor $dSt^{gr,endo} \to dSt_{\mathbb{F}_p}^{endo}$ is given by taking $0$ weights, that is $\mathbb{G}_m$-fixed points by Definition \ref{DefWeight0dSt} and taking the fiber on $\mathbb{F}_p$.
\end{Def}

We will give a more general definition of derived Frobenius lifts.

\begin{Def} [Homotopy coherent Frobenius lift]
Let $\mathcal{C}$ a category endowed with a functor $\mathcal{C} \to dSt_{\mathbb{F}_p}^{endo}$. We define the category of Frobenius lifts on $\mathcal{C}$ as
$$\mathcal{C}^{Fr} \coloneqq \mathcal{C} \times_{dSt_{\mathbb{F}_p}^{endo}} dSt_{\mathbb{F}_p}$$
where the $dSt_{\mathbb{F}_p} \to dSt_{\mathbb{F}_p}^{endo}$ is the canonical functor adjoining the Frobenius endomorphism to a derived stack on $\mathbb{F}_p$.
\end{Def}

\begin{Ex} 
\begin{itemize}  We give a series of examples for categories of "algebra-type objects"
\item with $\mathcal{C} = CRing_{\mathbb{Z}_{(p)}}$, we define the category of Frobenius lifts on $CRing_{\mathbb{Z}_{(p)}}$ :
$$CRing_{\mathbb{Z}_{(p)}}^{Fr} \coloneqq (CRing_{\mathbb{Z}_{(p)}}^{endo,op})^{Fr,op}$$
The canonical morphism $CRing_{\mathbb{Z}_{(p)}}^{endo,op} \to dSt^{endo}_{\mathbb{F}_p}$ is the derived affine functor modulo $p$.
	\item with $\mathcal{C} = SCR_{\mathbb{Z}_{(p)}}$, we define the category of Frobenius lifts on $SCR_{\mathbb{Z}_{(p)}}$ :
$$SCR_{\mathbb{Z}_{(p)}}^{Fr} \coloneqq (SCR_{\mathbb{Z}_{(p)}}^{endo,op})^{Fr,op}$$
The canonical morphism $SCR_{\mathbb{Z}_{(p)}}^{endo,op} \to dSt^{endo}_{\mathbb{F}_p}$ is the derived affine functor modulo $p$.
	\item with $\mathcal{C} = SCR_{\mathbb{Z}_{(p)}}^{gr}$, we define the category of Frobenius lifts on $SCR_{\mathbb{Z}_{(p)}}^{gr}$ :
$$SCR_{\mathbb{Z}_{(p)}}^{gr,Fr} \coloneqq (SCR_{\mathbb{Z}_{(p)}}^{gr,endo,op})^{Fr,op}$$
The canonical morphism $SCR_{\mathbb{Z}_{(p)}}^{gr,endo,op} \to dSt^{endo}_{\mathbb{F}_p}$ is the derived affine functor modulo $p$ after taking the weight $0$ componant defined by Proposition \ref{DefWeight0SCR}. See Remark \ref{CompatWeight0} for details.
\end{itemize}
\end{Ex}

\begin{Rq}
Our definition requires a derived Frobenius lift to be homotopy equivalent on $\mathbb{F}_p$ with the canonical Frobenius only on $0$-weights, we do not require an homotopy to $0$ on the other weights.
\end{Rq}

\begin{Prop}
Let $A$ be a $p$-torsion-free commutative algebra on ${\mathbb{Z}_{(p)}}$, the space of Frobenius lifts on $A$ is discrete and in bijection with the set of classical Frobenius lifts on $A$
$$\{ \phi : A \to A : \phi_p : A/p \to A/p = Fr_p \}$$
\end{Prop}

\begin{proof}
Given an endomorphism of $A$, the data of being a derived Frobenius lift is a path in $$Map_{SCR_{\mathbb{F}_p}}(A \otimes^{\mathbb{L}} \mathbb{F}_p, A \otimes^{\mathbb{L}} \mathbb{F}_p)$$
between the induced endomorphism and the canonical Frobenius. As $A$ is $p$-torsion-free we have 
$$A \otimes^{\mathbb{L}} \mathbb{F}_p \simeq A/p$$
which is a discrete commutative $\mathbb{F}_p$ algebra. Now, we deduce
$$Map_{SCR_{\mathbb{F}_p}}(A \otimes^{\mathbb{L}} \mathbb{F}_p, A \otimes^{\mathbb{L}} \mathbb{F}_p) \simeq Map_{SCR_{\mathbb{F}_p}}(A/p,A/p) \simeq Hom_{CRing_{\mathbb{F}_p}}(A/p,A/p)$$
which is a discrete space. Therefore the choice of a homotopy coherent Frobenius lift is equivalent to the choice of a classical Frobenius lift.
 \end{proof}

\begin{Def}
Let $A$ be a $p$-torsion-free commutative algebra on $\mathbb{Z}_{(p)}$, let $M$ and $N$ be $A$-module. An $(A,F)$-linear map is a morphism of ${\mathbb{Z}_{(p)}}$-modules $M \to N_F$ where $N_F$ is the ${\mathbb{Z}_{(p)}}$-module on $N$ endowed with the $A$-module structure induced by composition by $F$.
\end{Def}

\begin{Prop} \label{FLSymStr}
Let $A$ be a $p$-torsion-free commutative algebra on $\mathbb{Z}_{(p)}$, $M$ a projective $A$-module of finite type. The space of Frobenius lifts on the graded simplicial algebra $Sym_A(M[n])$, with $M$ in weight $1$ and $n>0$ a positive integer, is discrete and is made up by pairs $(F,\phi)$ with $F$ a Frobenius lift of $A$ and $\phi : M \to M$ an $(A,F)$-linear map. Therefore once the Frobenius $F$ is fixed, $\phi$ is the data of a classical $(A,F)$-module structure on the $A$-module $M$.
\end{Prop}

\begin{proof}
Taking weight $0$, a derived Frobenius lift on $Sym_A(M[n])$ induces a derived Frobenius lift on $A$. From the previous proposition, a Frobenius structure on $A$ is a classical Frobenius lift : we denote it $F$.

We start by considering the space of choices of endomorphisms on $Sym_A(M[n])$ compatible with $F$, which are morphisms of simplicial $A$-algebras from $Sym_A(M[n])$ to $Sym_A(M[n])$ with the former endowed with the canonical $A$-algebra structure and the latter with the $A$-algebra structure induced by $F$.
\begin{align*}
End_{SCR^{gr}_A}(Sym_A(M[n]))&\simeq Map_{A-Mod^{gr}}(M[n],Sym_A(M[n])_F)\\
&\simeq Map_{A-Mod}(M[n], M_F[n]) \\
&\simeq Map_{A-Mod}(M, M_F) \\
\end{align*}
where $M_F$ is $M$ endowed by the $A$-module structure induced by $F$.

Now the weight $0$ part of $Sym_A(M[1])$ is simply $A$. Therefore the homotopy with the canonical Frobenius is uniquely determined by $F$.
 \end{proof}

\begin{Prop}
The space of derived Frobenius lifts on $\mathbb{F}_p$ is empty.
\end{Prop}

Before we move on to the proof of the proposition, we will need a lemma.

\begin{Le}
Let $A$ be a discrete commutative algebra and $a \in A$, there exists a simplicial algebra $K(A,a)$ with the following universal property : let $B$ be a simplicial $A$-algebra, there is a functorial equivalence
$$Map_{A-SCR}(K(A,a),B) \to \Omega_{a,0} B$$
where the space of paths $\Omega_{a,0} B$ from $a$ to $0$ in $B$ is defined as the fiber of
$$B^{\Delta^1} \xrightarrow{ev_0,ev_1} B \times B$$
over $(a,0)$.
\end{Le}

\begin{proof}
The category $SCR$ is presentable and the functor
$$B \in SCR \to \Omega_{a,0} B$$
preserves limits and is accessible, therefore using \cite[Proposition 5.5.2.7]{Lur07}, it is representable.
\end{proof}

\begin{Rq}
We can explicitly construct $K(A,a)$ as $Sym_A(S^1)$ where $Sym_A(-)$ is the free simplicial $A$-algebra on a simplcial set and $S^1=\Delta^1/\partial \Delta^1$ is the simplicial circle.
\end{Rq}

\begin{proof}[Proof of the proposition]
We first compute the mapping space $Map_{SCR}(\mathbb{F}_p,\mathbb{F}_p)$. We know $$Map_{SCR_{\mathbb{Z}_{(p)}}}(\mathbb{F}_p,\mathbb{F}_p) = Hom_{CRing}(\mathbb{F}_p,\mathbb{F}_p)$$ is contractible. The only endomorphism of $\mathbb{F}_p$ as a simplicial algebra over $\mathbb{Z}_{(p)}$ is the identity up to contractible choice. We take
$$K(\mathbb{Z}_{(p)},p)$$ as a cofibrant model of $\mathbb{F}_p$ and moding out by p :
$$K(\mathbb{Z}_{(p)},p) \otimes_{\mathbb{Z}_{(p)}} \mathbb{F}_p \simeq K(\mathbb{F}_p,0) \simeq \mathbb{F}_p[\epsilon]$$
is the free simplicial algebra on one generator in degree $1$, as the construction $K$ is stable under base change. The identity induces in homology the identity morphism, however, the Frobenius on $\mathbb{F}_p[\epsilon]$ sends the degree $1$ generator $\epsilon$ to $\epsilon^p=0$ in homology. Therefore the identity does not have a derived Frobenius lift structure.

 \end{proof}

\begin{Le}
The forgetful functor $U : dSt^{gr, Frob} \to  dSt^{gr}$ preserves the classifying space construction $B$.
\end{Le}

\begin{proof}
For a graded Frobenius derived stack in groups $G$, $BG$ is explicitly given as a geometric realization of graded derived stacks of the form $G^n$. We write $U$ as a composition
$$dSt^{Fr} \to dSt^{endo} \to dSt$$
the first one is the projection of a fiber product of categories, therefore it preserves small limits. The second map also preserves small limits since in a category of presheaves, they are computed pointwise. Hence $U$ preserves finite products. The same decomposition show that $U$ preserves geometric realizations, which concludes the proof, see the proof of Proposition \ref{DieudTopos} for a similar argument.
 \end{proof}

\begin{Rq}
$B\mathbb{G}_m$ is the final element in $dSt^{gr}$ and it admits a unique derived Frobenius lift structure, therefore it is also the final element in $dSt^{gr,Frob}$.
\end{Rq}

\begin{Rq}
By definition
$$SCR^{Fr} \coloneqq dAff^{Fr,op}$$
We notice there is an equivalence of categories
$$SCR^{Fr} \simeq SCR^{endo} \times_{SCR_p^{endo}} SCR_p$$
 compatible with the functor
$$Spec : SCR \xrightarrow{\sim} dAff^{op}$$
\end{Rq}

\subsubsection{Derived Witt functor}

\begin{Prop} 
The forgetful functor
$$U : SCR^{Fr} \to SCR$$
admits a right adjoint adjoint, which we call the derived Witt ring functor.
\end{Prop}

\begin{proof} 
From the description of $SCR^{Fr}$ as the fiber product
$$SCR^{Fr} \coloneqq SCR^{endo} \times_{SCR^{endo}_p} SCR_p$$
the forgetful functor
$$U : SCR^{Fr} \to SCR$$
commutes with colimits. We conclude using the adjoint functor theorem \cite[Corollary 5.5.2.9]{Lur09}.
\end{proof}

\begin{Rq} 
When $A$ is discrete, $W(A)$ is the discrete commutative ring of Witt vectors. See \cite{Joy85} for details on the classical adjunctions. We can see the simplicial algebra $W(A)$ is discrete by using the following identifications
$$Map_{SCR}(R,W(A)) \simeq Map_{SCR^{Fr}}(L(R),W(A)) \simeq Map_{SCR}(L(R),A)$$
for any simplicial algebra $R$, where $L$ is left adjoint to the forgetful functor and $W$ is the right adjoint. The mapping space is discrete for every simplicial algebra $R$ therefore $W(A)$ is discrete.
\end{Rq}

\begin{Cons}
Let $A$ be a graded simplicial algebra, which is positively or negatively graded, we construct its associated graded simplicial algebra with Frobenius lift of graded Witt vectors $W^{gr}(A)$ as follows. As a graded simplicial algebra with an endomorphism, we define $W^{gr}(A)$ as the fiber product
$$A^{\mathbb{N}} \times_{(A(0))^{\mathbb{N}}} W(A(0))$$
where
$$W(A(0)) \to (A(0))^{\mathbb{N}}$$
is the ghost morphism, see Remark \ref{WeightNaiveVsNormal} for the connection between naive $0$ weights and $0$ weights. By construction, $W^{gr}(A)$ is positively or negatively graded and its $0$ weights ring is given by taking naive $0$ weights
$$W^{gr}(A)(0) \simeq W(A(0))$$

The natural Frobenius structure on $W(A(0))$ endowes $W^{gr}(A)$ with the structure of a graded simplicial algebra with a Frobenius lift.

This construction defines a functor
$$W^{gr} : SCR^{gr} \to SCR^{gr,Fr}$$
\end{Cons}

\begin{Prop} 
Let $A$ be a graded simplicial algebra, which is positively or negatively graded. The functor
$$R \in SCR^{gr,Fr} \mapsto Map_{SCR^{gr}}(B,A) \in \mathcal{S}$$
is representable by $W^{gr}(A)$.
\end{Prop}

\begin{proof} 
Let $A$ be a graded simplicial algebra and $R$ be a graded simplicial algebra with Frobenius lift. By the construction of $W^{gr}(A)$, a morphism of graded simplicial algebras with Frobenius lifts
$$R \to W^{gr}(A)$$
is given by a morphism of graded simplicial algebras with endomorphisms
$$R \to A^{\mathbb{N}}$$
and a morphism of graded simplicial algebras with Frobenius lifts
$$R \to W(A(0))$$
with a compatibility between the two morphisms. The former morphism is simply given by a morphism of graded simplicial algebras $A \to R$, the latter is given by a morphism of simplicial algebras with Frobenius lifts
$$R(0) \to W(A(0))$$
which is simply a morphism of simplicial algebras
$$R(0) \to A(0)$$
which is fixed by requiring the compatibility. Therefore a morphism of graded simplicial algebras with Frobenius lifts
$$R \to W^{gr}(A)$$
is uniquely determined by a morphism of graded simplicial algebras $A \to R$.
\end{proof}

\subsubsection{Modules on an algebra endowed with a Frobenius lift}

\begin{Def}
Let $(A,F)$ be a simplicial algebra endowed with an endomorphism. We define the category of modules on $(A,F)$ as the stabilization of the category of simplicial algebras over $(A,F)$, see section \ref{tangentFormalism} for details of the tangent category formalism. That is
$$(A,F)-Mod^{endo} \coloneqq Stab(SCR^{endo}_{/(A,F)})$$

Similarly, for $(A,F,h)$ a simplicial algebra endowed with a Frobenius lift, we define its category of modules
$$(A,F,h)-Mod^{Fr} \coloneqq Stab(SCR^{Fr}_{/(A,F,h)})$$
\end{Def}

\begin{Rq}
We can notice that a module on a simplicial algebra with endomorphism $(A,F)$ is simply a non-necessarily connective $A$-module endowed with an endophism compatible with the action of $F$.
\end{Rq}

\begin{Rq}  \label{AFhmodDescr}
The category $(A,F,h)-Mod$ can be identified with
$$(A,F)-Mod^{endo} \times_{(A_p,F_p)-Mod^{endo}} A_p-Mod$$
where $(A_p,F_p)$ is the simplicial algebra obtained by base changing $(A,F)$ to $\mathbb{F}_p$. We deduce this identification from the following description of $SCR^{Fr}_{/(A,F,h)}$ :
$$SCR^{Fr}_{/(A,F,h)} \simeq SCR^{endo}_{/(A,F)} \times_{SCR^{endo}_{p/(A_p,F_p)}} SCR_{p,/(A_p,F_p)}$$
and the fact that stabilizations commute with the "endomorphism category construction $(-)^{endo}$" and small limits.

Therefore we can identify an object of $(A,F,h)-Mod$ with a $(A,F)$-module and an homotopy between $F_p$ and $0$.
\end{Rq}

\begin{Prop}  \label{FrEndoModequi}
Let $F : (A,F)-Mod^{endo} \to (A,F,h)-Mod^{Fr}$ be the functor sending the pair $(M,\alpha)$ to $(M,p\alpha)$ with the canonical homotopy. The functor $F$ is an equivalence of categories.

$$F : (A,F)-Mod^{endo} \xrightarrow{\sim} (A,F,h)-Mod^{Fr}$$

\end{Prop}

\begin{proof} 
We use the exact triangle
$$M \xrightarrow{ \times p} M \to M_p$$
where $M_p$ denotes $M \otimes \mathbb{F}_p$. Given a pair $(M,\alpha)$ in $(A,F)-Mod^{endo}$, promoting the pair to a $(A,F,h)$-module is given by the data of a homotopy between
$$M \xrightarrow{\alpha} M \to M_p$$
and zero. Using the exact triangle, this is equivalent to specifying an element $\alpha' : M \to M$ and a homotopy between $p \alpha'$ and $\alpha$, that is to say, this is an element in $F^{-1}((M,\alpha))$. This shows the essential surjectivity of the functor, and even the fully faithfulness on mapping spaces with equal source and target. We proceed similarly for general mapping spaces.
\end{proof}

\begin{Rq} 
From the previous proposition, a triple in $(M,\alpha,t)$ in $(A,F,h)-Mod^{Fr}$ may be seen as a pair $(M,\beta)$ in $(A,F)-Mod^{endo}$, we will denote write $\frac{\alpha}{p}$ for $\beta$, this element is constructed using the endomorphism $\alpha$ and the homotopy $t$.
\end{Rq}

\begin{Def}
We can define the cotangent complex of a simplicial algebra endowed with a endomorphism or a Frobenius lift using the formalism of \cite{Lur07}.

Let $(A,F)$ be a simplicial algebra endowed with an endomorphism. The cotangent complex of $(A,F)$ is an $(A,F)-Mod$ representing the functor
$$Map_{SCR^{endo}/(A,F)}((A,F), -)$$

Let $(A,F,h)$ be a simplicial algebra endowed with a Frobenius lift. The cotangent complex of $(A,F,h)$ is an $(A,F,h)-Mod$ representing the functor
$$Map_{SCR^{endo}/(A,F,h)}((A,F,h), -)$$
\end{Def}

\begin{Rq} \label{CotangEndo}
The cotangent complex of $(A,F)$ is simply given by $(\mathbb{L}_A,dF)$ where $dF$ is the endomorphism functorialy induced from $F$.

The forgetful functor $SCR^{Fr} \to SCR^{endo}$ induces a forgetful functor
$$(A,F,h)-Mod^{Fr} \to (A,F)-Mod^{endo}$$
\end{Rq}

\begin{Prop}  \label{FreeCFMod}
Let $(A,F)$ be a simplicial algebra endowed with an endomorphism, the forgetful functor
$$SCR^{endo}_{(A,F)} \to (A,F)-Mod^{endo}$$
admits a left adjoint denoted $Sym_{(A,F)}$.
\end{Prop}

\begin{proof}
We use the adjoint functor theorem, see \cite[Corollary 5.5.2.9]{Lur09}.
\end{proof}

\subsection{Mixed graded Dieudonné complexes and algebras}

\subsubsection{The graded circle}

\begin{Def}
The graded circle $S^1_{gr}$ as in \cite{MRT20} is given by
$$S^1_{gr} \coloneqq Spec^{\Delta}({\mathbb{Z}_{(p)}}[\eta]) \simeq BKer$$
It is endowed with an endomorphism induced from the multiplication by $p$ endomorphism on $Ker$, we will call this endomorphism $[p]$.The morphism $[p]$ sends $\eta$ to $p \eta$ in cohomology.
\end{Def}

\begin{Prop}
The endomorphism $[p]$ gives $S^1_{gr}$ the structure of a group in $dSt^{Fr}$
\end{Prop}

\begin{proof}
The stack $S^1_{gr}$ is given by applying the $B$ construction to the derived stack with Frobenius lift $Ker$.
\end{proof}

\begin{Rq}
Defining a morphism ${\mathbb{Z}_{(p)}}[\frac{x^n}{n!}] \to {\mathbb{Z}_{(p)}}[\frac{x^n}{n!}]$ by sending $\frac{x^n}{n!}$ to $p^n \frac{x^n}{n!}$ is canonically a morphism of graded classical affine schemes endowed with a Frobenius lift. As $x$ is in weight $1$, the condition of being a Frobenius lift is trivial.
\end{Rq}

\begin{Rq}
We also notice that $[p]$ is compatible with the group structure, meaning $\eta \mapsto p\eta$ is compatible with the coalgebra action $\eta \mapsto 1 \otimes \eta + \eta \otimes 1$ and the counit map.
\end{Rq}

\begin{Def} [Graded spheres] \label{GradedSpheres}
We define variants of spheres as graded affine stacks
$$S^{k}_{gr}(n) \coloneqq Spec^\Delta(\bigslant{{\mathbb{Z}_{(p)}}[\eta_k]}{\eta_k^2})$$
with $\eta_k$ of weight $n$ and degree $k$. We simply denote $S^k_{gr} \coloneqq S^k_{gr}(1)$.
\end{Def}

\begin{Rq}
The graded sphere $S^1_{gr}$ is the graded affine stack considered in \cite{MRT20}. The superior spheres $S^n_{gr}$ can be recovered from the topological spheres as follow : $$S^n_{gr} \simeq Spec^\Delta(D(H^*(S^n,{\mathbb{Z}_{(p)}})))$$
where $H^*(S^n,{\mathbb{Z}_{(p)}})$ is the graded commutative differential graded algebra of cohomology of the topological sphere $S^n$ with the zero differential and $D$ denotes the denormalization functor.
\end{Rq}

\begin{Def}
For $F$ a pointed graded stack, we define its graded homotopy groups :
$$\pi_k^{(n)}(F) = Hom_{dSt^{gr,*}}(S^k_{gr}(n),F) = \pi_0 Map_{dSt^{gr,*}}(S^k_{gr}(n),F)$$
The notation $Hom$ denotes the $\pi_0$ set of the associated mapping space. 
\end{Def}

\begin{Prop} \label{GradedCirclePushout} The graded circle endowed with its Frobenius structure can be recovered as the following pushout
$$S^1_{gr} \simeq * \sqcup_{Spec({\mathbb{Z}_{(p)}}[\rho])} *$$
of graded affine stacks
where $\rho$ is a generator in degree $0$ and weight $1$ which squares to zero. Furthermore, the induced diagram obtained by taking the product with $Spec(B)$ for $B$ a simplicial algebra
$$
\begin{tikzcd}
Spec(\mathbb{Z}_{(p)}[\rho]) \times Spec(B) \arrow[d] \arrow[r] & Spec(B) \arrow[d]\\
Spec(B) \arrow[r] & S^1_{gr} \times Spec(B)
\end{tikzcd}
$$
is a pushout diagram against derived affine schemes, meaning it induces a pullback diagram on the simplicial algebras of functions.

\end{Prop}

\begin{proof}
We follow the proof in \cite[Theorem 5.1.3]{MRT20}. We construct a commutative diagram

\begin{center}
\begin{tikzcd}
Spec({\mathbb{Z}_{(p)}}[\rho]) \arrow[d] \arrow[r] & * \arrow[d]\\
* \arrow[r] & BKer
\end{tikzcd}
\end{center}
that is an element $ Spec({\mathbb{Z}_{(p)}}[\rho]) \to Ker \simeq  \Omega BKer$. We choose
$$[\rho]=(\rho,0,0 ...) \in Ker({\mathbb{Z}_{(p)}}[\rho])$$
It is an element of $Ker({\mathbb{Z}_{(p)}}[\rho])$ since the Frobenius $F$ acts on a Teichmüller element as $F[a] = [a^p]$.

We are reduced to showing the natural map
$$C_\Delta(BKer) = {\mathbb{Z}_{(p)}} \oplus {\mathbb{Z}_{(p)}}[-1] \to {\mathbb{Z}_{(p)}} \times_{{\mathbb{Z}_{(p)}}[\rho]} {\mathbb{Z}_{(p)}}$$
is an equivalence of graded cosimplicial algebras. We only need to show it is an equivalence on the underlying complexes which is obvious. The compatibility with the endomorphism structures in straightforward and the additional Frobenius structure is trivial since the weight $0$ part of $S^1_{gr}$ is simply $Spec({\mathbb{Z}_{(p)}})$.

We move on to the second part of the proof. We want to show the following diagram is a pushout diagram against derived affine schemes $X$.

$$
\begin{tikzcd}
Spec(\mathbb{Z}_{(p)}[\rho]) \times Spec(B) \arrow[d] \arrow[r] & Spec(B) \arrow[d]\\
Spec(B) \arrow[r] & S^1_{gr} \times Spec(B)
\end{tikzcd}
$$

We can reduce to the case of $X=\mathbb{A}^1$ using \cite[Proposition 4.1.9]{Lur11} since any derived affine scheme can be written as a limit of copies of $\mathbb{A}^1$. Therefore we show the natural morphism
$$\mathcal{O}(S^1_{gr} \times Spec(B)) \to B \times_{\mathbb{Z}_{(p)}[\rho] \otimes B} B$$
is an equivalence. Since the functor of simplicial functions $\mathcal{O}$ is given by the composition of the Dold-Kan functor with $C(-)$, the functor of $\mathbb{E}_\infty$-functions, we simply show we have an equivalence
$$C(S^1_{gr} \times Spec(B)) \to B \times_{\mathbb{Z}_{(p)}[\rho] \otimes B} B$$
of complexes. Now using the finite cohomology property of $S^1_{gr}$, see \cite[Lemma 3.4.9]{MRT20} and the base change formula \cite[A.1.5-(2)]{HLP14}, we have
$$C(S^1_{gr} \times Spec(B)) \simeq C(S^1_{gr}) \otimes B$$
Now as we know $C(S^1_{gr}) = H^*(S^1_{gr},\mathbb{Z}_{(p)}) = \mathbb{Z}_{(p)} \oplus \mathbb{Z}_{(p)}[-1]$, we deduce the result using the first part of the proof.

\end{proof}

\subsubsection{Mixed graded Dieudonné complexes}

\begin{Def}
The endomorphism $[p] : S^1_{gr} \to S^1_{gr}$ induces a pullback morphism
$$[p]^* : QCoh(BS^1_{gr}) \to QCoh(BS^1_{gr})$$
which is an endofunctor of $\epsilon-Mod^{gr}$.
\end{Def}

\begin{Rq}
Given $(M,d,\epsilon)$ a graded mixed complex, $[p]^*M$ is given by $(M,d,p\epsilon)$.
\end{Rq}

\begin{Def}
We define the category of mixed graded Dieudonné complexes, also called derived Dieudonné complexes, by
$$\epsilon-D-Mod^{gr} \coloneqq CFP_{[p]^*}(\epsilon-Mod^{gr})$$
where the category on the right hand side is the $\infty$-category of colax fixed points of $[p]^*$ on $\epsilon-Mod^{gr}$, as defined in Remark \ref{CFPDef}.

We see the colax fixed point morphism $[p]^* M \to M$ can be seen as a morphism of graded mixed complexes
$$(M,d,p\epsilon) \to (M,d,\epsilon)$$
which is a morphism of graded complexes satisfying the usual Dieudonné relation
$$\epsilon F = pF \epsilon$$
\end{Def}

\begin{Prop} \label{TwoDefDieudComplex}
We have a natural identifications
$$\epsilon-D-Mod^{gr} \simeq (k[\epsilon],[p])-Mod^{gr,endo}$$
The object $(k[\epsilon],[p])$ is seen here as a commutative algebra object in $Mod^{gr,endo}$.
\end{Prop}

\begin{proof}
Given a mixed graded module $M \in \epsilon-Mod^{gr}$, promoting it to a $(k[\epsilon],[p])$-module $(M,F)$ is equivalent to the data of a commutative square

\begin{center}
\begin{tikzcd}
k[\epsilon] \otimes M \arrow[d,"\textrm{[p]} \otimes F"] \arrow[r]
& M \arrow[d,"F"] \\
k[\epsilon] \otimes M  \arrow[r]& M
\end{tikzcd}
\end{center}
which is equivalent to a colax fixed point structure
$$M \to [p]^*M$$
\end{proof}

\paragraph*{The Beilinson t-structure}

\begin{Def}
We recall the definition of the Beilinson t-structure on graded mixed complex. Let $M$ be a graded mixed complex, $M$ is said to be t-connective for the t-structure when $H^i(M(n))=0$ for $i>-n$, $M$ is said to be t-coconnective for the t-structure when $H^i(M(n))=0$ for $i<-n$.
\end{Def}

\begin{Rq}
As a mixed graded complex, $Sym_A(\mathbb{L}_A[1])$ is t-connective, for $A$ a simplicial algebra. It is in the heart of the t-structure when $A$ is a smooth classical algebra.
\end{Rq}

\begin{Prop} \label{HeartBeilClass}
The heart of $\epsilon-dg-mod^{gr}$ identifies with the $1$-category of complexes $\bf{dg-mod_{\mathbb{Z}_{(p)}}}$. We associate to a complex $(M,d)$ the graded mixed complex being $M_n$ with trivial differential in weight $-n$, the differential $d$ defines the mixed structure. This defines a functor
$$i : \bf{dg-mod_{\mathbb{Z}_{(p)}}} \to \bf{\epsilon-dg-mod^{gr}}$$
which induces an equivalence on the heart.
\end{Prop}

\begin{Def} [Beilinson t-structure]
Following \cite[\S 3.3]{Rak20}, we define a t-structure on the category of mixed Dieudonné complexes by letting a graded mixed Dieudonné complexes be t-connective, respectively t-connective, when its underlying graded mixed complex is.
\end{Def}

\begin{Prop} \label{HeartDMod}
The heart of $\epsilon-D-Mod_{\mathbb{Z}_{(p)}}$ identifies with the abelian category of Dieudonné complexes of \cite{BLM22}.
\end{Prop}

\begin{proof}
This follows from the identification without Dieudonné structures, Proposition \ref{HeartBeilClass}.
\end{proof}

\subsubsection{Mixed graded Dieudonné algebras}

We introduce the definition of the main objects we will study.

\begin{Def}
Motivated by Proposition \ref{TwoDefDieudComplex}, we define mixed graded Dieudonné stacks, also called derived Dieudonné stacks, by
$$\epsilon-D-dSt^{gr} \coloneqq S^1_{gr}-dSt^{gr,Frob}$$

We also define mixed graded Dieudonné simplicial algebras, also called Dieudonné simplicial algebras, by
$$\epsilon-D-SCR^{gr} \coloneqq (S^1_{gr}-dAff^{gr,Frob})^{op}$$
An element of $S^1_{gr}-dAff^{gr,Frob}$ can be thought of as a morphism $X \to BS^1_{gr}$, in the topos $dSt^{gr,Frob}$, which is relative affine.
\end{Def}

\begin{Rq}
A derived Dieudonné stack can be seen as a derived stack endowed with an endomorphism, which has a Frobenius lift structure, a grading and a compatible $S^1_{gr}$-action and a compatibility condition between the Frobenius lift and the action of $S^1_{gr}$. This compatibility condition is an analogue of the Dieudonné complex equation $dF = pFd$ for derived stacks.
\end{Rq}

\begin{Rq}
We expect $\epsilon-D-SCR^{gr}$ to admit another description as the category of modules over a monad $LSym$ on $\epsilon-D-Mod^{gr}$ as is done in \cite{BM19} and\cite{Rak20}.
\end{Rq}

\begin{Prop} \label{DieudTopos}
The two $\infty$-categories $dSt^{gr,Frob}$ and $S^1_{gr}-dSt^{gr,Frob}$ are $\infty$-topoi.
\end{Prop}

\begin{proof}
The category $dSt^{gr,Frob}$ is, by definition, given by $$\tau^{endo} \times_{\tau_p^{endo}} \tau_p$$ with $\tau = dSt^{gr}$ and $\tau_p = dSt^{gr}_{\mathbb{F}_p}$. Recalling \cite[Proposition 6.3.2.3]{Lur09}, to show that $dSt^{gr,Frob}$ is a topos, it is enough to show that the projection morphisms $\tau^{endo} \to \tau_p^{endo}$ and $\tau_p \to \tau_p^{endo}$ are left adjoints and they preserve finite limits, as $\tau^{endo}$, $\tau_p^{endo}$ and $\tau_p$ are already topoi.

Using \cite[Proposition 6.3.5.1]{Lur09}, the morphism $\tau^{endo} \to \tau_p^{endo}$ admits a right adjoint and commutes with finite limits, since it admits a left adjoint : the forgetful functor.

The morphism $\tau_p \to \tau_p^{endo}$ commutes with finite limits, since it admits a left adjoint, which is given by sending $(X,F)$ the homotopy coequalizer of $F$ and the Frobenius on $X$ : $X_{F \simeq Fr}$ and admits a right adjoint, which is given by sending $(X,F)$ the homotopy equalizer of $F$ and the Frobenius on $X$ : $X^{F \simeq Fr}$, see Proposition \ref{AdjointAjoutFrobCano}.

Therefore, $S^1_{gr}-dSt^{gr,Frob}$ is also a topos as the classifying topos of objects in $dSt^{gr,Frob}$ with a $S^1_{gr}$ action.
 \end{proof}

\begin{Rq}
From the previous proposition, the previous categories can admit Postnikov decomposition and obstruction theory, see \cite[Proposition 7.2.1.10]{Lur07} for details.
\end{Rq}

\paragraph*{Dieudonné functions on a Dieudonné stack}

\begin{Cons}
We construct the functor of functions on a Dieudonné derived stack
$$C : \epsilon-D-dSt^{gr,op} \to \epsilon-D-Mod^{gr}$$

by composition with the forgetful functor
$$S^1_{gr}-dSt^{gr,Frob} \to S^1_{gr}-dSt^{gr,endo}$$

we simply construct
$$S^1_{gr}-dSt^{gr,endo,op} \to S^1_{gr}-Mod^{gr,endo}$$

The category $S^1_{gr}-dSt^{gr,endo}$ identifies with $(S^1_{gr} \rtimes (\mathbb{N} \times \mathbb{G}_m))$-equivariant derived stacks, that is derived stacks over $B(S^1_{gr} \rtimes (\mathbb{N} \times \mathbb{G}_m))$. Indeed a $\mathbb{N} \times \mathbb{G}_m$-action is given by a grading and a graded endomorphism and an $(S^1_{gr} \rtimes (\mathbb{N} \times \mathbb{G}_m))$-action is given by an additional $S^1_{gr}$-action compatible with gradings and the endomorphisms.

Let us consider $X \in \epsilon-D-dSt^{gr}$, with its structure morphism
$$\pi : X \to B(S^1_{gr} \rtimes (\mathbb{N} \times \mathbb{G}_m))$$
Now the pushforward of the structure sheaf $\pi_* \mathcal{O}_X$ defines an element of
$$QCoh(B(S^1_{gr} \rtimes (\mathbb{N} \times \mathbb{G}_m)))$$
The category $QCoh(B(S^1_{gr} \rtimes (\mathbb{N} \times \mathbb{G}_m)))$ identifies with $S^1_{gr}-QCoh(B(\mathbb{N} \times \mathbb{G}_m)))$, which is $\epsilon-D-Mod$. We denote this element $\pi_* \mathcal{O}_X$ in $\epsilon-D-Mod^{gr}$ as $C(X)$.
\end{Cons}

\begin{Rq}
The inclusion $\epsilon-D-SCR^{gr,op} \subset \epsilon-D-dSt^{gr}$ defines functions on a Dieudonné simplicial algebra by composition
$$\epsilon-D-SCR^{gr} \to \epsilon-D-Mod^{gr}$$
\end{Rq}

\begin{Rq}
Since the pushforward $\pi_*$ is canonically lax monoidal, $C(X)$ is in fact an element of $CAlg(\epsilon-D-Mod)$. We may call $CAlg(\epsilon-D-Mod)$ the category of mixed Dieudonnée $\mathbb{E}_\infty$-algebras. This notion could give the definition of Dieudonné structures for spectral stacks but we will not explore these notions.
\end{Rq}

\begin{Prop}
The forgetful functor $\epsilon-D-SCR^{gr} \to \epsilon-D-Mod^{gr}$ commutes with filtered colimits and small limits.
\end{Prop}

\begin{proof}
Since the forgetful functor $\epsilon-D-Mod^{gr} \to Mod$ commutes with the filtered colimits and small limits and is conservative, we show the forgetful functor
$$\epsilon-D-SCR^{gr} \to Mod$$ preserves the necessary colimits and limits.

Now this functor factors as a composition
$$\epsilon-D-SCR^{gr} \to S^1_{gr}-SCR^{gr,endo} \simeq (\mathbb{N} \times \mathbb{G}_m) \ltimes S^1_{gr}-SCR \to Mod$$
where both functors commute with the desired limits and colimits.

\end{proof}

\begin{Prop} \label{FreeFrobStr}
The forgetful functor $U : dSt^{Fr} \to dSt$ admits a left adjoint denoted $L$.
\end{Prop}

\begin{Prop}
We define the two pullback squares :

\begin{center}
\begin{tikzcd}
SCR^{gr,Fr} \arrow[r]
& SCR^{gr,endo} \arrow[r] & SCR\\
CRing^{p-tf,gr,Fr} \arrow[ur, phantom, "\urcorner", very near start] \arrow[u] \arrow[r]& CRing^{p-tf,gr,endo} \arrow[u] \arrow[ur, phantom, "\urcorner", very near start] \arrow[r] \arrow[r] & CRing^{p-tf} \arrow[u]
\end{tikzcd}
\end{center}
where $CRing^{p-ft}$ is the category of discrete commutative rings which are $p$-torsion-free.

Then the morphism
$$CRing^{p-tf,gr,Fr} \to CRing^{p-tf,gr,endo}$$
is a fully faithful functor of $1$-categories.
\end{Prop}

\begin{Rq}
The previous proposition is often going to be used implicitly throughout this thesis. In our constructions, the graded simplicial algebras will have underlying weight $0$ simplicial algebras which are discrete and $p$-torsion free. Therefore the choice of a Frobenius lift on such graded simplicial algebras will simply be the choice of an endomorphism which is equal to the canonical Frobenius after reduction modulo $p$.
\end{Rq}

\begin{Def}
We define a subcategory $(\epsilon-D-SCR^{gr})_{\le 0} \subset \epsilon-D-SCR^{gr}$ on objects which are sent into $(\epsilon-D-Mod^{gr})_{\le 0}$ when applying the forgetful functor
$$U : \epsilon-D-SCR^{gr} \to \epsilon-D-Mod^{gr}$$

We define a subcategory $(\epsilon-D-SCR^{gr})_{\le 0} \subset \epsilon-D-SCR^{gr}$ on objects which are sent into $(\epsilon-D-Mod^{gr})_{\le 0}$ when applying the forgetful functor
$$U : \epsilon-D-SCR^{gr} \to \epsilon-D-Mod^{gr}$$

We will abuse terminology and call these objects coconnective, respectively connective, for a "t-structure" on the non-stable category $\epsilon-D-SCR^{gr}$.

Let us denote $(\epsilon-D-SCR^{gr})^{\heartsuit}$ the subcategory on connective and coconnective objects. We will call this category the heart of $\epsilon-D-SCR^{gr}$.
\end{Def}

\begin{Rq}
These notions might be formalized using the notion of connectivity structure on an $\infty$-category, developed in \cite{BL21}.
\end{Rq}

\begin{Prop}
The inclusion
$$(\epsilon-D-SCR^{gr})_{\le 0} \subset \epsilon-D-SCR^{gr}$$
admits a left adjoint denoted $t_{\le 0}$.

The left adjoint $t_{\le 0}$ commutes with the forgetful functor
$$U : \epsilon-D-SCR^{gr} \to \epsilon-D-Mod^{gr}$$
\end{Prop}

\begin{proof}
The category $(\epsilon-D-SCR^{gr})_{\le 0}$ is given by the pullback of categories
$$\begin{tikzcd}
(\epsilon-D-SCR^{gr})_{\le 0} \arrow[r,"i"] \arrow[d,"U_0"] \arrow[rd, phantom, "\lrcorner", very near start]
& \epsilon-D-SCR^{gr}  \arrow[d,"U"] \\
(\epsilon-D-Mod^{gr})_{\le 0} \arrow[r,"j"]
& \epsilon-D-Mod^{gr}
\end{tikzcd}$$

Using the adjoint functor theorem, see \cite[Corollary 5.5.2.9]{Lur09}, we are reduced to proving that $i$ commutes with filtered colimits and small limits, since $\epsilon-D-SCR^{gr}$ is presentable and $(\epsilon-D-SCR^{gr})_{\le 0}$ is presentable as a limit of presentable categories. The functor $i$ is a projection associated to a fiber product, hence we deduce the result from the fact that $U$ and $j$ commute with the required colimits and limits.
\end{proof}

\begin{Rq}
The subcategory of $(\epsilon-D-SCR^{gr})^{\heartsuit}$ on objects which are of $p$-torsion-free identify with the $1$-category of classical Dieudonné algebra, see Remark \ref{DADef}.
This defines a functor
$$i : \textbf{DA} \to (\epsilon-D-SCR^{gr})^{\heartsuit}$$
\end{Rq}

\begin{Rq}
As seen before, the graded derived stack with Frobenius lift $S^1_{gr}$ admits a canonical morphism
$$S^1_{gr} \to B\mathbb{G}_m$$
which has as a total space the semi-direct product
$$\mathcal{H} \coloneqq \mathbb{G}_m \ltimes S^1_{gr}$$
Therefore we can see a graded mixed Dieudonné structure on a derived stack as a $\mathbb{G}_m$ action and an action of $S^1_{gr}$ which are compatible. We have identifications
$$\epsilon-dSt^{gr} \simeq \mathcal{H}-dSt$$
\end{Rq}

 \subsubsection{Graded loopspace}

\begin{Def}
Let $X$ a derived stack endowed with a Frobenius lift. We define the category of Frobenius graded loopspaces on $X$ as
$$\mathcal{L}^{gr,Fr}(X) \coloneqq \textbf{Map}_{dSt^{Fr}}(S^1_{gr},X)$$
where $S^1_{gr}$ is endowed with its canonical Frobenius action.

The canonical point $* \to S^1_{gr}$ defined by the augmentation ${\mathbb{Z}_{(p)}}[\eta] \to {\mathbb{Z}_{(p)}}$, is a morphism of graded derived stacks with Frobenius structures and induces a morphism of graded derived stacks with Frobenius structures
$$\mathcal{L}^{gr,Fr}(X) \to X$$

We also define 
$$\mathcal{L}^{gr,endo}_{triv}(X) \coloneqq \textbf{Map}_{dSt^{endo}}(S^1_{gr},X)$$
where $S^1_{gr}$ is endowed with the trivial endomorphism structure given by identity.
\end{Def}

We recall the forgetful functor
$$U : dSt^{gr,Fr} \to dSt^{gr}$$

\begin{Prop} \label{ComputeLoopSpace} 
Let $(X,F)$ be a affine derived scheme endowed with a Frobenius structure. We write $X = Spec(C)$. The graded Frobenius loopspace's underlying graded stack identifies with the shifted linear stack:
$$U \mathcal{L}^{gr,Fr}(X) \simeq Spec_X Sym_{\mathcal{O}_X}\left(\mathbb{L}_{(X,F)}^{tw}[1]\right)$$
where $Sym$ denotes the free $(C,F)$-module construction of Proposition \ref{FreeCFMod}.

The $(C,F)$-module $\mathbb{L} \coloneqq \mathbb{L}_{(X,F)}^{tw}$ fits in a triangle of $(C,F)$-modules
$$\bigoplus_{\mathbb{N}} \mathbb{L}_{(C,F)} \otimes \mathbb{F}_p \to \mathbb{L} \to \mathbb{L}_{(C,F)}$$
\end{Prop}

\begin{proof} 
Using the description of the graded circle as a graded derived stack with a Frobenius lift of Proposition \ref{GradedCirclePushout}:
$$* \sqcup_{Spec({\mathbb{Z}_{(p)}}[\rho])} * \xrightarrow{\sim} S^1_{gr}$$ which induces an equivalence
\begin{align*}
\mathcal{L}^{gr,Fr}(X)& \xrightarrow{\sim} X \times_{\textbf{Map}_{dSt^{Fr}}(Spec({\mathbb{Z}_{(p)}}[\rho]),X)} X
\end{align*}

AS $U$ preserves limits, we deduce an equivalence of underlying graded derived stacks
\begin{align*}
U \mathcal{L}^{gr,Fr}(X)& \xrightarrow{\sim} U X \times_{U \textbf{Map}_{dSt^{Fr}}(Spec({\mathbb{Z}_{(p)}}[\rho]),X)} U X
\end{align*}

Let us compute the target of this map. We take $B$ a graded simplicial algebra endowed with an Frobenius structure and $Spec(B) \to X$ a $B$-point of $X$. Using Corollary \ref{pointsFrobdSt}, we can reduce to points on derived affine schemes with a free Frobenius lift. We may assume $B$ to be negatively graded as $\mathcal{L}^{gr,Fr}(X)$ is negatively graded. We compute the former stack at $D \coloneqq L(Spec(B))$ over $X$, where $L$ denotes the free graded derived stack with Frobenius lift construction :
\begin{align*}  
Map_{dSt^{gr,St}}(L(Spec(B)),\textbf{Map}_{dSt^{Fr}}(Spec({\mathbb{Z}_{(p)}}[\rho]),X))
\end{align*}
\begin{align*}
& \simeq Map_{dSt^{gr,Fr}_{D/}}(Spec({\mathbb{Z}_{(p)}}[\rho]) \times D,X) \\
&\simeq Map_{SCR^{gr,Fr}_{/W^{gr}(B)}}(C,  W^{gr}(B) \otimes \mathbb{Z}_{(p)}[\rho] )
\end{align*}

where we have used the following identification
$$\mathcal{O}( Spec({\mathbb{Z}_{(p)}}[\rho]) \times D ) \simeq \mathcal{O}( Spec({\mathbb{Z}_{(p)}}[\rho])) \times \mathcal{O}(D)$$
which can be seen on the underlying modules where it follows from base change, see \cite[A.1.5-(2)]{HLP14}.

Therefore we want to compute the fiber of the canonical morphism from the source space
$$Map_{SCR^{gr,endo}}(C,  R ) \times_{Map_{SCR_p^{gr,endo}}(C_p,  R(0)_p )} Map_{SCR_p^{gr}}(C_p,  R(0)_p )$$
where $R$ denotes $W^{gr}(B) \otimes \mathbb{Z}_{(p)}[\rho]$, to the target space
$$Map_{SCR^{gr,endo}}(C,  R' ) \times_{Map_{SCR_p^{gr,endo}}(C_p,  R'(0)_p )} Map_{SCR_p^{gr}}(C_p,  R'(0)_p )$$
where $R'$ denoted $W^{gr}(B)$. This fiber is simply given by the fiber of 
$$Map_{SCR^{gr,endo}}(C,  R ) \to Map_{SCR^{gr,endo}}(C,  R')$$
which is
$$Map_{SCR^{gr,endo}_{/W^{gr}(B)}}(C,  W^{gr}(B) \otimes \mathbb{Z}_{(p)}[\rho] ) \simeq Map_{SCR^{endo}_{/W^{gr}(B)(0)}}(C,  (W^{gr}(B) \otimes \mathbb{Z}_{(p)}[\rho])(0) )$$
as $C$ is concentrated in degree $0$. Explicitly, we have
$$(W^{gr}(B) \otimes \mathbb{Z}_{(p)}[\rho])(0) \simeq W(B(0)) \oplus B_{-1}^{\mathbb{N}}$$
where $W(B(0))$ is endowed by its canonical endomorphism and $B_{-1}^{\mathbb{N}}$ is endowed by its twisted endomorphism $pS$, by taking into account the endomorphism of $\mathbb{Z}[\rho]$ sending $\rho$ to $p \rho$. Therefore, the fiber is computed as

$$ Map_{SCR^{endo}_{/W^(B(0))}}(C, W(B(0)) \oplus B_{-1}^{\mathbb{N}}) \simeq Map_{(C,F)-Mod}(\mathbb{L}_{(C,F)},(B_{-1}^{\mathbb{N}},pS))$$

\begin{Le} 
Let $(M,u)$ and $(N,v)$ be $(C,F,h)$-modules, the fiber of
$$Map_{(A,F,h)-Mod^{Fr}}((M,u),(N,v)) \to Map_{(A,F)-Mod^{endo}}((M,u),(N,v))$$
is given by
$$Map_{C-Mod}(M_p,N)$$
\end{Le}

\begin{proof} 
Using Remark \ref{AFhmodDescr}, the mapping space $Map_{(A,F,h)-Mod^{Fr}}((M,u),(N,v))$ is given
$$Map_{(A,F)-Mod^{endo}}((M,u),(N,v)) \times_{Map_{(A_p,F_p)-Mod^{endo}}((M_p,u_p),(N_p,v_p))} Map_{A_p-Mod}(M_p,N_p)$$

Therefore the fiber of the map
$$Map_{(A,F,h)-Mod^{Fr}}((M,u),(N,v)) \to Map_{(A,F)-Mod^{endo}}((M,u),(N,v))$$
coincides with the fiber of 
$$Map_{A_p-Mod}(M_p,N_p) \to Map_{(A_p,F_p)-Mod^{endo}}((M_p,u_p),(N_p,v_p))$$

Now since $M$ and $N$ have $(A,F,h)$-structures, $u_p$ and $v_p$ are homotopic to the zero endomorphisms. Therefore we have an identification
$$Map_{(A_p,F_p)-Mod^{endo}}((M_p,u_p),(N_p,v_p)) \simeq Map_{A_p-Mod}(M_p,N_p) \times Map_{A_p-Mod}(M_p[1],N_p)$$

We deduce the fiber is given by
$$\Omega Map_{A_p-Mod}(M_p[1],N_p) \simeq Map_{A-Mod}(M_p[1],N)$$
using the identification
$$\underline{Hom}_{A-Mod}(\mathbb{F}_p,N) \simeq N[-1]$$
\end{proof}

We deduce that the fiber of 
$$Map_{(C,F,h)-Mod}((\mathbb{L}_{(C,F)},dF),(B_{-1}^{\mathbb{N}},pS)) \to Map_{(C,F)-Mod}(\mathbb{L}_{(C,F)},(B_{-1}^{\mathbb{N}},pS))$$
is given by
$$Map_{C-Mod}(\mathbb{L}_{(C,F)} \otimes \mathbb{F}_p [1], B_{-1}^{\mathbb{N}})$$
which is also given by
$$Map_{(C,F)-Mod}((\bigoplus_{\mathbb{N}} \mathbb{L}_{(C,F)} \otimes \mathbb{F}_p [1],S), (B_{-1}^{\mathbb{N}},S))$$

Using Proposition \ref{FrEndoModequi}, we have a natural identification
$$Map_{(C,F,h)-Mod}((\mathbb{L}_{(C,F)},dF),(B_{-1}^{\mathbb{N}},pS)) \simeq Map_{(C,F)-Mod}((\mathbb{L}_{(C,F)},\frac{dF}{p}),(B_{-1}^{\mathbb{N}},S))$$

Therefore there exists $\mathbb{L}$ which fits in a triangle of $(C,F)$-modules
$$\bigoplus_{\mathbb{N}} \mathbb{L}_{(C,F)} \otimes \mathbb{F}_p \to \mathbb{L} \to \mathbb{L}_C$$
such that
$$Map_{(C,F)-Mod}(\mathbb{L}_{(C,F)},(B_{-1}^{\mathbb{N}},pS)) \simeq Map_{(C,F)-Mod}(\mathbb{L},(B_{-1}^{\mathbb{N}},S))$$

Then we deduce the equivalence
$$\mathcal{L}^{gr,Fr}(X) \xrightarrow{\sim} X \times_{\textbf{Map}_{dSt^{Fr}}(Spec({\mathbb{Z}_{(p)}}[\rho]),X)} X \simeq Spec(Sym_{(C,F)}(\mathbb{L}[1]))$$

\end{proof}

\begin{Cor}  \label{CorGrLoopSpaceUnderlying}
With the notation of the previous proposition, the derived stack associated to $\mathcal{L}^{gr,Fr}(X)$ is given, after moding out by the $p$-torsion, by
$$Spec(Sym(\mathbb{L}_C[1]))$$
where the induced endomorphism is induced by
$$\frac{dF}{p} : \mathbb{L}_C[1] \to \mathbb{L}_C[1]$$
\end{Cor}

\begin{proof}
The underlying derived stack of $\mathcal{L}^{gr,Fr}(X)$ is given by
$$Spec(Sym(\mathbb{L}[1]))$$
From Remark \ref{CotangEndo}, the cotangent complex $\mathbb{L}_{(C,F)}$ has $\mathbb{L}_C$ as an underlying $C$-module. In the description of the twisted cotangent complex in a triangle of $(C,F)$-modules
$$\bigoplus_{\mathbb{N}} \mathbb{L}_{C} \otimes \mathbb{F}_p \to \mathbb{L} \to \mathbb{L}_C$$
where $\mathbb{L}_C$ is endowed with the endomorphism $\frac{dF}{p}$, we notice this sequence is split since the first morphism is null. This concludes the proof.
\end{proof}

The proof of Proposition \ref{ComputeLoopSpace} can easily be adapted to prove the following result.

\begin{Rq}
The trivial graded endomorphism loopspace identifies with the shifted tangent stack of $X$ :
$$\mathcal{L}^{gr,endo}_{triv}(X) \simeq Spec_X Sym_{\mathcal{O}_X}((\mathbb{L}_{(X,F)},dF)[1])$$
where $X$ is no longer required to have a $p$-torsion free cotangent complex.
\end{Rq}

\subsection{Comparison theorems}

\subsubsection{Mixed structures classification : the non-Dieudonné case}

In this section, we recall a theorem of \cite[Proposition 2.3.1]{To20}, which was announced with an outlined proof. We give a detailed proof so as to generalize to a Dieudonné variant of this theorem.

\begin{Thm} \label{classifClassique}

Let $A$ be smooth commutative ${\mathbb{Z}_{(p)}}$-algebra, $M$ be a projective $A$-module of finite type. We define $X$ as the derived affine scheme $Spec(Sym_A(M[1]))=\mathbb{V}(M[1])$ endowed with its natural grading. The classifying space of mixed graded structures on $X$ compatible with its grading is discrete and in bijection with the set of commutative differential graded algebra structures on the graded commutative ${\mathbb{Z}_{(p)}}$-algebra $\bigoplus_i \wedge^i_A M [-i]$.

\end{Thm}

\begin{proof}

The classifying space of mixed structures is given by the fiber of 
the forgetful functor
$$(\mathbb{G}_m \ltimes S^1_{gr})-dSt \to \mathbb{G}_m-dSt$$
over $X$.

Since $(\mathbb{G}_m \ltimes S^1_{gr})-dSt$ identifies naturally with $S^1_{gr}-(\mathbb{G}_m-dSt)$, the classifying space is given by the mapping space $$Map_{Mon(dSt^{gr})}(S^1_{gr},\textbf{End}_{gr}(X))$$ which, by connexity of $S^1_{gr}$, is also $$Map_{Mon(dSt^{gr})}(S^1_{gr},\textbf{End}^0_{gr}(X))$$ where we define

$$\textbf{End}^0_{gr}(X) \coloneqq fib(\textbf{End}_{gr}(X) \to \pi_0(\textbf{End}^0_{gr}(X)))$$
the fiber over the identity, meaning $\textbf{End}^0_{gr}(X)$ is the substack on $\textbf{End}_{gr}(X)$ on endomorphisms that are homotopy equivalent to identity. This space is equivalent to the mapping space of pointed stacks $$Map_{dSt^{gr,*}}(BS^1_{gr},B\textbf{End}^0_{gr}(X))$$ Since $BS^1_{gr}$ is a stack, we may consider $B\textbf{End}^0(X)$ as a stack in $St \subset dSt$. Therefore, we are reduced to the computation of
$$Map_{St^{gr,*}}(BS^1_{gr},B\textbf{End}^0_{gr}(X))$$

We will need a Postnikov decomposition on $B\textbf{End}^0_{gr}(X)$, therefore we have to compute its homotopy groups. To study the behaviour on the base scheme $S=Spec(A)$, we notice that we have a fiber sequence of graded derived stacks :

\begin{center}
\begin{tikzcd}
\textbf{End}_{gr,S}^0(X) \arrow[d] \arrow[r] \arrow[dr, phantom, "\lrcorner", very near start] & B\mathbb{G}_m \arrow[d] \\
\textbf{End}^0_{gr}(X) \arrow[r]& \textbf{Hom}^0_{gr}(X,S)
\end{tikzcd}
\end{center}

where $B\mathbb{G}_m$ is the final object in $St^{gr} = \bigslant{St}{B\mathbb{G}_m}$, $\textbf{Hom}_{gr,S}(X)$ is the graded derived stack of endomorphisms of $X$ over $S$ and $\textbf{Hom}_{S,gr}^0(X)$ is the substack of maps that are equivalent to the canonical projection $X \to S$. In this diagram, the bottom arrow is the composition with the projection $X \to S$ and the left arrow is the inclusion.

Being a relative linear stack, $X \to S$ has a section : the "zero" section. Therefore $$\textbf{End}^0(X) \to \textbf{Hom}^0(X,S)$$ is surjective on all homotopy groups and the long exact sequence of homotopy groups splits as short exact sequences of graded sheaves of abelian groups on $S$
$$0 \to \pi_k \textbf{End}_S^0(X) \to \pi_k \textbf{End}^0(X) \to \pi_k \textbf{Hom}^0(X,S) \to 0$$

\paragraph*{Computation of $\pi_k \textbf{End}^0_S(X)$} Since all the derived stacks are now truncated, we can take $B$ a test graded commutative ${\mathbb{Z}_{(p)}}$-algebra and $Spec(B) \to S$ a $B$-point of $S$, which is an $A$-algebra structure on $B$, ie $B$ is a relative affine scheme over $S \times B\mathbb{G}_m$. We compute, for $k>0$ :

\begin{align*}
\textbf{End}_{S,gr}(X)(B)&= Map_{SCR^{gr}_A}(Sym_A(M[1]),Sym_A(M[1]) \otimes_A B)\\
&= Map_{A-Mod^{gr}}(M[1],Sym_A(M[1]) \otimes_A B)
\end{align*}
From which we deduce

\begin{align*}
\pi_k \textbf{End}^0_{S,gr}(X)&= \pi_k \textbf{End}_{S,gr}(X)\\
&= Hom_{A-Mod^{gr}}(M[k+1],Sym_A(M[1]) \otimes_A B)\\
&= Hom_{A-Mod^{gr}}(M,\wedge^{k+1} M \otimes_A B)\\
&= Hom_{A-Mod^{gr}}(M \otimes_A (\wedge^{k+1} M)^\vee,  B)
\end{align*}

Therefore $\pi_k \textbf{End}^0_S(X) \simeq \mathbb{V}(M \otimes (\wedge^{k+1} M)^\vee)$

\paragraph*{Computation of $\pi_k Hom^0(X,S)$} We have on $B$-points :
$$\textbf{Hom}^0(X,S)(B) = Map_{SCR}^0(A,Sym_A(M[1]) \otimes B)$$
where the $0$ exponent denotes the connected component of the canonical map
$$A \to Sym_A(M[1]) \to Sym_A(M[1]) \otimes B$$

We can recover the homotopy groups from the Postnikov decomposition :
$$\pi_k \textbf{Hom}^0(X,S)(B)[k] = fib(t_{\le k} \textbf{Hom}^0(X,S)(B) \to t_{\le k-1} \textbf{Hom}^0(X,S)(B))$$

As $Map(A,-)$ preserves Postnikov decompositions, $\pi_k \textbf{Hom}^0(X,S)(B)[k]$ is given by the fiber

$$fib(Map_{SCR}^0(A,Sym_A^{\le k}(M[1]) \otimes B) \to Map_{SCR}^0(A,Sym_A^{\le k-1}(M[1]) \otimes B)) $$

Which is simply $Map_{SCR_A}^0(\Omega_A^1, \Lambda^k M[k] \otimes B)$

We find $\pi_k \textbf{Hom}^0(X,S)= \mathbb{V}(\Omega_A^1 \otimes_A (\Lambda^k M)^\vee) $

When $k$ is strictly greater than the rank of $M$ as an $A$-module, $\Lambda^k M$ and $\Lambda^{k+1} M$ both vanish, therefore $\pi_k=0$, we deduce $B\textbf{End}^0(X)$ is $(n+1)$-truncated.

We will, in fact, need a more precise description of $\pi_k \coloneqq \pi_k \textbf{End}^0_{gr}(X)$ :
\begin{Prop} \label{piEnd}
For $k \ge 1$, the sheaf of groups $\pi_k$ is given on a commutative algebra $B$ by the discrete set of pairs $(d_0,d_1)$ with $d_0 : A_B \to \bigwedge^k_{A_B}M_B$ a derivation and $d_1 : M_B \to \bigwedge^{k+1}_{A_B} M_B$ a ${\mathbb{Z}_{(p)}}$-linear map, with the compatibility condition
$$d_1(am) = a d_1(m) + d_0(a) \wedge m$$
for $a \in A_B$ and $m \in M_B$.

In this proposition, we have defined $A_B \coloneqq A \otimes_{\mathbb{Z}_{(p)}} B$ and $M_B \coloneqq M \otimes_{\mathbb{Z}_{(p)}} B$.
\end{Prop}

\begin{proof}
On a $B$ point, 
$$Map_{SCR}(Sym_A(M[1]),Sym_A(M[1]))(B) \simeq Map_{SCR_B}(Sym_{A_B}(M_B[1]),Sym_{A_B}(M_B[1]))$$

We want to compute
$$\pi_k Map_{SCR_B}^0(Sym_{A_B}(M[1]),Sym_{A_B}(M_B[1]) )=\pi_k Map_{SCR_B}(Sym_{A_B}(M_B[1]),Sym_{A_B}(M_B[1]) )$$
which is given by the fiber over the identity of
$$\pi_0( Map(S^k,Map_{SCR_B}(C_B,C_B)) \to Map_{SCR_B}(C_B,C_B) )$$
where $C_B$ denotes $Sym_{A_B}(M_B[1])$.

Since $Map(S^k,Map_{SCR_B}(C_B,C_B))$ can be computed as $Map_{SCR_B}(C_B,C_B^{S^k})$, we can rewrite
$$\pi_k Map_{SCR_B}^0(Sym_{A_B}(M_B[1]),Sym_{A_B}(M_B[1]) ) \simeq Hom_{SCR_B/C_B}(C_B,C_B^{S^k})$$

We will use the following lemma.

\begin{Le}
Let $D$ be a simplicial commutative algebra, we have a natural identification
$$\pi_0 Map_{SCR}(Sym_A(M[1]),D) \simeq \left \{ (u,v): u : A \to \pi_0(D) \in CRing, v : M \to \pi_1(D) \in A-Mod \right \}$$
where $\pi_1(D)$ is endowed with the $A$-module structure induced by $u$.
\end{Le}

\begin{proof}
We explicitly construct this bijection. An element
$$f \in \pi_0 Map_{SCR}(Sym_A(M[1]),D)$$
induces by composition a morphism of commutative rings
$$A \to Sym_A(M[1]) \xrightarrow{f} D \to \pi_0(D)$$
which we call $u$. It also induces
$$M[1] \to  Sym_A(M[1]) \to D$$
which defines $v$ after taking $\pi_1$.

On the other hand, let us give ourselves a pair $(u,v)$ satisfying the required conditions. By smoothness of $A$, $u$ extends to a morphism of simplicial rings
$$a : A \to D$$
Now giving a morphism $Sym_A(M[1]) \to D$ under $A$ is equivalent to the data of an $A$-module morphism
$$M[1] \to D$$

We check that these maps are mutually inverse.

\end{proof}

The lemma can be used for the base-changed versions of $A$ and $M$ : $A_B$ and $M_B$. Taking $D=Sym_{A_B}(M_B[1])^{S^k}$, we know that
$$\pi_0(D) = A_B \oplus \Lambda^k M_B$$
and
$$\pi_1(D) = M_B \oplus \Lambda^{k+1} M_B$$

Now the set $$Hom_{SCR_B/C_B}(C_B,C_B^{S^k})$$ is identified with the set of pairs of $(\delta,d)$ with $\delta$ being a morphism of simplicial algebras
$$\delta : A_B \to A_B \oplus \Lambda^k M_B$$
and $d$ being a morphism of $A$-modules
$$d : M_B \to M_B \oplus \Lambda^{k+1} M_B$$
such that they induce the identity $C_B \to C_B$ after composition. Therefore $\delta$ is a derivation
$$\delta(a) = a + d_0(a)$$
and $d$ is given by
$$d(m) = m + d_1(m)$$
Requiring $d$ the be $A$-linear gives the compatibility condition on $(d_0,d_1)$.

\end{proof}

\paragraph*{Cell decomposition on $BS^1_{gr}$} : 
We now define a variation of cellular decomposition for $BS^1_{gr}$.

The classical cellular decomposition of the topological space $BS^1 \simeq \mathbb{C}P^\infty$, from its CW complex description, gives a tower of topological spaces
$$ (BS^1)_{\le 0} = * \to (BS^1)_{\le 1} = * \to (BS^1)_{\le 2} \simeq S^2 \to (BS^1)_{\le 3} \simeq S^2 \to (BS^1)_{\le 4} \to ...$$

The tower is given by iterated homotopy pushouts :
$$\begin{tikzcd}
(BS^1)_{\le 2n} \arrow[r]
& (BS^1)_{\le 2n+2} \arrow[ld, phantom, "\llcorner", very near start] \\
S^{2n+1} \arrow[u] \arrow[r]
& D^{2n+2} \simeq * \arrow[u]
\end{tikzcd}$$

We want a similar decomposition for $BS^1_{gr}$ as a graded affine stack.

\begin{Def}
Let us define a sequence of affine stacks $(BS^1_{gr})_{\le n}$ for $n \ge 0$ :
$$ (BS^1_{gr})_{\le 2n} = (BS^1_{gr})_{\le 2n+1} \coloneqq Spec^{\Delta}({\mathbb{Z}_{(p)}}[u]/(u^{n+1}))$$
where $u$ is in weight $1$ and degree $2$.

The canonical projection ${\mathbb{Z}_{(p)}}[u]/(u^{n+2}) \to {\mathbb{Z}_{(p)}}[u]/(u^{n+1})$ induces a morphism of derived stack $(BS^1_{gr})_{\le 2n} \to (BS^1_{gr})_{\le 2n+2}$.
\end{Def}

\begin{Prop}
There is a homotopy pushout diagram of graded affine stacks

$$\begin{tikzcd}
(BS^1_{gr})_{\le 2n} \arrow[r]
& (BS^1_{gr})_{\le 2n+2} \arrow[ld, phantom, "\llcorner", very near start] \\
S^{2n+1}_{gr}(n+1) \arrow[u] \arrow[r]
& * \arrow[u]
\end{tikzcd}$$

where $S^{2n+1}_{gr}(n+1)$ is defined as $Spec^{\Delta}({\mathbb{Z}_{(p)}} \oplus {\mathbb{Z}_{(p)}}[-2n-1]((n+1)) )$.
\end{Prop}

\begin{Rq}
The diagram is shown to be a pushout diagram in affine stacks, not in derived stacks.
\end{Rq}

\begin{proof}
This pushout diagram is equivalent to a pullback diagram in graded cosimplicial algebras :
$$\begin{tikzcd}
{\mathbb{Z}_{(p)}}[u]/(u^{n+1}) \arrow[d]
& {\mathbb{Z}_{(p)}}[u]/(u^{n+2}) \arrow[ld, phantom, "\llcorner", very near start] \arrow[d] \arrow[l] \\
{\mathbb{Z}_{(p)}} \oplus {\mathbb{Z}_{(p)}}[-2n-1]((n+1)) 
& {\mathbb{Z}_{(p)}} \arrow[l]
\end{tikzcd}$$

We define a diagram of commutative differential graded algebras :
$$\begin{tikzcd}
{\mathbb{Z}_{(p)}}[u,v]/(u^{n+2},dv=u^{n+1}) \arrow[r,"f_1","\sim"'] \arrow[rd,"f_2"']& {\mathbb{Z}_{(p)}}[u]/(u^{n+1})
& {\mathbb{Z}_{(p)}}[u]/(u^{n+2}) \arrow[d] \arrow[l] \\
&{\mathbb{Z}_{(p)}} \oplus {\mathbb{Z}_{(p)}}[-2n-1]((n+1)) 
& {\mathbb{Z}_{(p)}} \arrow[l]
\end{tikzcd}$$
where $f_1$ is a quasi-isomorphism sending $u$ to $u$ and $v$ to $0$ and $g$ sends $u$ to $0$ and $v$ to $\epsilon_{2n-1}$, $v$ being in degree $2n-1$ and weight $n+1$.

Let us denote $T$ the fiber of
$${\mathbb{Z}_{(p)}}[u,v]/(u^{n+2},dv=u^{n+1}) \to {\mathbb{Z}_{(p)}} \oplus {\mathbb{Z}_{(p)}}[-2n-1]((n+1))$$
over ${\mathbb{Z}_{(p)}}$. We check that sending $u$ to $u$ defines a quasi-isomorphism
$${\mathbb{Z}_{(p)}}[u]/(u^{n+2}) \xrightarrow{\sim} T$$
which recovers the canonical projection
$${\mathbb{Z}_{(p)}}[u]/(u^{n+2}) \to {\mathbb{Z}_{(p)}}[u]/(u^{n+1})$$
after composing with $f_1$.

Therefore this diagram in homotopy pullback diagram on the underlying complexes : it is strictly a pullback diagram and $f_2$ is a fibration.

Since denormalization $D$ respects quasi-isomorphisms, applying $D$ yields the required diagram, which is a homotopy pullback diagram since it is on the underlying complexes and the forgetful functor from cosimplicial algebras to complexes is conservative.

\end{proof}

\begin{Le}
The diagram on $(\mathbb{N},\le)$ defined by the tower $(BS^1_{gr})_{\le n}$ admits $BS^1_{gr}$ as a colimit in affine stacks.
\end{Le}

\begin{proof}

The canonical projections ${\mathbb{Z}_{(p)}}[u] \to {\mathbb{Z}_{(p)}}[u]/(u^{n+1})$, where $u$ is in weight $1$ and cohomological degree $2$, are morphisms of graded cosimplicial algebras. This defines a morphism of graded cosimplicial algebras
$${\mathbb{Z}_{(p)}}[u] \to lim_{n} {\mathbb{Z}_{(p)}}[u]/(u^{n+1})$$
which is an equivalence on the underlying complexes, therefore it is an equivalence.

Passing to $Spec^{\Delta}$ yields the result, using Proposition \ref{BS1grequi}.

\end{proof} 

As we want to compute $Map_{St^{gr,*}}(BS^1_{gr},B\textbf{End}^0_{gr}(X))$, we first need a result :

\begin{Prop} \label{PropPushAgainst}
The commutative diagram, denoted $(\star)$, 
$$\begin{tikzcd}
(BS^1_{gr})_{\le 2n} \arrow[r]
& (BS^1_{gr})_{\le 2n+2} \arrow[ld, phantom, "\llcorner", very near start] \\
S^{2n+1}_{gr}(n+1) \arrow[u] \arrow[r]
& * \arrow[u]
\end{tikzcd}$$
is a pushout against $B\textbf{End}_{gr}^0(X)$, meaning
$$\begin{tikzcd}
Map_{St^{gr,*}}((BS^1_{gr})_{\le 2n},B\textbf{End}_{gr}^0(X)) \arrow[d]
& Map_{St^{gr,*}}((BS^1_{gr})_{\le 2n+2},B\textbf{End}_{gr}^0(X)) \arrow[l] \arrow[d] \\
Map_{St^{gr,*}}(S^{2n+1}_{gr}(n+1),B\textbf{End}_{gr}^0(X))  \arrow[ru, phantom, "\llcorner", very near end]
& * \arrow[l]
\end{tikzcd}$$
is a pullback diagram.
\end{Prop}

The result is not obvious as $B\textbf{End}_{gr}^0(X)$ need not be an affine stack.

\begin{Le} \label{QCohCell}
The diagram $QCoh(\star)$ is a pullback diagram, ie the following diagram is a pullback diagram.

\
$$\begin{tikzcd}
QCoh( (BS^1_{gr})_{\le 2n} ) \arrow[d]
& QCoh( (BS^1_{gr})_{\le 2n+2} ) \arrow[l] \arrow[d] \\
QCoh( S^{2n+1}_{gr}(n+1) )  \arrow[ru, phantom, "\llcorner", very near end]
& QCoh(*) \arrow[l]
\end{tikzcd}$$
\end{Le}

\begin{proof}
We first show the canonical morphism
$$QCoh( S^{2n+1}_{gr}(n+1) ) \to QCoh( (BS^1_{gr})_{\le 2n} ) \times_{QCoh( S^{2n+1}_{gr}(n+1) )} QCoh(*)$$
is an equivalence on bounded complexes.

We use the canonical t-structure on $QCoh(A)$ for $A$ an affine stack defined in Remark \ref{tstrQCoh}. With this dévissage argument, we are reduced to showing the morphism is an equivalence on hearts and on extension groups between objects in the heart.

The $1$-connectivity of the stacks, see Proposition \ref{conneAffSt}, identifies their quasi-coherent complexes with $Mod_{\mathbb{Z}_{(p)}}$. Therefore we have an equivalence on the heart.

We invoke Proposition \ref{extAffine}, we deduce
$$Map_{QCoh((BS^1_{gr})_{\le 2n})}(M,N) \simeq C((BS^1_{gr})_{\le 2n} ) \otimes Map_{Mod}(M,N)$$

Similar results for $QCoh( S^{2n+1}_{gr}(n+1) )$, $QCoh( S^{2n+1}_{gr}(n+1) )$ and $Mod$ hold. We deduce the result from the pullback diagram of $C (\star)$. Now, adding small limits gives the equivalence on eventually coconnective complexes, then left-completing the category with respect to the t-structure recovers the categories of quasi-coherent complexes and yields the equivalence, see \cite[Notation 1.2.12]{MRT20} and \cite[\S 6.2]{Lur18}.

\end{proof}

\begin{Le} \label{DiagPushLe}
Let $F$ be a stack such that
\begin{itemize}
\item $F$ is $1$-connective, ie $\pi_0(F) =*$ and $\pi_1(F)=*$.
\item $F$ is $n$-truncated for some $n$.
\item $\pi_i(F)$ is quasi-coherent for all $i$.
\end{itemize}
then the pushout diagram $(\star)$ is a pushout diagram against $F$.
\end{Le}

\begin{proof}
The Postnikov tower of $F$ exhibits $F_{\le i+1}$ as the fiber of
$$F_{\le i} \to K(\pi_{i+1},i+2)$$

Therefore, we can assume $F$ be an Eilenberg-MacLane stack $K(\pi,i)$. We have, for $Y$ a stack,
$$Map_{St}(Y,F)= Map_{QCoh(Y)}(\mathcal{O}_Y,p^*\pi[i])$$
where $p : Y \to Spec({\mathbb{Z}_{(p)}})$ is the canonical projection. Now using Lemma $\ref{QCohCell}$ and the fact that mapping spaces of fiber products of categories are given by fiber products of the mapping spaces, we deduce that
$$Map_{St}(-,F)$$
sends $\star$ to a pullback diagram.
\end{proof}

\begin{proof}[Proof of Proposition \ref{PropPushAgainst} ]
We combine \ref{QCohCell}, \ref{DiagPushLe} and the fact that $B\textbf{End}_{gr}^0(X)$ is truncated.
\end{proof}

We now move on to the computation of $Map_{St^{gr,*}}(BS^1_{gr},B\textbf{End}_{gr}^0(X))$ using Proposition \ref{PropPushAgainst}. We will compute $Map_{St^{gr,*}}(S^{2n+1}_f(n+1),B\textbf{End}^0_{gr}(X))$ using the Postnikov decomposition of $B\textbf{End}_{gr}^0(X)$. By convention, $\pi_n$ is $\pi_n(\textbf{End}_{gr}^0(X))$, it is a graded abelian sheaf on $\textbf{End}_{gr}^0(X)$. We will denote $\Gamma(\pi_n)$ the graded abelian group of global sections and $\Gamma^{\mathbb{G}_m}(\pi_n)$ the abelian group of weight $0$ functions.

We need a lemma to compute $Map_{St^{gr,*}}(S^1_{gr},B\textbf{End}^0_{gr}(X))$.

\begin{Le} \label{computeMapSphere}
Using the previous notations, we have an equivalence of spaces
$$Map_{St^{gr,*}}(S^{2n+1}_f(n+1),B\textbf{End}^0_{gr}(X)) \simeq 
\left \{
\begin{array}{ll}
B \Gamma^{\mathbb{G}_m}(\pi_1) & \text{if } n=0 \\
\Gamma^{\mathbb{G}_m}(\pi_2) & \text{if } n=1 \\
0 & \text{if } n>1
\end{array}
\right.$$

\end{Le}

\begin{proof}
The Postnikov tower exhibits the graded stack $(B\textbf{End}_{gr}^0(X))_{\le m}$ as a homotopy fiber
$$(B\textbf{End}_{gr}^0(X))_{\le m} \to
(B\textbf{End}_{gr}^0(X))_{\le m-1} \to
K(\pi_{m-1},m+1)$$

We deduce a triangle of topological spaces
$$\begin{tikzcd}
Map_{St^{gr,*}}(S^{2n+1}_f(n+1),(B\textbf{End}_{gr}^0(X))_{\le m}) \arrow[d] \\ Map_{St^{gr,*}}(S^{2n+1}_f(n+1),(B\textbf{End}_{gr}^0(X))_{\le m-1})\arrow[d] \\
 Map_{St^{gr,*}}(S^{2n+1}_f(n+1),K(\pi_{m-1},m+1))
\end{tikzcd}$$

Now, 
$$\pi_k Map_{St^{gr,*}}(S^{2n+1}_f(n+1),K(\pi_{m-1},m+1))$$
can be computed, following \cite[\S 1.3]{To06} as
$$\pi_k \Gamma(B\mathbb{G}_m,p_* p^* \pi_{m-1}[m+1])$$
where $p: S^{2n+1}_f(n+1) \to B\mathbb{G}_m$ is the structure morphism. The later is simply given by

$$ H_{\mathbb{G}_m}^{m+1-k}(S^{2n+1}_f(n+1),\pi_{m-1})$$

which is trivial when $n+1 \neq m-1$ as $\pi_{m-1}$ is in weight $m-1$. When $m=n+2$,
$$H_{\mathbb{G}_m}^{n+3-k}(S^{2n+1}_f(n+1),\pi_{n+1})$$
is non-trivial only when $k=2-n$, then it is simply $\Gamma^{\mathbb{G}_m}(\pi_{n+1})$.

We deduce
$$Map_{St^{gr,*}}(S^{1}_f(1),K(\pi_{1},3)) \simeq K(\Gamma^{\mathbb{G}_m}(\pi_1),2)$$

$$Map_{St^{gr,*}}(S^{3}_f(2),K(\pi_{2},4)) \simeq K(\Gamma^{\mathbb{G}_m}(\pi_2),1)$$

$$Map_{St^{gr,*}}(S^{5}_f(3),K(\pi_{3},5)) \simeq \Gamma^{\mathbb{G}_m}(\pi_3)$$

Using the fact that $B\textbf{End}^0_{gr}(X)$ is truncated, therefore its Postnikov tower converges, we have

$$Map_{St^{gr,*}}(S^{2n+1}_f(n+1),B\textbf{End}^0_{gr}(X)) \simeq \Omega_*(Map_{St^{gr,*}}(S^{2n+1}_f(n+1),K(\pi_{n+1},n+3)))$$
which yields the required result.

\end{proof}

\begin{Le}
For $n>1$,
$$Map_{St^{gr,*}}((BS^1_{gr})_{\le 2n+2},B\textbf{End}^0_{gr}(X)) \to Map_{St^{gr,*}}((BS^1_{gr})_{\le 2n} ,B\textbf{End}^0_{gr}(X))$$
is an equivalence.
\end{Le}

\begin{proof}
The result is easily deduced from Proposition \ref{PropPushAgainst} and Lemma \ref{computeMapSphere}.
\end{proof}

We need to compute the fiber product

$$\begin{tikzcd}
Map_{St^{gr,*}}((BS^1_{gr})_{\le 2},B\textbf{End}^0_{gr}(X)) \arrow[d] & Map_{St^{gr,*}}((BS^1_{gr})_{\le 4},B\textbf{End}^0_{gr}(X)) \arrow[d] \arrow[l] \arrow[dl, phantom, "\llcorner", very near start] \\
Map_{St^{gr,*}}(S^3(2),B\textbf{End}^0_{gr}(X)) & * \arrow[l]
\end{tikzcd}$$

that is

$$\begin{tikzcd}
\Gamma^{\mathbb{G}_m}(\pi_1) \arrow[d] & Map_{St^{gr,*}}((BS^1_{gr})_{\le 4},B\textbf{End}^0_{gr}(X)) \arrow[d] \arrow[l] \arrow[dl, phantom, "\llcorner", very near start] \\
\Gamma^{\mathbb{G}_m}(\pi_2) & * \arrow[l]
\end{tikzcd}$$

Intuitively, we want $\Gamma^{\mathbb{G}_m}(\pi_1) \mapsto \Gamma^{\mathbb{G}_m}(\pi_2)$ to send a pair $(\delta,d)$, with $\delta : A \to M$ and $d : M \to \Lambda M$, to its "composition with itself". The equation $d \circ d = 0$ involves $Sym_A(M[1])$ as a complex, without the full simplicial algebra structure. We first make the following reduction.

\begin{Le}
The canonical commutative diagram
$$\begin{tikzcd}
Map_{St^{gr,*}}( (BS^1)_{\le 4},B\textbf{End}^0_{SCR}(X) ) \arrow[d] \arrow[r] & Map_{St^{gr,*}}( (BS^1)_{\le 4},B\textbf{End}_{Mod}(X) ) \arrow[d] \\
Map_{St^{gr,*}}( S^2_f(1),B\textbf{End}^0_{SCR}(X) ) \arrow[r] & Map_{St^{gr,*}}( S^2_f(1),B\textbf{End}_{Mod}(X) )
\end{tikzcd}$$
is a pullback diagram.
\end{Le}

\begin{proof}
We have the following commutative diagram

$$\begin{tikzcd}
Map_{St^{gr,*}}( (BS^1)_{\le 4},B\textbf{End}^0_{SCR}(X) ) \arrow[d] \arrow[r] & Map_{St^{gr,*}}( (BS^1)_{\le 4},B\textbf{End}_{Mod}(X) ) \arrow[d] \\
Map_{St^{gr,*}}( S^2_f(1),B\textbf{End}^0_{SCR}(X) ) \arrow[d] \arrow[r] & Map_{St^{gr,*}}( S^2_f(1),B\textbf{End}_{Mod}(X) )\arrow[d] \\
Map_{St^{gr,*}}( S^3_f(2),B\textbf{End}^0_{SCR}(X) ) \arrow[r] & Map_{St^{gr,*}}( S^3_f(2),B\textbf{End}_{Mod}(X) ) 
\end{tikzcd}$$

showing a natural transformation of triangles of spaces. Formality of $\textbf{End}_{Mod}(X)$ allows for a computation of $Map_{St^{gr,*}}( S^3_f(2),B\textbf{End}_{Mod}(X) )$ and $Map_{St^{gr,*}}( S^2_f(1),B\textbf{End}_{Mod}(X) ) $ using similar methods as before. We compute the homotopy groups of $\textbf{End}^0_{Mod}(X)$ :
$$\Gamma^{\mathbb{G}_m}(\pi_k(\textbf{End}^0_{Mod}(X))) = \bigoplus_n Hom(\Lambda^n M,\Lambda^{n+k} M)$$

A Postnikov decomposition argument yields
$$ Map_{St^{gr,*}}( S^3_f(2),B\textbf{End}_{Mod}(X) ) \simeq \Gamma^{\mathbb{G}_m}(\pi_2 (\textbf{End}^0_{Mod}(X) )) = \bigoplus_n Hom(\Lambda^n M,\Lambda^{n+2} M)$$
is discrete. Similarly,
$$ Map_{St^{gr,*}}( S^2_f(1),B\textbf{End}_{Mod}(X) ) \simeq \Gamma^{\mathbb{G}_m}(\pi_1 (\textbf{End}^0_{Mod}(X) )) = \bigoplus_n Hom(\Lambda^n M,\Lambda^{n+1} M)$$
is also discrete. To show the commutative diagram of the lemma is a pullback diagram, we simply show
$$Map_{St^{gr,*}}( S^3_f(2),B\textbf{End}^0_{SCR}(X) )  \to Map_{St^{gr,*}}( S^3_f(2),B\textbf{End}_{Mod}(X) ) $$
is injective. But this map is the inclusion
$$\Gamma^{\mathbb{G}_m}(\pi_2) \subset  \bigoplus_n Hom(\Lambda^n M,\Lambda^{n+2} M)$$

\end{proof}

We need to compute the fiber of the morphism
$$Map_{St^{gr,*}}( S^2_f(1),B\textbf{End}_{Mod}(X) ) \to Map_{St^{gr,*}}( S^3_f(2),B\textbf{End}_{Mod}(X) )$$

Taking $B$ a test commutative algebra,
$$B\textbf{End}_{QCoh(Spec({\mathbb{Z}_{(p)}}))}(X)(B) \simeq B\textbf{End}_{QCoh(B)}(Sym_{A_B}(M_B[1]))$$
which is the connected component of the space $QCoh(B)$ at $Sym_{A_B}(M_B[1])$.

This defines a morphism of stacks $B\textbf{End}_{QCoh(Spec({\mathbb{Z}_{(p)}}))}(X) \to QCoh(-)$ that is fully faithful, meaning it induces an equivalence on homotopy sheaves $\pi_k$, for $k>0$.

This induces a commutative diagram
$$\begin{tikzcd}
Map_{St^{gr,*}}( S^2_f(1),B\textbf{End}_{Mod}(X) ) \arrow[d,hook] \arrow[r] & Map_{St^{gr,*}}( S^3_f(2),B\textbf{End}_{Mod}(X) ) \arrow[d,hook]  \\
Map_{St^{gr,*}}( S^2_f(1),QCoh(-) ) \arrow[r] & Map_{St^{gr,*}}( S^3_f(2),QCoh(-) )
\end{tikzcd}$$

Therefore we identify the fiber of the top arrow with the fiber of the bottom arrow.

We now compute the fiber of
$$QCoh(S^2_f(1)) \simeq {\mathbb{Z}_{(p)}}[u]/u^2-Mod \to {\mathbb{Z}_{(p)}}[v]-Mod \simeq QCoh(S^3_f(2))$$
using a Koszul duality argument.

Using Remark \ref{tstrQCoh}, we can simply compute the fiber of the morphism
$${\mathbb{Z}_{(p)}}[u]/u^2-Mod \to {\mathbb{Z}_{(p)}}[v]-Mod$$
where $u$ is in degree $2$ and $v$ is in degree $3$.

\begin{Le} [Koszul duality]
We have equivalences
$${\mathbb{Z}_{(p)}}\{\alpha\}-Mod \simeq IndCoh({\mathbb{Z}_{(p)}}[u]/u^2-Mod)$$
and
$${\mathbb{Z}_{(p)}}\{\beta\}-Mod \simeq IndCoh({\mathbb{Z}_{(p)}}[v]-Mod)$$
where ${\mathbb{Z}_{(p)}}[\alpha]$ is the free differential graded algebra on one generator in degree $1$, ${\mathbb{Z}_{(p)}}[\beta]$ is the free differential graded algebra on one generator in degree $2$ and $IndCoh(-)$ denotes the completion by filtered colimits applied to the quasi-coherent complex construction $QCoh(-)$ .

Moreover, the induced morphism
$${\mathbb{Z}_{(p)}}\{\alpha\}-Mod \to {\mathbb{Z}_{(p)}}\{\beta\}-Mod$$
sends $(M,\alpha : M \to M[-1])$ to $(M,\alpha \circ \alpha : M \to M[-2])$, that is the forgetful functor induced by
$$\alpha \in {\mathbb{Z}_{(p)}}\{\alpha\} \mapsto \alpha \circ \alpha \in {\mathbb{Z}_{(p)}}\{\beta\}$$
\end{Le}

\begin{proof}
The functor $Hom_{{\mathbb{Z}_{(p)}}[u]/u^2}({\mathbb{Z}_{(p)}},-)$ gives an adjunction
$$End_{{\mathbb{Z}_{(p)}}[u]/u^2}({\mathbb{Z}_{(p)}})-Mod \leftrightarrows IndCoh({\mathbb{Z}_{(p)}}[u]/u^2-Mod)$$
which is an equivalence of category using Schwede-Shiple-Lurie's theorem, see \cite[Theorem 7.1.2.1]{Lur17}. We can take ${\mathbb{Z}_{(p)}}$ as a compact generator of $IndCoh({\mathbb{Z}_{(p)}}[u]/u^2-Mod)$.

Similarly, $Hom_{{\mathbb{Z}_{(p)}}[v]}({\mathbb{Z}_{(p)}},-)$ induces the equivalence
$$End_{{\mathbb{Z}_{(p)}}[v]}({\mathbb{Z}_{(p)}})-Mod \leftrightarrows IndCoh({\mathbb{Z}_{(p)}}[v]-Mod)$$

Now to conclude the proof of the lemma, we only need to show
$$Ext^2_{{\mathbb{Z}_{(p)}}[v]}({\mathbb{Z}_{(p)}},{\mathbb{Z}_{(p)}}) \simeq {\mathbb{Z}_{(p)}} \to {\mathbb{Z}_{(p)}} \simeq Ext^2_{{\mathbb{Z}_{(p)}}[u]/u^2}({\mathbb{Z}_{(p)}},{\mathbb{Z}_{(p)}})$$
is an isomorphism. The standard resolutions of ${\mathbb{Z}_{(p)}}$ yields the required result.
\end{proof}

The fiber of ${\mathbb{Z}_{(p)}}\{\alpha\}-Mod \to {\mathbb{Z}_{(p)}}\{\beta\}-Mod$ is given by $(M,\alpha : M \to M[-1])$ where $M$ is a complex and $\alpha \circ \alpha \simeq 0$ : which is the relation we wanted.

\end{proof}

\subsubsection{Mixed structures classification : the Dieudonné case}

\begin{Thm} \label{classiDieud}

Let $A$ be smooth commutative ${\mathbb{Z}_{(p)}}$-algebra, $M$ a projective $A$-module of finite type. We fix a derived Frobenius lift structure on the graded  simplicial algebra $Sym_A(M[1])$, with $M$ in weight $1$. From Proposition \ref{FLSymStr}, it is equivalent to a classical Frobenius lift $F$ on $A$ and a linear map of $A$-modules $\phi : M \to M$. We define $X$ as the derived affine scheme $Spec(Sym_A(M[1]))=\mathbb{V}(M[1])$ endowed with its natural grading, we regard it as an element of $dSt^{gr,Frob}$. The classifying space of Dieudonné mixed graded structures on $X$ compatible with its grading and Frobenius structure is discrete and in bijection with the set of Dieudonné algebra structures on the graded commutative ${\mathbb{Z}_{(p)}}$-algebra $\bigoplus_i \wedge^i_A M [-i]$ endowed with its natural canonical Frobenius lift structure.

\end{Thm}

\begin{proof}
The classifying space of mixed structures is given by the fiber of 
the forgetful functor
$$S^1_{gr}-dSt^{gr,Frob} \to dSt^{gr,Frob}$$
over $X$.

The classifying space is given by the mapping space $Map_{Mon(dSt^{gr,Fr})}(S^1_{gr},\underline{\textbf{End}}_{gr,Fr}(X))$, here $\underline{\textbf{End}}_{gr,Fr}(X)$ is a monoid in graded derived stack endowed with a Frobenius lift. By connexity of $S^1_{gr}$, the classifying space is equivalent to
$$Map_{Mon(dSt^{gr,Fr})}(S^1_{gr},\underline{\textbf{End}}^0_{gr,Fr}(X))$$
which is equivalent to the mapping space of pointed stacks $$Map_{dSt^{gr,Fr,*}}(BS^1_{gr},B\underline{\textbf{End}}^0_{gr,Fr}(X))$$

Here $\underline{\textbf{End}}^0_{gr,Fr}(X)$ is the full subcategory on $\underline{\textbf{End}}_{gr,Fr}(X)$ on endomorphisms that are homotopy equivalent to identity. Since $BS^1_{gr}$ is a stack, we may consider $B\underline{\textbf{End}}^0_{gr,Fr}(X)$ as a stack in $St^{gr,Fr} \subset dSt^{gr,Fr}$.

We first reduce the classification to the computation of mapping spaces compatible with endomorphisms instead of Frobenius structures. The mapping space $Map_{Gp(dSt^{gr})^{Fr}}(S^1_{gr},\underline{\textbf{End}}_{gr,Fr}^0(X))$ is computed by the fiber product
$$Map_{Gp(dSt^{gr})^{endo}}(S^1_{gr},Z) \times_{Map_{Gp(dSt_p^{gr})^{endo}}(S^1_{gr,p}(0),Z_p (0))}
Map_{Gp(dSt_p^{gr})}(S^1_{gr,p}(0),Z_p (0))$$
where $Z \coloneqq \underline{\textbf{End}}_{gr,Fr}^0(X)$ and the $p$ index denotes base changing to $\mathbb{F}_p$.

We start with a lemma to compute $Z_p(0)$

\begin{Le} \label{End0}
We have
$$\underline{\textbf{End}}^0_{gr,endo}(X)(0) \simeq *$$
\end{Le}

\begin{proof}
By definition of $\textbf{End}^0$, we have a triangle of derived stacks
$$\underline{\textbf{End}}^0_{gr,endo}(X) \to \underline{\textbf{End}}_{gr,endo}(X) \to \pi_0(\underline{\textbf{End}}_{gr,endo}(X))$$
and since taking $0$-weighted invariants of a graded stack preserve limits (see Proposition \ref{Weight0}), we have an induced triangle
$$\textbf{End}^0_{gr,endo}(X)(0) \to \textbf{End}_{gr,endo}(X)(0) \to \pi_0(\textbf{End}_{gr,endo}(X))(0) \simeq \pi_0(\textbf{End}_{gr,endo}(X)(0))$$

We only need to see that $\textbf{End}_{gr,endo}(X)(0)$ is a discrete stack. Now, by Proposition \ref{Endendo}, $\textbf{End}_{gr,endo}(X)$ is computed as an equalizer
$$\textbf{End}_{gr,endo}(X) \simeq eq(\textbf{End}_{gr}(X)^\mathbb{N} \rightrightarrows \textbf{End}_{gr}(X)(0)^\mathbb{N})$$
which gives on $0$ weighted invariants :
$$\textbf{End}_{gr,endo}(X)(0) \simeq eq(\textbf{End}_{gr}(X)(0)^\mathbb{N} \rightrightarrows \textbf{End}_{gr}(X)(0)^\mathbb{N})$$

From the proof of Theorem \ref{classifClassique}, $\textbf{End}_{gr}(X)(0)$ is discrete, which concludes the proof.
\end{proof}

\begin{Le} \label{EndFrendo}
If we denote $\phi$ the endomorphism of $\textbf{End}^0_{gr,Fr}(X)$ and $h'$ the homotopy between $\phi_{\mkern 1mu \vrule height 2ex\mkern2mu p}$ and $Fr_p$ and $\psi$ the endomorphism of $\textbf{End}^0_{gr,endo}(X)$. We have an equivalence between the underlying graded derived stacks endowed with endomorphisms
$$(\underline{\textbf{End}}^0_{gr,Fr}(X),\phi) \simeq (\underline{\textbf{End}}^0_{gr,endo}(X),\psi)$$
meaning that we forget the "being homotopic to the canonical Frobenius modulo $p$" part on the left hand side.
\end{Le}

\begin{proof}
We will promote $(\underline{\textbf{End}}^0_{gr,endo}(X),\psi)$ from an object with an endomorphism to an object with a Frobenius lift in an essentially unique way, then we will show that this object has the universal property that holds for $(\underline{\textbf{End}}^0_{gr,Fr}(X),\phi)$.

Promoting $(\textbf{End}^0_{gr,endo}(X),\psi)$ to a graded derived stack with a Frobenius has
$$Map_{\underline{\textbf{End}}^0_{gr,endo}(X)_p(0)}(\psi_p,Fr_p)$$ as a space of choices. We know, by Lemma \ref{End0}, that $$Map_{\underline{\textbf{End}}^0_{gr,endo}(X)_p(0)}(\psi_p,Fr_p)$$
is discrete and equivalent to a point, therefore $Map_{\underline{\textbf{End}}^0_{gr,endo}(X)_p(0)}(\psi_p,Fr_p)$ is contractible and $(\textbf{End}^0_{gr,endo}(X),\psi)$ can be promoted to an object with Frobenius in an essentially unique way : $(\textbf{End}^0_{gr,endo}(X),\psi,h)$.

Let $(T, \lambda,h_T)$ be an object of $dSt^{Fr}$ and we write $\mu$ the endomorphism of $X$ and $h_X$ the homotopy between $\mu_p$ and $Fr_p$.

The data of a morphism
$$(T,\lambda) \times (X,\mu) \to (X,\mu)$$
is equivalent to the data of a morphism

$$(T,\lambda) \to (\underline{\textbf{End}}^0_{gr,endo}(X),\psi)$$
by the universal property of $(\underline{\textbf{End}}^0_{gr,endo}(X),\psi)$, which is equivalent to the data of a morphism of objects with Frobenius lifts
$$(T,\lambda,h_T) \to (\underline{\textbf{End}}^0_{gr,endo}(X),\psi,h)$$

since the choice of compatibility of the morphism of endomorphism objects with the Frobenius lift structures is given by the data of a commutative cube as follows :

\begin{center}
\begin{tikzcd}[row sep=1.5em, column sep = 1.5em]
Y_p(0) \arrow[rr] \arrow[dr,"id"] \arrow[dd,"\lambda_p"] \arrow[rddd,phantom,"\sim (h_T)"]
&& Z_p(0) \arrow[dd,"\phi_p",near end] \arrow[dr,"id"] \arrow[rddd,phantom,"\sim (h)"] \\
& T_p(0) \arrow[rr] \arrow[dd,"Fr_p", near start]&&
Z_p(0) \arrow[dd,"Fr_p"] \\
T_p(0) \arrow[rr] \arrow[dr, "id"]
&& Z_p(0) \arrow[dr,"id"] \\
& Y_p(0) \arrow[rr]
&& Z_p(0)
\end{tikzcd}
\end{center}

where $Z \coloneqq \underline{\textbf{End}}^0_{gr,endo}(X)$. This data lives in a $2$-simplex in $Map_{gr}(Y_p(0),Z_p(0))$ and since $Z_p(0)$ is contractible, so is $Map_{gr}(Y_p(0),Z_p(0))$, therefore the morphism of endomorphism objects extends in an essentially unique way to a morphism of objects with Frobenius lifts.

Now a similar argument shows that any morphism
$$(T,\lambda) \times (X,\mu) \to (X,\mu)$$
extends in an essentially unique way to a morphism of objects with Frobenius lifts
$$(T,\lambda,h_T) \times (X,\mu,h_X) \to (X,\mu,h_X)$$

Therefore $(\underline{\textbf{End}}^0_{gr,endo}(X),\psi,h)$ satisfies the universal property of $(\underline{\textbf{End}}^0_{gr,Fr}(X),\phi,h')$ and they are equivalent.
\end{proof}

\begin{Le}
We have
$$\underline{\textbf{End}}_{gr,Fr}^0(X)_p(0) \simeq *$$
\end{Le}

\begin{proof}
It is enough to show that $\underline{\textbf{End}}_{gr,Fr}^0(X)(0)$ is contractible. We simply combine Lemma \ref{End0} and Lemma \ref{EndFrendo}.

\end{proof}

We deduce that
$$Map_{Gp(dSt^{gr,Fr})}(S^1_{gr},\underline{\textbf{End}}_{gr,Fr}^0(X)) \simeq Map_{Gp(dSt^{gr,endo})}(S^1_{gr},\underline{\textbf{End}}_{gr,endo}^0(X))$$

We have reduced the study of the Frobenius structure to a mere endomorphism structure. The previous mapping space is given by the equalizer
$$Map_{Gp(dSt^{gr})}(S^1_{gr},\underline{\textbf{End}}_{gr,endo}^0(X)) \doublerightarrow Map_{Gp(dSt^{gr})}(S^1_{gr},\underline{\textbf{End}}_{gr,endo}^0(X))$$
where the top map is given by precomposing by $[p]$ and the bottom map is postcomposing by $S$ the endomorphism of $\underline{\textbf{End}}_{gr,endo}^0(X)$.

Now, an element of $Map_{Gp(dSt^{gr})}(S^1_{gr},\underline{\textbf{End}}_{gr,endo}^0(X))$
can be described using Proposition \ref{Endendo} as the set of maps
$$f_k : S^1_{gr} \to \textbf{End}_{gr}^0(X)$$
with $prec_{\phi_X} \circ f_{k+1} = post_{\phi_X} \circ f_k$, where $prec$ and $post$ are the precomposing and postcomposing morphisms. The equalizer is then given by a family $(f_k)$ such that $f_{k+1} = f_k \circ [p]$. This amount to the data of
$$f : S^1_{gr} \to \textbf{End}_{gr}^0(X)$$
such that $prec_{\phi_X} \circ f \circ [p] = post_{\phi_X} \circ f_k$.

We now use the classification for classical mixed structures, Theorem \ref{classifClassique}, to identify the element
$$ f : S^1_{gr} \to \textbf{End}_{gr}^0(X)$$
with a couple $(\delta : A \to M, d : M \to \Lambda^2_A M)$ satisfying the usual equations. Unwinding the action of the endomorphisms on $f$. We see that precomposing with $\phi_X$ induces the postcomposition map
$$(\delta,d) \mapsto (\phi_M \circ \delta, (\phi_M \wedge \phi_M) \circ d)$$

Similarly, postcomposing with $\phi_X$ induces the precomposition map
$$(\delta,d) \mapsto (\delta \circ \phi_A, d \circ \phi_M)$$

And finally the action of $[p]$ is multiplication by $p$
$$(\delta,d) \mapsto (p \delta,pd)$$

The equations are therefore
$$ \phi_M \circ \delta = p \delta \circ \phi_A$$
and
$$ (\phi_M \wedge \phi_M) \circ d = p d \phi_M$$
which are the relations required for $(\delta,d)$ to define a Dieudonné algebra structure on $\bigoplus_i \wedge^i_A M [-i]$.
\end{proof}

\begin{Cor} \label{CorClassDieudStr}
With the notations of Theorem \ref{classiDieud}, the graded derived affine scheme $Spec(Sym_A(\Omega_A[1]))$ admits a unique Dieudonné structure induced by the standard Dieudonné structure on $\bigoplus_{i \ge 0} \Lambda_A^i \Omega_A$.
\end{Cor}

\subsubsection{De Rham-Witt complex}

\begin{Cons}
The functor of "simplicial functions"
$$\mathcal{O} : dSt^{op} \to SCR$$
induces a functor from $G$-equivariant derived stacks, where $G$ is a group in $dSt$ :
$$\mathcal{O} : G-dSt^{op} \to G-SCR$$
where $G-SCR$ is abusively defined as the category of relative affine derived schemes over $BG$. Therefore we can see this construction as taking the scheme-affinization of a derived stack, relatively to $BG$, using Proposition \ref{GActionStack}.

By taking $G = S^1_{gr} \rtimes (\mathbb{G}_m \times \mathbb{N})$, we can check that we obtain a functor
$$\mathcal{O} : \epsilon-D-dSt^{gr,op} \to \epsilon-D-SCR$$
\end{Cons}

\begin{Def}
Let $(A,F)$ be a simplicial algebra with a Frobenius lift, we define the Dieudonné de Rham complex of $(A,F)$ as
$$\mathcal{O}(\mathcal{L}^{gr,Fr}(Spec(A),Spec(F)))$$
where $\mathcal{O}$ is the functor of "simplicial functions" on derived Dieudonné stacks
$$\epsilon-D-dSt^{gr,op} \to \epsilon-D-SCR^{gr}$$
defined in the previous proposition. We denote the element we obtain $DDR(A,F)$.
\end{Def}

\begin{Thm} \label{ComparaisonDRW}
Let $A$ be a smooth discrete $p$-torsion-free algebra, the Dieudonné de Rham complex coincides with the de Rham complex endowed with its classical Dieudonné structure defined in \cite{BLM22}
\end{Thm}

\begin{proof}
It is a simple application of Corollary \ref{CorClassDieudStr}.
\end{proof}

\begin{proof}
It is an easy corollary of Theorem \ref{classiDieud}.
\end{proof}

\begin{Thm} \label{AdjointDeRhamW} The Dieudonné de Rham functor
$$SCR^{Fr} \to \epsilon-D-SCR^{gr}$$
is left adjoint to the forgetful functor
$$(0) : \epsilon-D-SCR^{gr} \to SCR^{Fr}$$
\end{Thm}

\begin{proof}
This is an application of the definition of the various categories and the Dieudonné de Rham functor.
\end{proof}

\begin{Thm}
For $A$ a commutative algebra, non-necessarily smooth, the t-truncation of $DDR(A)$, with respect to the Beilinson t-structure, is equivalent to the classical Dieudonné complex. That is
$$t_{\ge 0}(DDR(A)) \simeq i(\Omega^\bullet_A)$$
where $\Omega^\bullet_A$ is endowed with its canonical classical Dieudonné structure, see Proposition \ref{FormDieud} and $i$ is the functor
$$ i : \textbf{DA} \to \epsilon-D-SCR^{gr}$$
identifying the category of classical Dieudonné algebra with the heart of $\epsilon-D-SCR^{gr}$.
\end{Thm}

\begin{proof}
As the t-truncation is compatible with the forgetful map
$$\epsilon-D-SCR^{gr} \to \epsilon-Mod^{gr}$$
and the forgetful map is conservative, the natural morphism
$$Sym_A(\Omega_A^\bullet[1]) \to t_{\ge 0}(Sym_A(\mathbb{L}_A[1]))$$
is an equivalence on underlying graded mixed Dieudonné complexes, therefore it is an equivalence.
\end{proof}

\subsection{Saturation}
In this section we study in more details the properties of saturation and comparing it to the classical saturation for Dieudonné algebras.

\subsubsection{The décalage functor}

\begin{Def}
We define an endofunctor of $\epsilon-Mod^{gr}_{\mathbb{Z}_{(p)}}$ denoted $[p]^*$ by forgetting along
$$\epsilon \in {\mathbb{Z}_{(p)}}[\epsilon] \mapsto p\epsilon \in {\mathbb{Z}_{(p)}}[\epsilon]$$
\end{Def}

\begin{Def}
We define the décalage $\eta_p$ functor as an endofunctor of $1$-categories $\bf{\epsilon-Mod^{gr}_{\mathbb{Z}_{(p)}}}$ sending $M$ to the graded mixed subcomplex of $M \times M$ on elements $(x,y)$ such that $\epsilon x = py$ and $\epsilon y=0$.
\end{Def}

\begin{Prop} \label{adjDec}
We have an adjunction of $1$-categories
$$ [p]^* : \bf{\epsilon-Mod^{gr}_{\mathbb{Z}_{(p)}}} \rightleftarrows \bf{\epsilon-Mod^{gr}_{\mathbb{Z}_{(p)}}} : \eta_p$$
\end{Prop}

\begin{Prop}
The adjunction is a Quillen adjunction
$$ [p]^* : \bf{(\epsilon-dg-mod^{gr}_{\mathbb{Z}_{(p)}})_{inj}} \rightleftarrows \bf{(\epsilon-dg-mod^{gr}_{\mathbb{Z}_{(p)}})_{inj}} : \eta_p$$
Therefore this adjunctions induces an adjunction between the corresponding $\infty$-categories.
\end{Prop}

\begin{proof}
The injective model structure on $(\epsilon-dg-mod^{gr}_{\mathbb{Z}_{(p)}})_{inj}$ is defined by transfer along the forgetful functor $U : \bf{\epsilon-dg-mod^{gr}_{\mathbb{Z}_{(p)}}} \to \bf{dg-mod_{\mathbb{Z}_{(p)}}}$. We need to verify that $F^*_p$ sends cofibrations and trivial cofibrations to cofibrations and trivial cofibrations. Since for any graded mixed complex $M$, the underlying complex of $[p]^*M$ and $M$ are the same, the result is obvious.
\end{proof}

\begin{Prop}
Under the equivalence $(\epsilon-Mod^{gr}_{\mathbb{Z}_{(p)}})^{\heartsuit} \simeq Mod_{\mathbb{Z}_{(p)}}$ of \ref{HeartBeilClass}, the endofunctor $\eta_p$ induces an endofunctor on $Mod_{\mathbb{Z}_{(p)}}$, which, when restricted to $p$-torsion free complex, identifies with the décalage functor $\eta_p$ from \cite{BLM22}.
\end{Prop}

\begin{proof}
When $M$ is a $p$-torsion free graded mixed complex, an element $x \in M_n$ such that $\epsilon x$ is $p$-divisible defines a unique $y \in M_{n+1}$ such that $\epsilon x=py$, furthermore $\epsilon y=0$ is automatic since $p \epsilon y = \epsilon^2 x = 0$.
\end{proof}

\begin{Rq} \label{RqLFP}
From the definition of Dieudonné structures on complexes and the adjunction of Proposition \ref{adjDec}, the data of a Dieudonné structure on a graded mixed complex is equivalent to the data of a lax fixed structure with respect to $\eta_p$.

\end{Rq}

\begin{Prop} \label{decFilt}
The functor $\eta_p : \epsilon-Mod^{gr}_{\mathbb{Z}_{(p)}} \to \epsilon-Mod^{gr}_{\mathbb{Z}_{(p)}}$ commutes with filtered colimits.
\end{Prop}

\begin{proof}
The forgetful functor $\epsilon-Mod^{gr}_{\mathbb{Z}_{(p)}} \to Mod_{\mathbb{Z}_{(p)}}$ is conservative. Furthermore the composition $U \circ \eta_p$ is given by forming a finite limit
$$M \mapsto M \times_{M[1]} M[1] \times_{M[2]} 0$$ where the arrow are $\epsilon : M \to M[1]$, $- \times p : M[1] \to M[1]$, $\epsilon : M[1] \to M[2]$ and $0 : 0 \to M[2]$. Then $U \circ \eta_p$ commutes with filtered colimits, hence so does $U$.
\end{proof}

\begin{Prop}
The fully faithful inclusion of fixed points
$$i_{FP} : FP_{\eta_p}(\epsilon-Mod^{gr}_{\mathbb{Z}_{(p)}}) \subset LFP_{\eta_p}(\epsilon-Mod^{gr}_{\mathbb{Z}_{(p)}})$$
admits a left adjoint.
\end{Prop}

\begin{proof}
The category $\epsilon-Mod^{gr}_{\mathbb{Z}_{(p)}}$ is presentable as a module category. The functor $\eta_p$ commutes with filtered colimits (Proposition \ref{decFilt}) and with small limits (Proposition \ref{adjDec}). We conclude using Proposition \ref{LFAdj}.
\end{proof}

\begin{Def}
Identifying the category of mixed Dieudonné complexes with the category $LFP_{\eta_p}(\epsilon-Mod^{gr}_{\mathbb{Z}_{(p)}})$, the subcategory of saturated mixed Dieudonné complexes is defined as $FP_{\eta_p}(\epsilon-Mod^{gr}_{\mathbb{Z}_{(p)}})$, we denote it $\epsilon-D-Mod^{gr,sat}$.
\end{Def}

\subsubsection{Saturated Dieudonné algebra}

\begin{Def}
We define the décalage functor for Dieudonné stacks. Let $ \pi : X \to BS^1_{gr}$ be a Dieudonné stack, the \textit{décalé} Dieudonné stack is simply $X$ endowed with the composed structure morphism : $X \xrightarrow{\pi} BS^1_{gr} \xrightarrow{[p]} BS^1_{gr}$. This \textit{décalage} construction defines an endofunctor of $D-dSt$.

\end{Def}

\begin{Rq}
A Dieudonné simplicial algebra, that is a Dieudonné derived stack $X \to BS^1_{gr}$ which is relatively affine, has an associated \textit{décalé} Dieudonné derived stack, which is not necessarily a Dieudonné simplicial algebra.
\end{Rq}

\begin{Prop}
The \textit{décalage} construction has a right adjoint given by taking the fiber product of the structure map along $[p] : BS^1_{gr} \to BS^1_{gr}$.
\end{Prop}

\begin{Def}
We define $\epsilon-D-SCR^{gr,sat}$, the category of saturated mixed graded Dieudonné simplicial algebras, as the subcategory of mixed graded Dieudonné algebras having an underlying mixed graded Dieudonné complex which is saturated. Meaning :
$$\epsilon-D-SCR^{gr,sat} = \epsilon-D-SCR^{gr} \times_{\epsilon-D-Mod^{gr}} \epsilon-D-Mod^{gr,sat}$$
\end{Def}

\begin{Prop}
The inclusions
$$i : \epsilon-D-Mod^{gr,sat} \subset\epsilon-D-Mod^{gr}$$
and
$$j : \epsilon-D-SCR^{gr,sat} \subset \epsilon-D-SCR^{gr}$$
both admit a left adjoint.
\end{Prop}

\begin{proof}
For the first inclusion, it follows from Remark \ref{RqLFP} and Proposition \ref{LFAdj}.

By definition of saturated Dieudonné algebras, we have a pullback diagram :

\begin{tikzcd}
\epsilon-D-SCR^{gr,sat} \arrow[d,"U_0"] \arrow[r,"j"] \arrow[dr, phantom, "\lrcorner", very near start] & \epsilon-D-SCR^{gr} \arrow[d,"U"] \\
\epsilon-D-Mod^{gr,sat} \arrow[r,"i"]& \epsilon-D-Mod^{gr}
\end{tikzcd}

We show $j$ commutes with filtered colimits and small limits. Since $U$ commutes with these colimits and limits and it is reflective, we show $U \circ j \simeq i \circ U_0$ preserves the required colimits and limits. The functor $j$ preserves these colimits and limits from the proof of Proposition \ref{LFAdj} and so does $U_0$ since it is the projection of a fiber product and fibered product commutes with filtered colimits and small limits. We conclude with the adjoint functor Theorem \ref{adjointFunct}.

\end{proof}

\begin{Rq}
We can notice the décalé Dieudonné derived stack of a Dieudonné stack has functions given by the saturation of functions of the Dieudonné stack, therefore this construction gives a geometrical interpretation of saturation.
\end{Rq}

\section{Future perspectives} \label{futuPers}

\subsection{Improvements}

We will expose many possible directions for improvements and generalizations of the results. In this thesis, we have defined the de Rham-Witt complex of a simplicial $\mathbb{Z}_{(p)}$-algebra with a Frobenius lift. In order to define de Rham-Witt of a simplicial $\mathbb{F}_p$-algebra R, we will need to construct the functor Verschiebung $V$. We will consider a completed version with respect to a $V$-filtration of the Dieudonné de Rham algebra of $(W(R),F)$.

\paragraph*{The $V$ functor and the $V$ filtration}
Following the construction of the de Rham-Witt complex in \cite{BLM22}, we can construct a Verschiebung map $V$. The fundamental compatibility
$$FV = p$$
is expected to hold. From which we deduce, assuming we work on $p$-torsion-free saturated Dieudonné complexes
$$p \epsilon V = V\epsilon$$

That is, we want to construct a decomposition of multiplication by $p$ :
$$M \xrightarrow{V} [p]^*M \xrightarrow{F} M$$

The construction of such a map is obvious for $M$ a $p$-torsion-free saturated derived Dieudonné complex.  We assume to have constructed such a map $V$.

\begin{Def}
For $M$ a saturated derived Dieudonné complex, we denote $\mathcal{W}_r(M)$ the graded mixed complexes representing the functor
$$N \in \epsilon-Mod^{gr} \mapsto Map^{V^r}(M,N)$$
where $Map^{V^r}(M,N)$ is the mapping space of graded mixed morphisms $M \to N$ such that the composition $M \xrightarrow{V^r} M \to N$ is homotopic to zero. Meaning $Map^{V^r}(M,N)$ is the fiber of
$$Map_{\epsilon-Mod^{gr}}(M,N) \xrightarrow{- \circ V^r} Map_{Mod}(M,N)$$

We have a canonical restriction map
$$res : M \to lim_r \mathcal{W}_r(M)$$

We denote $\mathcal{W}(M) \coloneqq lim_r \mathcal{W}_r(M)$.

The Dieudonné complex $M$ is said to be strict when $res$ is an equivalence.
\end{Def}

\begin{Rq}
The usual definition of $\mathcal{W}_r(M)$ is given by
$$\mathcal{W}_r(M) \coloneqq M/(Im(V^r) + Im(dV^r))$$
they coincide as taking the cofiber of $M \to M$ in the category of graded mixed complexes kills $Im(dV^r)$ automatically.
\end{Rq}

We recall an alternative description of the process of strictification.

\begin{Cor} (\cite[Corollary 2.8.2]{BLM22}) Let $M$ be a saturated Dieudonné complex. The restriction map
$$M \to \mathcal{W}(M)$$
exhibits $\mathcal{W}(M)$ as a $p$-completion of $M$ in the derived category of abelian groups.
\end{Cor}

This corollary motivates defining a derived Dieudonné complex to be strict when is is derived $p$-complete.

We can then extend these definitions to derived Dieudonné algebras : a derived Dieudonné algebra is said to be strict when it is as a derived Dieudonné complex. However the description of the strictification process seems more involved, in particular, defining the $A/VA$ as a simplicial algebra seems to be a difficult task in this homotopy context.

We define the category of strict derived Dieudonné complexes and algebras, respectively $\textbf{DC}_{str}$ and $\textbf{DA}_{str}$. We expect the inclusion
$$\textbf{DA}_{str} \subset \textbf{DA}$$
to admit a left adjoint denoted $\mathcal{W}Sat(-)$.

Defining $f : B \in \textbf{DA}_{str} \mapsto B^0/VB^0 \in SCR_{\mathbb{F}_p}$, we can ask ourselves

\begin{Ques}
Does the functor $$f : \textbf{DA}_{str} \to SCR_{\mathbb{F}_p}$$
admit a left adjoint given by
$$R \mapsto \mathcal{W}Sat(DDR(R^{red}))$$
where $R^{red}$ is the reduction of $R$ ?
\end{Ques}

\begin{Rq}
The construction of the left adjoint might not need to consider the reduction since our main results hold for simplicial algebras which are not $p$-torsion free.
\end{Rq}

\paragraph*{De Rham-Witt complex of a $\mathbb{F}_p$-algebra}

To complete the construction of the de Rham-Witt complex, we need a universal property of the Witt vectors. In \cite[Proposition 3.6.3]{BLM22}, the Witt vectors give the following identification
$$Hom_F(W(R),B) \cong Hom_{\textbf{DA}}(R,B^0/VB^0)$$
where $B$ is a strict Dieudonné algebra and $R$ is a $\mathbb{F}_p$-algebra. The set $Hom_F(W(R),B)$ consists of ring morphisms commuting with the Frobenius morphisms.

We hope to have a similar result in the context of simplicial algebras with Frobenius lifts.

\begin{Def}
We can define the Dieudonné algebra of a simplicial $\mathbb{F}_p$-algebra $R$ as $\mathcal{W}Sat(\mathbb{L}_{W(R)})$.
\end{Def}

\paragraph*{A more general base}
Our main results are given for schemes on $\mathbb{Z}_{(p)}$, we could develop a theory of graded loopspaces on a derived scheme endowed with a Frobenius over a general commutative ring $k$. However some precautions have to be taken, for example, for the crystalline circle
$$S^1_{gr} = Spec^\Delta(k \oplus k[1])$$
to have a derived Frobenius lift, we need $k$ to admit a derived Frobenius lift. When $k= \mathbb{Z}_{(p)}$, such a lift is essentially unique, for $k=\mathbb{F}_p$, there is not such lift since a commutative ring of finite $p$-torsion does not admit a $\delta$-ring structure. Therefore the base commutative ring $k$ must be endowed with a Frobenius lift structure.

Another generalization would be to replace the base $\mathbb{Z}_{(p)}$ by $\mathbb{Z}$. In this context, all prime numbers have to be considered for the Frobenius lifts, therefore, we have to use the more general notion of commutating Frobenius lifts. Furthermore big Witt vectors need to be used instead of $p$-typical Witt vectors.

\begin{Cons}
Let $\mathcal{C}$ be an $\infty$-category. We define the category of commutating endomorphisms on $\mathcal{C}$
$$\mathcal{C}^{\mathbb{N}^*endo} \coloneqq Fun(B\mathbb{N}^*,\mathcal{C})$$
where $(\mathbb{N}^*,\times) \simeq \mathbb{N}^{(\mathbb{N})}$ is the free abelian monoid on the set $\mathbb{N}$. Sending $1$ to the $i$-th prime number $p_i$ induces a morphism of monoids
$$\mathbb{N} \to \mathbb{N}^*$$
which gives by restriction
$$p_i : \mathcal{C}^{\mathbb{N}^*endo} \to  \mathcal{C}^{endo}$$

Taking $\mathcal{C}=SCR_{\mathbb{Z}_{(p)}}$, we define the category of simplicial algebras with commutating Frobenius lifts as
$$SCR^{\mathbb{N}^*Fr} \coloneqq \mathcal{C}^{\mathbb{N}^*endo} \times_{(SCR_p^{endo})^\mathbb{N}} (SCR_p)^{\mathbb{N}}$$

An element in this category is a simplicial algebra $A$ endowed with morphisms
$$\phi_p : A \to A$$
commutating up to coherent homotopy and the data of a homotopy between $\phi_p\vert_{\mathbb{F}_p}$ and the canonical Frobenius $Fr_p$.
\end{Cons}

This definition is close to the notion of cyclic spectra developed in \cite{NS18} and some connections with topological cyclic homology might be found.

Combining these notions could allow us to work over any general discrete commutative ring endowed with commutating Frobenius lifts.

\paragraph*{Study of derived stacks with Frobenius lift}

The theory of derived stacks with Frobenius lift has been quickly description in this thesis but a deeper examination could be fruitful. In particular, a theory of Postnikov towers could shorten many proofs.

We have proven that $dSt^{Fr}$ is an $\infty$-topos. Extending Proposition\ref{pointsEndodSt}, we can ask ourselves the following question.

\begin{Ques}
Let $\mathcal{C}_0$ be the full subcategory of $dSt^{Fr}$ on objets which are given by freely adjoining a Frobenius to a derived affine scheme, these objects are of the form $L(Spec(C))$, with $C$ a simplicial algebra. Does the category $\mathcal{C}_0$ admit a site structure such that the topos of sheaves on $\mathcal{C}_0$ identifies with $dSt^{Fr}$ ? In this case, we would have the identification :
$$Sh(\mathcal{C}_0) \simeq dSt^{Fr}$$
\end{Ques}
 
A precise description of the contructions and results of this thesis on derived stacks with endomorphisms instead of Frobenius lifts could be illuminating. We expect to be able to recover the construction of the de Rham-Witt complex when considering
$$\textbf{Map}_{dSt^{endo}}(S^1_{gr},X)$$
when $X$ admits a Frobenius lift, instead of $$\textbf{Map}_{dSt^{Fr}}(S^1_{gr},X)$$

\paragraph*{HKR-type filtration}
Inspired by \cite{MRT20}, we could construct a filtration analoguous to the one on hochschild cohomology. In the classical case, the affinization of the circle
$$Aff(S^1) \simeq BFix$$
admits a filtration which has $BKer$ as its associated graded stack.

The stack $BFix$ admits an endomorphism structure induced by the morphism
$$\mathbb{Z}\left[\binom{X}{n}\right] = \{ Q \in \mathbb{Q}[X] : Q(\mathbb{Z}) \subset \mathbb{Z} \} \to \{ Q \in \mathbb{Q}[X] : Q(\mathbb{Z}) \subset \mathbb{Z} \} = \mathbb{Z}\left[\binom{X}{n}\right]$$
sending $Q$ to $Q(pX)$.

This endomorphism is not a Frobenius lift. We denote it $[p]$.

\begin{Def}
Let $(X,F)$ be a derived scheme endowed with an endomorphism, we define the loopspace of $X$ as
$$\mathcal{L}^{endo}((X,F)) \coloneqq \textbf{Map}_{dSt^{endo}}((BFix,[p]),(X,F))$$

similarly, the graded loopspace is defined as
$$ \mathcal{L}^{endo}_{gr}((X,F)) \coloneqq \textbf{Map}_{dSt^{endo}}((BKer,[p]),(X,F))$$
\end{Def}

\begin{Rq}
The "multiplication by $p$" map of topological spaces
$$S^1 \to S^1$$
defines an endomorphism of the stack $S^1$. And we can expect an identification of graded affine stacks with endomorphisms
$$Aff((S^1,\times p)) \simeq (BFix,[p])$$
and also an equivalence
$$\textbf{Map}_{dSt^{endo}}((S^1,\times p),(X,F)) \simeq \textbf{Map}_{dSt^{endo}}((BFix,[p]),(X,F))$$
\end{Rq}

\begin{Def}
When $A$ is a simplicial algebra over $\mathbb{Z}_{(p)}$ and $F$ is en endomorphism of $A$, we define the Hochschild cohomology of $(A,F)$ as
$$HH((A,F)) \coloneqq \mathcal{O}_{SCR}(\mathcal{L}^{endo}((Spec(A),Spec(F))))$$
and the de Rham algebra of $(A,F)$ as 
$$DR((A,F)) \coloneqq \mathcal{O}_{SCR}(\mathcal{L}^{endo}_{gr}((Spec(A),Spec(F))))$$
\end{Def}

\begin{Rq}
We expect $DR((A,F))$ to be very close to
$$Sym_A((\mathbb{L}_A[1],\frac{dF}{p}))$$
at least when $A$ is $p$-torsion free and the endomorphism can be promoted to a Frobenius lift.
\end{Rq}

We can ask the following

\begin{Ques}
Does the Hochschild cohomology of a simplicial algebra with endomorphism $(A,F)$ admit a natural filtration which has $DR(A,F)$ as its associated graded simplicial algebra ?
\end{Ques}

Defining such a filtration can be achieved by constructing a filtered analogue of our graded circle. However some precautions have to be taken since the natural endomorphism on the affinization of the circle $Aff(S^1)$ is not a Frobenius lift. Therefore a careful comparision between Frobenius graded loopspace and endomorphism graded loopspace will probably be necessary.

\paragraph*{Prismatic circle} We expect our construction of graded circle to have an analogue over the prismatic site. We would call this object the prismatic circle. Precisely, for $(A,I)$ a prism, $A$ is a $\delta$-ring, therefore it has a canonical Frobenius lift. We can then consider
$$\textbf{Map}_{dSt^{Fr}}(S^1_{gr},Spec(A))$$

We expect taking levelwise mapping space with the crystalline circle to define a sheaf on the prismatic site. This sheaf could then recover prismatic cohomology from a mixed graded complex.

\begin{Rq}
The above definition is incomplete and needs to be modified to take into account the ideal $I$ when the prism is not a crystalline prism.
\end{Rq}

\subsection{Symplectic forms}
One of the main motivations behind our work on the graded Dieudonné loopspaces was to define a theory of shifted symplectic forms in mixed characteristic. The theory of shifted symplectic structures and shifted Poisson structures is carried out in \cite{PTVV} and \cite{CPTVV}. Furthermore they are helpful for the study of quantization deformations.

We recall some definitions from \cite{PTVV}.

\begin{Def}
Let $F$ be a derived Artin stack over a commutative ring of characteristic zero $k$. The space of $n$-shifted $p$-forms on $F$ is
$$\mathcal{A}^p(F,n) \coloneqq Map_{Mod}(k(p)[-p-n],DR(X))$$
and the space of closed $n$-shifted $p$-forms on $F$ is
$$\mathcal{A}^{p,cl}(F,n) \coloneqq Map_{\epsilon-Mod^{gr}}(k(p)[-p-n],DR(X))$$

We have a natural forgetful functor
$$\mathcal{A}^{p,cl}(F,n) \to \mathcal{A}^p(F,n)$$
and the differential induces
$$d_{DR} : \mathcal{A}^p(F,n) \to \mathcal{A}^{p+1,cl}(F,n)$$
\end{Def}

\begin{Def}
Let $\omega$ be a closed $2$-form of degree $n$ on $F$, by adjunction it induces a morphism of quasi-coherent sheaves on
$$\Theta_w : \mathbb{T}_F \to \mathbb{L}_F[n]$$
since $F$ is a derived Artin stack, therefore $\mathbb{L}_F$ is dualizable. The form $w$ is said to be an $n$-shifted symplectic form when $\Theta_w$ is an equivalence.
\end{Def}

We generalize these notions to the Dieudonné de Rham complex.

\begin{Def}
Let $X$ be a derived scheme over $\mathbb{Z}_{(p)}$, endowed with a Frobenius lift $F$. The space of $n$-shifted Dieudonné $p$-forms on $X$ as
$$\mathcal{A}^p(X,F,n) \coloneqq Map_{Mod}(k(p)[-p-n],DDR(X,F))$$
and the space of closed $n$-shifted Dieudonné $p$-forms on $X$ as
$$\mathcal{A}^{p,cl}(X,F,n) \coloneqq Map_{\epsilon-Mod^{gr}}(k(p)[-p-n],DDR(X,F))$$

We have a natural forgetful functor
$$\mathcal{A}^{p,cl}(X,F,n) \to \mathcal{A}^p(X,F,n)$$
and the differential induces
$$d_{DR} : \mathcal{A}^p(X,F,n) \to \mathcal{A}^{p+1,cl}(X,F,n)$$

Furthermore, the endomorphism on $DDR(X,F)$, which we can see as $\frac{dF}{p}$ on $\mathbb{L}_X$ induces natural transformations :
$$F_{form} : \mathcal{A}^p(X,F,n) \to \mathcal{A}^p(X,F,n)$$
and
$$F_{cl} : \mathcal{A}^{p,cl}(X,F,n) \to \mathcal{A}^{p,cl}(X,F,n)$$
\end{Def}

We define a variation on the definition which takes the Frobenius structure into account.

\begin{Def}
The space of $n$-shifted Dieudonné $p$-forms on $X$ as
$$\mathcal{A}^p_{Fr}(X,F,n) \coloneqq Map_{Mod^{endo}}(k(p)[-p-n],DDR(X,F))$$
where we endow $k(p)[-p-n]$ with the "multiplication by $p$" endomorphism and the space of closed $n$-shifted Dieudonné $p$-forms on $X$ as
$$\mathcal{A}^{p,cl}_{Fr}(X,F,n) \coloneqq Map_{\epsilon-Mod^{gr,endo}}(k(p)[-p-n],DDR(X,F))$$

We have a natural forgetful functor
$$\mathcal{A}^{p,cl}_{Fr}(X,F,n) \to \mathcal{A}^p_{Fr}(X,F,n)$$
and the differential induces
$$d_{DR} : \mathcal{A}^p_{Fr}(X,F,n) \to \mathcal{A}^{p+1,cl}_{Fr}(X,F,n)$$
\end{Def}

\begin{Rq}
We recall the definition of Frobenius-derivations used for Fedosov quantization in \cite{BK07}. Let $A$ be commutative ring and $M$ an $A$-module, a Frobenius-derivation $D : A \to M$ is a morphism of abelian groups such that $D(1)=0$ and
$$D(ab) = a^p D(b) + b^p D(a)$$

This definition suggest yet another notion of $n$-shifted $p$-forms, on a simplicial algebra with Frobenius lift $(A,F)$, based on a notion of Frobenius-twisted cotangent complex
$$\mathbb{L}^{tw}_{(A,F)} \coloneqq \mathbb{L}_{A} \otimes_A A$$
where the morphism $A \to A$ is $F$.
\end{Rq}

We expect to be able to define symplectic forms and use them to study deformation quantizations for derived schemes in mixed characteristic.

The difference between the Frobenius-twisted cotangent complex and the cotangent complex we study here $\mathbb{L}_{(A,F)}$ seems to be connected to the difference between $p$-restricted and partition Lie algebra, see \cite{BM19}.

\subsection{Foliations}

We recall from \cite{To20}, the definition of derived foliations.

\begin{Def}
Let $X$ be a derived scheme over a commutative ring $k$. A derived foliation on $X$ is the data of a graded mixed derived stack $\mathcal{F}$ and a morphism of graded mixed derived stack $\mathcal{F} \to \mathcal{L}^{gr}(X)$ such that
\begin{itemize}
\item The projection $\mathcal{F} \to X$ is relatively affine.
\item The quasi-coherent complex $\mathcal{O}_\mathcal{F}(1)$ is in Tor-amplitude $]-\infty,m]$, for an integer $m$, where $\mathcal{O}_\mathcal{F}(1)$ is the $1$-weighted part of functions of $\mathcal{F}$ relative to $X$.
\item The natural morphism of graded derived stack
$$\mathcal{F} \to \mathbb{V}(\mathcal{O}_\mathcal{F}(1))$$
is an equivalence.
\end{itemize}
\end{Def}

In this definition, we were able to define Dieudonné foliations for possibly non-connective cotangent complexes $\mathbb{L}_\mathcal{F}$ using Corollary \ref{CoroLinStackFunction}. Indeed, we keep the expected property that
$$\mathcal{F} \mapsto \mathbb{L}_\mathcal{F}$$ is conservative.

We extend the definition of foliations to our framework.

\begin{Def}
Let $X$ be a derived scheme over $\mathbb{Z}_{(p)}$, endowed with a Frobenius lift. A Dieudonné foliation on $(X,F)$ is the data of a derived Dieudonné stack $F$ and a morphism of derived Dieudonné stack $\mathcal{F} \to \mathcal{L}^{gr}(X,F)$ such that
\begin{itemize}
\item The projection $\mathcal{F} \to X$ is relatively affine.
\item The quasi-coherent complex $\mathcal{O}_\mathcal{F}(1)$ is in Tor-amplitude $]-\infty,-1]$ where $\mathcal{O}_\mathcal{F}(1)$ is the $1$-weighted part of functions of $\mathcal{F}$ relative to $X$.
\item The natural morphism of graded derived stack endowed with a Frobenius lift
$$\mathcal{F} \to \mathbb{V}(\mathcal{O}_\mathcal{F}(1))$$
is an equivalence.
\end{itemize}
\end{Def}

Our classification Theorem \ref{classiDieud} gives a more precise description of Dieudonné foliations.

\begin{Thm}
Let $A$ be a smooth commutative $k$-algebra, $M$ a projective $A$-module of finite type. We fix a derived Frobenius lift structure on the graded  simplicial algebra $Sym_A(M[1])$, with $M$ in weight $1$. From Proposition \ref{FLSymStr}, it is equivalent to a classical Frobenius lift $F$ on $A$ and a linear map of $A$-modules $\phi : M \to M$. The space of Dieudonné foliations $\mathcal{F}$ over $A$ having $M$ as a cotangent complex is discrete and in bijection with the set of Dieudonné algebra structures on the graded commutative $k$-algebra $\bigoplus_i \wedge^i_A M [-i]$ endowed with its natural canonical Frobenius lift structure.

\end{Thm}

This theorem gives an easier description of Dieudonné foliations, which have a quite abstract definition.

\begin{Ex}
We outline the construction of a fundamental example of a Dieudonné foliation using the previous theorem. Let $(X,F)$ be a smooth derived scheme endowed with a Frobenius lift. We define the tangent complex $\mathbb{T}_{(X,F)}$ as the dual of the cotangent complex $\mathbb{L}_{(X,F)}$ in the category of $(X,F)$-modules, we note that $\mathbb{T}_{(X,F)}$ does not admit $\mathbb{T}_X$ as an underlying $\mathcal{O}_X$-module : an element of $\mathbb{T}_{(X,F)}$ can be thought of as a sequence of elements of $\mathbb{T}_X$ which satisfy compatibility conditions with the Frobenius morphism. We fix a sub-bundle $I$ of $\mathbb{T}_{(X,F)}$ which is stable by the Lie bracket. The derived stack
$$\mathcal{F} \coloneqq \mathbb{V}(I^\vee[1])$$
has a canonical mixed graded structure and a Frobenius structure. The canonical map
$$I \to \mathbb{T}_{(X,F)}$$
induces a morphism
$$\mathcal{F} \to \mathcal{L}^{gr}(X,F)$$
which makes $\mathcal{F}$ into a Dieudonné foliation.
\end{Ex}

\begin{Ex}
Let $f : X \to Y$ be a morphism of derived schemes with Frobenius lifts. Pulling back along $f$ should define a canonical foliation $f^* 0_Y \in Fol(X)$.
\end{Ex}

Extending on the theory of Dieudonné foliations, we hope to construct, for a fixed Dieudonné foliation, the de Rham-Witt complex along the leaves of the foliation, from which we will deduce crystalline cohomology along the foliation.

From these constructions, we can expect results on Dieudonné foliations regarding the vanishing of cohomology classes. In \cite{To20}, for a crystal $E$ on a foliation $\mathcal{F}$, cohomology classes $c_i(E(0))$ in $H^{2i}_{DR}(X/S)$ are shown to vanish in $H^{2i}_{DR}(X/\mathcal{F})$ where $\mathcal{F}$ is a foliation on $X$ relative to $S$. These classes are shown to come from classes in crystalline cohomology and a similar vanishing theorem is proven. These classes are, in fact, canonically lifted to rigid cohomology classes and the theory of Dieudonné foliations could help better understand the vanishing of the crystalline classes and give similar results for rigid cohomology classes.

\newpage

\begin{appendices}
\appendixheaderon

\numberwithin{Def}{section}

\section{Categorical results}

\begin{Prop} \label{LFPgroup}
Let $\mathcal{C}$ be an $\infty$-topos, $G$ a group in $\mathcal{C}^{endo}$. $G$ is naturally a group in $\mathcal{C}$ and it has a compatible endomorphism $\alpha : G \to G$ which induces a forgetful functor $\alpha^* : G-\mathcal{C} \to G-\mathcal{C}$. Then there is an equivalence :
$$G-\mathcal{C}^{endo} \simeq CFP_{\alpha^*}(G-\mathcal{C})$$

\end{Prop}

\begin{Prop} \label{actionProduct}
Let $\mathcal{C}_0 \to \mathcal{C}_{01} \leftarrow \mathcal{C}_1$ a pullback diagram of $\infty$-categories and $G$ be group in $\mathcal{C}_0 \times_{\mathcal{C}_{01}} \mathcal{C}_1$. $G$ induces a group $G_0$ in $\mathcal{C}_0$, $G_1$ in $\mathcal{C}_1$, $G_{01}$ in $\mathcal{C}_{01}$. The canonical morphism is an equivalence :
$$G-(\mathcal{C}_0 \times_{\mathcal{C}_{01}} \mathcal{C}_1) \xrightarrow{\sim} (G_0-\mathcal{C}_0) \times_{G_{01}-\mathcal{C}_{01}} (G_1-\mathcal{C}_1)$$
\end{Prop}

\begin{Prop} \label{LFAdj}
Let $\mathcal{C}$ a presentable category and $\eta$ an endofunctor of $\mathcal{C}$ which commutes with filtered colimits and small limits. Then the inclusion functor
$$U : FP(\mathcal{C}) \subset LFP(\mathcal{C})$$
admits a left adjoint.
\end{Prop}

\begin{proof}
The functor $U$ commutes with filtered colimits and with small limits as the functor $\eta$ does. The categories $FP_{\eta}(\mathcal{C})$ and $LFP(\mathcal{C})$ are presentable as limits of presentable categories, see Proposition \ref{defLFP}.
Using the adjoint functor theorem, Theorem \ref{adjointFunct}, concludes the proof.
\end{proof}

\section{Geometrical results}

\begin{Prop} \label{extAffine}
Let $F$ be an affine stack with $C(F)$ projective of finite type as a complex on $k$. We denote $p: F \to *$ the canonical structure morphism. Let $M$ and $N$ be $k$-complexes.

$$C(F) \otimes Map_{Mod}(M,N) \to Map_{QCoh(F)}(p^*M,p^*N)$$
is an equivalence, where $Map_{Mod}(M,N)$ is considered with its $k$-module structure.
\end{Prop}

\begin{proof}
As a complex, $C(F)$ is a retract of a free complex, therefore tensoring with $C(F)$ preserves all limits. Now, the morphism
$$C(F) \otimes Map_{Mod}(M,N) \to Map_{QCoh(F)}(M,N)$$
is compatible with colimits in $M$ and limits in $N$. We are reduced to proving the result for $M=k$ and $N=k$, where the result is obvious.
\end{proof}

\begin{Prop}
Let $A$ be a simplicial algebra over $\mathbb{F}_p$. Endowing a simplicial algebra over $A$ with its canonical Frobenius endomorphism defines a functor
$$B \in SCR_A \mapsto (B,Fr_B) \in SCR_A^{endo}$$
which admits a right adjoint given by taking the homotopy equalizer
$$(B,F) \in SCR_A^{endo} \mapsto B^{F \simeq Fr_B} \coloneqq eq(B \xrightrightarrows[F]{Fr_B} B) \in SCR_A$$ 
and a left adjoint given by taking homotopy coinvariants
$$(B,F) \in SCR_A^{endo} \mapsto B_{F \simeq Fr_B} \coloneqq coeq(B \xrightrightarrows[F]{Fr_B} B) \in SCR_A$$

\end{Prop}

\begin{proof}
We sketch the proof. Let $B \in SCR_A$ and $C \in SCR_A^{endo}$.

In the following diagram :

\begin{center}
\begin{tikzcd}
B \arrow[d,"Fr_B"] \arrow[r,"f"]
& C \arrow[d,"Fr_C"] \arrow[r,"Id"]
& C \arrow[d,"F"]\\
B  \arrow[r,"f"]
& C \arrow[r,"Id"]
& C
\end{tikzcd}
\end{center}

the left diagram is a commutative diagram.

For a fixed morphism $f : B \to C$, a homotopy making the outer diagram commute is equivalent to a homotopy between $Fr_C \circ f$ and $F \circ f$ in the following diagram.

\begin{center}
\begin{tikzcd}
B \arrow[r,"f"]
& C \arrow[d,"Fr_C"] \arrow[r,"Id"]
& C \arrow[d,"F"]\\
& C \arrow[r,"Id"]
& C
\end{tikzcd}
\end{center}

The latter is then the data of $f$ factoring through
$$C^{F\simeq Fr_C} \coloneqq eq(C \xrightrightarrows[F]{Fr_C} C)$$

The dual proof for the left adjoint goes similarly.

\end{proof}

A similar proof yield the following proposition.

\begin{Prop} \label{AdjointAjoutFrobCano}
Let $A$ be a simplicial algebra over $\mathbb{F}_p$. Endowing a derived stack over $A$ with its canonical Frobenius endomorphism defines a functor
$$X \in dSt_A \mapsto (X,Fr_X) \in dSt_A^{endo}$$
which admits a right adjoint given by taking the homotopy equalizer
$$(X,F) \in dSt_A^{endo} \mapsto X^{F \simeq Fr_X} \coloneqq eq(X \xrightrightarrows[F]{Fr_X} X) \in dSt_A$$ 
and a left adjoint given by taking homotopy coinvariants
$$(X,F) \in dSt_A^{endo} \mapsto X_{F \simeq Fr_X} \coloneqq coeq(X \xrightrightarrows[F]{Fr_X} X) \in dSt_A$$

\end{Prop}

\begin{Prop}  \label{FrobeniusPointsdSt}
Let $\mathcal{C}_0$ be the full subcategory of $dSt^{Fr}$ on objects which are given by freely adjoining a Frobenius lift to a derived affine scheme, these objects are of the form $L(Spec(C))$, with $C$ a simplicial algebra. See Proposition \ref{FreeFrobStr} for the construction of $L$.

The subcategory $\mathcal{C}_0$ generates $dSt^{Fr}$ under colimits.
\end{Prop}

\begin{proof}
The comonad $T$ induced from the adjunction
$$L : dSt \rightleftarrows dSt^{Fr} : U$$
induces a resolution of $X$ denoted $T^{\bullet+1}(X)$ using the bar construction, by combining the fact that $T$ commutes with small colimits with the Barr-Beck Theorem, see \cite[Theorem 4.7.3.5]{Lur17} and the proof of Proposition \ref{DieudTopos}. Therefore $X$ is given by the geometric realization of $T^{\bullet+1}(X)$.

Now we consider $Y = L(Z)$ obtained by freely adjoining a Frobenius lift to a derived stack. Writing $Z$ as a colimit of derived affine schemes, using the co-Yoneda lemma
$$Z \simeq colim_i Spec(A_i)$$
we deduce
$$Y = L(colim_i Spec(A_i)) \simeq colim_i L(Spec(A_i))$$
\end{proof}

\begin{Cor} \label{pointsFrobdSt}
The restriction of the Yoneda embedding gives a functor
$$dSt^{Fr} \to Fun(\mathcal{C}_0,\mathcal{S})$$
which is fully faithful.
\end{Cor}

Similar arguments provide analoguous results for derived stacks endowed with endomorphisms.  

\begin{Prop} 
Let $\mathcal{C}_0$ be the full subcategory of $dSt^{endo}$ on objets which are given by freely adjoining an endomorphism to a derived affine scheme, these objects are of the form $L(Spec(C))$, with $C$ a simplicial algebra.

The subcategory $\mathcal{C}_0$ generates $dSt^{endo}$ under colimits.
\end{Prop}

\begin{Cor} \label{pointsEndodSt}
The restriction of the Yoneda embedding gives a functor
$$dSt^{endo} \to Fun(\mathcal{C}_0,\mathcal{S})$$
which is fully faithful.
\end{Cor}

\begin{Prop} \label{conneAffSt}
Let $A \to k$ be an augmented cosimplicial $k$-algebra such that $H^0(A) \to k$ is an isomorphism, where $k$ is a principal torsion-free commutative ring. Let $n \ge 0$ be such that
$$H^i(A)=0$$
for $1 \le i \le n$. Note that this condition is empty for $n=0$. Then there is a cofibrant replacement of $A$
$$QA \xrightarrow{\sim} A$$
such that
\begin{itemize}
\item All coface maps $QA_i \to QA_j$ are flat.
\item All coface maps $QA_i \to QA_0$ are isomorphisms for $i \le n-1$.
\end{itemize}
\end{Prop}

\begin{proof}
We construct a tower of cosimplicial algebras by adding successive cells to kill the various cohomology groups since the model structure on cosimplicial algebras is cofibrantly generated. Let us define $X \coloneqq Spec^\Delta(A)$. We assume that we have constructed a cofibrant model $X^m=Spec^\Delta(A_m)$ with a canonical morphism
$$p_m : X \to X^m$$
which is an isomorphism on all cohomology groups $H^i$ for $i \le n$. We have a natural identification
$$H^{m+1}(A) \cong [X,K(\mathbb{G}_a,m+1)]$$

We define $I \coloneqq [X,K(\mathbb{G}_a,m+1)]$ as a set, and the tautological morphism
$$X \to K(\mathbb{G}_a,m+1)^I$$
We define $X'^m(1) \coloneqq X^m \times K(\mathbb{G}_a,m+1)^I=Spec^\Delta(A'^m(1))$, where $A'^m(1)$ is the tensor product of the free cosimplicial algebra on a coproduct of $I$ copies of $k[-m-1]$ and $A^m$. The natural morphism
$$A'^m(1) \to A$$
is surjective on the $H^{m+1}$ group. From the isomorphisms
$$H^{m+1}(A'^m(1)) \cong H^{m+1}(K(\mathbb{G}_a,m+1)^I) \times H^{m+1}(A^m) \cong End_{Gp}(\mathbb{G}_a)^{(I)} \times H^{m+1}(A^m)$$
We deduce
$$H^{m+1}(A'^m(1)) \cong k^{(I)} \times H^{m+1}(A^m)$$
because $k$ is torsion-free. The kernel of the surjective morphism
$$k^{(I)} \times H^{m+1}(A^m) \to H^{m+1}(A)$$
is denoted $J$. We deduce a map
$$X^m \to X'^m(1) \to K(\mathbb{G}_a,m+1)^J$$
which factors through
$$X^m(1) \coloneqq X'^m(1) \times_{K(\mathbb{G}_a,m+1)^J} E(\mathbb{G}_a,m+1)^J$$

We iterate this construction to obtain a tower
$$X^m \to ... \to  X^m(k) \to X'^m(k) \to ... X^m(1) \to X'^m(1)$$
which gives
$$\alpha_m : X^m \to X^{m+1}$$
where $X^{m+1}$ is the limit of the tower
$$X^m(k) \to X'^m(k) \to ... X^m(1) \to X'^m(1)$$
We check that $\alpha_m$ is an isomorphism on the $H^{i}$ groups where $i \le m+1$ and we define $QA$ by requiring
$$Spec^\Delta(QA) \coloneqq colim_m X^{m}$$
where we take the colimit in the category of affine stacks.

The flatness of the transition morphisms come from the flatness of
$$E(\mathbb{G}_a,i) \to K(\mathbb{G}_a,i)$$
when $i>0$.

Now to get the last assertion, we choose a specific model for $K(\mathbb{G}_a,i)$ and $E(\mathbb{G}_a,i)$. We take the following :
$$K(\mathbb{G}_a,i) \coloneqq Spec^\Delta(Free_{coSCR}(k[-m]))$$
and
$$K(\mathbb{G}_a,i) \coloneqq Spec^\Delta(Free_{coSCR}(k_{m-1} \xrightarrow{Id} k_m))$$

As the associated cosimplicial algebras of these affine stacks satisfy
$$Y_i \xrightarrow{\sim} Y_0$$
where $i<m-1$, this concludes the proof.
\end{proof}

\begin{Rq}
We notice in the proof that if $H^{n+1}(A)$ is a free $k$-module, $A'^{n+1}(1) \to A$ is an isomorphism on $H^{m+1}$ groups, therefore we can take $X^{n+1}(1)$ to be $X'^{n+1}(1)$ and we obtain a cofibrant model such that $QA_n \to QA_0$ is also an isomorphism.
\end{Rq}

From the remark, we deduce :

\begin{Cor}
The formal $n$-sphere
$$S^n_f \coloneqq Spec^\Delta (D H^*(S^n,k))$$
is $(n-1)$-connective.
\end{Cor}

\end{appendices}

\newpage
\setcounter{tocdepth}{2}

\bibliographystyle{amsalpha}

\end{document}